\documentclass[reqno,12 pt]{amsart}

\usepackage{amsfonts,amscd,mathrsfs}
\usepackage{amsthm}
\usepackage{amssymb}
\usepackage{latexsym}
\usepackage{cite}
\usepackage{url}
\usepackage{amsmath}
\usepackage{verbatim}
\usepackage{graphicx}
\usepackage{graphics}
\usepackage{color}
\usepackage[displaymath, mathlines]{lineno}
\usepackage{epsf}
\usepackage{enumerate}
\usepackage{geometry}
\usepackage{hyperref}
\definecolor{dark-blue}{rgb}{0,0,0.6}
\definecolor{Purple}{rgb}{0.2,0,0.25}
\hypersetup{breaklinks=true,
pagecolor=white,
colorlinks=true,
citecolor=dark-blue,
filecolor=black,%
linkcolor=Purple,%
urlcolor=black
}
\newtheorem{thm}{Theorem}[section]
\newtheorem{cor}[thm]{Corollary}
\newtheorem{lem}[thm]{Lemma}
\newtheorem{prop}[thm]{Proposition}

\newtheorem{defin}[thm]{Definition}

\theoremstyle{definition}

\newtheorem{remark}[thm]{Remark}

\newcommand{\dom}{\textnormal{dom}}
\newcommand{\sign}{\textnormal{sign}}

\newcommand{\Int}{\textnormal{Int}}

\newcommand{\diam}{\textnormal{diam}}
\newcommand{\cl}{\overline}

\newcommand{\R}{\mathbb{R}}

\newcommand{\N}{\mathbb{N}}

\newcommand{\bref}[1]{\textbf{\ref{#1}}} 
\newcommand{\beqref}[1]{\textbf{(\ref{#1})}} 

\subjclass[2010]{52A41, 52B55, 46N10, 90C25, 90C30, 46T99, 47N10, 49M37, 26B25, 58C05}
\keywords{Bregman divergence, Bregman function, gauge,  negative Boltzmann-Gibbs-Shannon entropy, negative Burg entropy, negative Havrda-Charv\'at-Tsallis entropy, negative iterated log entropy,
relative uniform convexity, strongly convex, uniformly convex}

\begin{document}

\date{8 April 2019}
\title[Re-examination of Bregman functions]{Re-examination of Bregman functions and new properties of their  divergences}

\author{Daniel Reem}
\address{Daniel Reem, Department of Mathematics, The Technion - Israel Institute of Technology, 3200003 Haifa, Israel.} 
\email{dream@technion.ac.il}
\author{Simeon Reich}
\address{Simeon Reich, Department of Mathematics, The Technion - Israel Institute of Technology, 3200003 Haifa, Israel.} 
\email{sreich@technion.ac.il}
\author{Alvaro De Pierro}
\address{Alvaro De Pierro, CNP$\textnormal{q}$, Brazil}
\email{depierro.alvaro@gmail.com}

\begin{abstract} 
The Bregman divergence (Bregman distance, Bregman measure of distance) is a certain useful substitute for a distance, obtained from a well-chosen function (the ``Bregman function''). Bregman functions and divergences have been extensively investigated during the last decades and have found applications in optimization, operations research, information theory, nonlinear analysis, machine learning and more. This paper re-examines various aspects related to the theory of Bregman functions and divergences. In particular, it  presents many  sufficient conditions which allow the construction of  Bregman functions in a general setting and introduces new Bregman functions (such as a negative iterated log entropy). Moreover, it sheds new light on several known Bregman functions such as quadratic entropies, the negative Havrda-Charv\'at-Tsallis entropy, and the negative Boltzmann-Gibbs-Shannon entropy, and it shows that the negative Burg entropy, which is not a Bregman function according to the classical theory but nevertheless is known to have ``Bregmanian properties'', can, by  our re-examination of the theory, be considered as a Bregman function. Our analysis yields several by-products of independent interest such as the introduction of the concept of relative uniform convexity (a certain generalization of uniform convexity), new properties  of uniformly and strongly convex functions, and results in Banach space theory.  
\end{abstract}

\maketitle
\tableofcontents

\section{Introduction}\label{sec:Introduction}
The Bregman divergence (the Bregman distance, the Bregman measure of distance) is a certain  substitute for a distance. Roughly speaking, given a space $X$, for example, $\R^n$ with the Euclidean norm ($n\in\N$)  or a more  general real normed space, and given a well-chosen convex function $b$ which is differentiable  on an open and convex subset of $X$, the Bregman divergence induced by $b$ is 
\begin{equation}\label{eq:BregmanB}
B(x,y):=b(x)-b(y)-\langle b'(y),x-y\rangle,
\end{equation}
where $y$ is in the open set, $x\in X$, and $\langle b'(y),x-y\rangle$ is the 
 value of the continuous linear functional $b'(y):X\to \mathbb{R}$ at the point $x-y$. This divergence, which  was introduced in the pioneering paper of Bregman \cite{Bregman1967jour} in 1967 and re-emerged in the 1981 paper of Censor and Lent  \cite{CensorLent1981jour}, has been extensively investigated since then and found applications in various areas  including  information theory, nonlinear analysis, optimization,  operations research, inverse problems, machine learning, and even computational geometry. For a rather short and partial list of related papers, see \cite{BanerjeeGuoWang2005jour,  BanerjeeMeruguDhillonGhosh2005jour, BauschkeBorweinCombettes2001jour, BauschkeBorweinCombettes2003jour,  
BauschkeMacklemSewellWang, BeckTeboulle2003jour, BoissonnatNielsenNock2010jour, BorweinReichSabach2011jour,  BregmanCensorReich1999jour, BurachikScheimberg2000jour, ButnariuIusem2000book, ButnariuResmerita2006jour, ButnariuReichZaslavski2001jour, ButnariuReichZaslavski2001incol, ButnariuReichZaslavski2006col, Cayton, CensorDePierroIusem1991jour, CensorIusemZenios1998jour, CensorReich1996jour,  CensorZenios1997Book, ChenTeboulle1993jour, CichockiAmari2010jour, CollinsSchapireSinger2002jour, Csiszar1991jour, DePierro1991incol, DePierroIusem1986jour, Eckstein1998jour, Elfving1989jour, GuptaHuang, JonesByrne1990jour, KaplanTichatschke2004jour, Kiwiel1997jour, MurataTakenouchiKanamoriEguchi2004jour, ReichSabach2010jour, ReichSabach2010b-jour,TaskarLacoste-JulienJordan2006jour, Teboulle1992jour, Teboulle2007jour, YinOsherGoldfarbDarbon2008jour, Zaslavski2010b-jour}.

In this paper we re-examine various aspects related to the theory of Bregman functions and divergences. One major aspect that we consider is the very definition of this concept in a rather general setting, and the presentation of many sufficient conditions which allow the construction of Bregman functions in such a setting. This general treatment, which is presented in Section \bref{sec:Bregman} and is especially relevant to real Hilbertian spaces (namely, spaces which are isomporphic to Hilbert spaces), is complemented in three ways; first, by the introduction of concrete Bregman functions in finite- and infinite-dimensional spaces and (such as a negative iterated log entropy which is discussed in Section \bref{sec:IteratedLog}, and the $\ell_2$-type entropy which is discussed in Section \bref{sec:Ell2type}); second, by shedding new light on several known Bregman functions such as the negative Boltzmann-Gibbs-Shannon entropy (Section \bref{sec:GibbsShannon}), the negative Havrda-Charv\'at-Tsallis entropy (Section \bref{sec:HavrdaCharvatTsallis}), and quadratic entropies (Subsection \bref{subsec:Quadratic}); third, by showing (in Remark \bref{rem:DefBregman} and Section \bref{sec:Burg}) that the negative Burg entropy, which is not a Bregman function according to the classical theory but nevertheless is known to have ``Bregmanian properties'', can, by our re-examination (and extension) of the theory, be considered a Bregman function. 

Our analysis yields a few by-products of independent interest. For instance, we re-examine in Section \bref{sec:UniformConvexity} the notion of uniform convexity and present a generalization of it which we call ``relative uniform convexity''. This notion turns out to be useful in proving the boundedness of the level sets of functions which are candidates to be Bregman functions as shown, for instance, in Proposition \bref{prop:BregmanProperties}\beqref{BregProp:BregPsiS1S2},\beqref{BregProp:B(x,S)Bounded==>SisBounded},\beqref{BregProp:LevelSetBounded}). Along the way we derive new properties of uniformly and strongly convex functions (see, for example, Section \bref{sec:UniformConvexity}, Proposition \bref{prop:BregmanProperties}, Corollaries \bref{cor:BregmanFunction-b'-is-unicont}--\bref{cor:BregmanFunctionFiniteDim-b'-is-cont} and  Subsection \bref{subsec:StrongConvexity}), and introduce certain new relative or non-relative uniform convexity  properties of concrete functions (for instance, in Section \bref{sec:GibbsShannon} we discuss  the negative Boltzmann-Gibbs-Shannon entropy, in Section \bref{sec:HavrdaCharvatTsallis} we discuss the negative Havrda-Charv\'at-Tsallis entropy, and in Section \bref{sec:Burg} we discuss the negative Burg entropy). We also present a certain contribution to the theory of Banach spaces, that is, we  introduce the class of Banach spaces having the component-* property (Subsection \bref{subsec:weak-to-weak*}). Finally, in a few sections of our paper we discuss briefly several not very well-known historical aspects related to the theory of Bregman functions and divergences. We also note that in the companion paper \cite{ReemReichDePierro(TEPROG)2019accept}, we apply the theory mentioned above to proximal forward-backward algorithms based on Bregman divergences in certain Banach spaces.  

We want to elaborate more on the relation of our paper to optimization theory. As can be seen in many of the references mentioned earlier, various optimization algorithms which involve Bregman divergences have been published. Convergence proofs of such algorithms have essentially two aspects:
one is the geometrical properties of the  divergence that generate Fej\'er type sequences, and the other one is related to the behavior of the  divergence on the boundary of the effective domain of the Bregman function which induces the divergence. This second aspect has generated the necessity to develop different convergence proofs depending on the involved functions. An important contribution of our paper is that we address this issue in various ways, for instance in Proposition \bref{prop:BregmanProperties} and in the sections devoted to entropy-like functions (Sections \bref{sec:GibbsShannon}--\bref{sec:SomewhatKnownBregman}). 

To the best of our knowledge, the results that we present in this paper are new. We note, however, that versions of a limited number of our results are known, mainly in a restricted setting (for example, in finite-dimensional Euclidean spaces instead of, say, in all normed spaces). In such cases we provide references to these versions. It is also worthwhile noting that large parts of the  discussion below are rather detailed and, in particular, full proofs are provided. We decided to do so not only in order to make the discussion as self-contained as possible, but also in order to clarify some issues which are usually  overlooked in the literature and in order to avoid the possibility of missing certain delicate points. Another reason for the detailed discussion is the fact that the treatment of several statements and examples  requires a lot of case analysis (for instance, due to certain parameters which appear in the formulation of these statements/examples). 

\section{Preliminaries}\label{sec:Preliminaries}
This section introduces the notation and main definitions used in this paper. 
We consider a real normed space $(X,\|\cdot\|)$, $X\neq \{0\}$. The dual of $X$, that is, the set of all continuous linear functionals from $X$ to $\R$, is denoted by $X^*$ and we let $\langle x^*,x\rangle:=x^*(x)$ for each $x^*\in X^*$ and $x\in X$. We say that $X$ is Hilbertian if there exists a continuous and invertible linear mapping between $X$ and a Hilbert space. In particular, if $X$ is a Banach space and the norm of $X$ is equivalent to another norm on $X$ which is induced by an inner product, then $X$ is Hilbertian. The closure of $V\subseteq X$ is denoted by $\cl{V}$ and the interior of $V$ is denoted by $\Int(V)$. The effective domain of a function $b:X\to(-\infty,\infty]$ is the set $\dom(b):=\{x\in X: b(x)<\infty\}$ and $b$ is said to be proper if its effective domain is nonempty. It is well known and immediate that $\dom(b)$ is convex whenever $b$ is 
convex. The closed interval between two points $x,y\in X$ is the set $[x,y]:=\{tx+(1-t)y: t\in [0,1]\}$, the open interval between them is the set $(x,y):=[x,y]\backslash\{x,y\}$ and the half open interval with endpoint $x$ is $[x,y):=[x,y]\backslash\{y\}$. Given $\emptyset\neq V\subseteq X$ and $f:V\to X^*$, we say that $f$ is weak-to-weak$^*$ sequentially continuous at $x\in V$ if for each sequence $(x_i)_{i=1}^{\infty}\in V$ which converges weakly to $x$ and for each $z\in X$ we have $\lim_{i\to\infty}\langle f(x_i),z\rangle=\langle f(x),z\rangle$. If $f$ is weak-to-weak$^*$ sequentially continuous at each $x\in V$, then $f$ is said to be weak-to-weak$^*$ sequentially continuous on $V$. A well-known fact  which will be used frequently in the sequel is that in finite-dimensional spaces the weak and strong (norm) topologies coincide, and on the dual spaces of finite-dimensional spaces the weak$^*$ and strong topologies coincide. 

We say that $M:X^2\to\R$ is a bilinear form if   both $x\mapsto M(x,y)$ and $y\mapsto M(x,y)$ are linear functions from $X$ to $\R$ for each $y\in X$ and $x\in X$, respectively. The norm of $M$ is $\|M\|:=\sup\{M(x,y)/(\|x\|\|y\|): x,y\in X\backslash\{0\}\}$ and we say that $M$ is bounded whenever $\|M\|<\infty$. A well-known fact is that if $M$ is bounded, then it is continuous, and conversely, if $X$ is a Banach space and $M$ is continuous, then it is bounded (see, for example, \cite[p. 49]{Brezis2011book} for a more general statement of the latter implication; in general, various well-known facts which are stated here without a reference can be found in, say, \cite{Brezis2011book, Kreyszig1978book}). 

\begin{defin}\label{def:Gateaux}
Given $b:X\to(-\infty,\infty]$, assume that $\Int(\dom(b))$ is nonempty and that $x\in \Int(\dom(b))$. We say that $b$ is G\^ateaux differentiable at $x$ if there exists a continuous linear functional $b'_G(x)\in X^*$ such that 
\begin{equation}\label{eq:Gateaux}
\langle b'_G(x),y\rangle=\lim_{t\to 0}\frac{b(x+ty)-b(x)}{t},\quad \forall\, y\in X.
\end{equation}
We say that $b$ is Fr\'echet differentiable at $x$  if there exists a continuous linear functional $b'_F(x)\in X^*$ such that for all $h\in X$ sufficiently small, 
\begin{equation}\label{eq:Frechet}
 b(x+h)=b(x)+\langle b'_F(x),h\rangle+o(\|h\|).
\end{equation}
We say that $b$ is twice Fr\'echet  differentiable at $x$ if there exists a continuous bilinear form $b''(x):X^2\to\R$ such that for all $h\in X$ sufficiently small, we have
\begin{equation}\label{eq:b''}
 b(x+h)=b(x)+\langle b'_F(x),h\rangle+\frac{1}{2}b''(x)(h,h)+o(\|h\|^2).
\end{equation}
We say that $b$ is G\^ateaux/Fr\'echet differentiable in an open subset $U\subseteq\dom(b)$ if $b$ is G\^ateaux/Fr\'echet differentiable at each $x\in U$. If there is no ambiguity regarding  the type of differentiation, then we denote by $b'(x)$ the corresponding derivative of $b$ at $x$ instead of $b'_G(x)$ or $b'_F(x)$ and say that $b$ is differentiable. 
\end{defin}
It is well known \cite[pp. 13-14]{AmbrosettiProdi1993book} that if $b$ is Fr\'echet differentiable at some $x$, then it is G\^ateaux differentiable at $x$ and $b'_G(x)=b'_F(x)$, and, on the other hand, if $b'_G$ exists in $U$ and is continuous at $x\in U$, then $b'_F(x)$ exists and is equal to $b'_G(x)$. Another well-known fact is that when $X$ is finite-dimensional and $b$ is lower semicontinuous, convex and proper, then  $b$ is G\^ateaux differentiable at $x\in\dom(b)$ if and only if it is Fr\'echet differentiable there \cite[Corollary 17.44, p. 306]{BauschkeCombettes2017book}, \cite[Theorem 25.2, p. 244]{Rockafellar1970book}. In all the concrete examples considered in this paper the derivatives will coincide since $b$ will be  continuously Fr\'echet differentiable on $U$, but in a few auxiliary results we will assume the weaker condition of G\^ateaux differentiability. In certain cases we may identify $b'$ with a vector in a space $Y$ which is naturally isomporphic to $X^*$ (for instance, when $X$ is a finite-dimensional space, a Hilbert space or an  $\ell_p$ spaces, $p\in (1,\infty)$). 

The subdifferential of $b$ at $x\in X$ is the set $\partial b(x):=\{x^*\in X^*: b(x)+\langle x^*,w-x\rangle\leq b(w)\,\,\forall w\in X\}$. The effective  domain of $\partial b$ is the set $\dom(\partial b):=\{x\in X:  \partial b(x)\neq\emptyset\}$. Of course, $\partial b(x)=\emptyset$ if $b$ is proper and $x\notin\dom(b)$, namely $\dom(\partial b)\subseteq \dom(b)$. If $X$ is finite-dimensional, and $b$ is proper and convex, and the subset $U:=\Int(\dom(b))$ is nonempty, and $b$ is G\^ateaux differentiable in $U$ and  $\lim_{i\to\infty}\|b'(y_i)\|=\infty$ for each sequence $(y_i)_{i=1}^{\infty}$ in $U$ which converges to some point on the boundary of $\dom(b)$, then we say that $b$ is essentially smooth. If $X$ is finite dimensional and $b$ is proper, convex, and strictly convex on every subset of $\dom(\partial b)$, then we say that $b$ is essentially strictly convex. If $b$ is essentially smooth and essentially strictly convex, then we say that $b$ is Legendre.

\section{Uniform convexity and relative uniform convexity}\label{sec:UniformConvexity}
In this section we recall the notions of uniform convexity and  strong convexity, and introduce the notions of relative uniform convexity and relative strong convexity. These notions, the definitions of which  are given in Definition \bref{def:TypesOfConvexity} below, are central to later sections. We clarify several  issues related to them in Remark \bref{rem:TypesOfConvexity} below, and then prove a certain lemma (Lemma \bref{lem:MonotonePsi}). 

\begin{defin}\label{def:TypesOfConvexity}
Let $(X,\|\cdot\|)$ be a normed space and let $b:X\to(-\infty,\infty]$ be convex. Denote $K:=\dom(b)$ and  suppose that $S_1$ and $S_2$ are two nonempty subsets (not necessarily convex) of $K$.  
\begin{enumerate}[(I)] 
\item\label{def:RelativeUniformlyConvex} The function $b$ is called uniformly convex relative to $(S_1,S_2)$ (or relatively uniformly convex on  $(S_1,S_2)$) if there exists $\psi:[0,\infty)\to [0,\infty]$, called a relative gauge, such that $\psi(t)\in (0,\infty]$ whenever $t>0$ and for each $\lambda\in (0,1)$ and each $(x,y)\in S_1\times S_2$, 
\begin{equation}\label{eq:RelativeUniformConvexity}
b(\lambda x+(1-\lambda)y)+\lambda(1-\lambda)\psi(\|x-y\|)\leq \lambda b(x)+(1-\lambda) b(y).
\end{equation}
If $S:=S_1=S_2$ and $b$ is uniformly convex relative to $(S_1,S_2)$, then $b$ is said to be  uniformly convex on $S$. 

\item\label{def:ModulusConvexity} The optimal gauge of $b$ relative to $(S_1,S_2)$  is the function defined for each $t\in [0,\infty)$ by:
\begin{multline}\label{eq:Relative_psi}
\psi_{b,S_1,S_2}(t):=\inf\left\{ \frac{\lambda b(x)+(1-\lambda)b(y)-b(\lambda x+(1-\lambda)y)}{\lambda(1-\lambda)}: (x,y)\in S_1\times S_2, 
\right.\\\left.
\|x-y\|=t,\lambda\in(0,1)
\vphantom{\frac{\lambda b(x)+(1-\lambda)b(y)-b(\lambda x+(1-\lambda)y)}{\lambda(1-\lambda)}}
\right\},
\end{multline} 
where we use the standard convention that $\inf \emptyset=\infty$, namely, if there is no $(x,y)\in S_1\times S_2$ such that $\|x-y\|=t$, then $\psi_{b,S_1,S_2}(t):=\infty$. 
 The optimal gauge is also called the modulus of relative uniform convexity of $b$ on $(S_1,S_2)$ or simply the optimal relative gauge, and $b$ is uniformly convex on $(S_1,S_2)$ if and only if $\psi_{b,S_1,S_2}(t)>0$ for every $t\in (0,\infty)$. If $S:=S_1=S_2$, then we denote $\psi_{b,S}:=\psi_{b,S_1,S_2}$ and call $\psi_{b,S}$ the modulus of uniform convexity of $b$ on $S$.
\item A function $\psi:S_1\times S_2\to \R$ which satisfies the inequality 
\begin{equation}\label{eq:pre-gauge}
\lambda(1-\lambda)\psi(x,y)+b(\lambda x+(1-\lambda)y)\leq \lambda b(x)+(1-\lambda)b(y),\quad\forall \,(x,y)\in S_1\times S_2,\lambda\in (0,1),
\end{equation} 
is called a relative pre-gauge of $b$ on $(S_1,S_2)$ (or relative to $(S_1,S_2)$). When $S_1=S_2$, then $\psi$ is called a pre-gauge. 
\item The function $b$ is said to be uniformly convex on closed, convex, and bounded subsets of $K$ if $b$ is  uniformly convex on each nonempty subset $S\subseteq K$ which is closed, convex and bounded.  
\item\label{def:StronglyConvexGlobal} The function $b$ is said to be strongly convex relative to $(S_1,S_2)$ if there exists $\mu>0$ (which depends on $S_1$ and $S_2$ and sometimes will be denoted by $\mu[S_1,S_2]$), called a parameter of strong convexity  of $b$ on $(S_1,S_2)$, such that $b$ is uniformly convex relative to $(S_1,S_2)$ with $\psi(t):=\frac{1}{2}\mu t^2$, $t\in [0,\infty)$ as a relative gauge. If $S:=S_1=S_2$ and $b$ is strongly convex relative to $(S_1,S_2)$, then $b$ is said to be strongly convex on $S$.
\end{enumerate}
\end{defin}

\begin{remark}\label{rem:TypesOfConvexity}
Here are a few comments regarding Definition \bref{def:TypesOfConvexity}. In all of these comments we assume that the setting of Definition \bref{def:TypesOfConvexity} holds.  
\begin{enumerate}[(i)]
\item\label{item:PsiIsNonnegative} The convexity assumption on $b$ is needed simply to ensure that the fraction which appears in the definition of $\psi_{b,S_1,S_2}$ (in Definition \bref{def:TypesOfConvexity}\beqref{def:ModulusConvexity}) is bounded below by zero, and hence  $\psi_{b,S_1,S_2}$ is nonnegative even if $b$ is not uniformly convex relative to $(S_1,S_2)$. Of course, if $b$ satisfies \beqref{eq:RelativeUniformConvexity} with some relative gauge $\psi$, then $b$ is automatically convex on $[x,y]$.
\item\label{item:diam(S_1,S_2)} A simple property which follows from \beqref{eq:RelativeUniformConvexity}  is that if $b$ is uniformly convex relative to $(S_1,S_2)$, then, given a relative gauge $\psi$, we have $\psi(t)\in (0,\infty)$ for all $t>0$ which satisfies $t<\textnormal{diam}(S_1,S_2):=\sup\{\|x-y\|: x\in S_1,\,y\in S_2\}$. Moreover, if $\textnormal{diam}(S_1,S_2)=\|x-y\|$ for some $x\in S_1$, $y\in S_2$, then $\psi(\textnormal{diam}(S_1,S_2))<\infty$. If $\textnormal{diam}(S_1,S_2)<\infty$, then for all $t>\textnormal{diam}(S_1,S_2)$ the value of $\psi(t)$ is not important, but, nonetheless, we have $\psi_{b,S_1,S_2}(t)=\infty$ according to \beqref{eq:Relative_psi}. In particular, if $\diam(S_1,S_2)=0$, a case which can only happen  when $S:=S_1=S_2$ is a singleton, then $\psi_{b,S_1,S_2}(t)=\infty$ for every $t>0$; in this degenerate case we still regard $b$ as being uniformly convex with respect to $S$. 
\item\label{item:psi(0)=0} If $S_1\cap S_2\neq \emptyset$, then $\psi_{b,S_1,S_2}(0)=0$, as follows immediately from  \beqref{eq:Relative_psi}. Another immediate observation (which follows from \beqref{eq:RelativeUniformConvexity} and \beqref{eq:Relative_psi})  is that if $b$ is uniformly convex relative to $(S_1,S_2)$, then $\psi_{b,S_1,S_2}(t)\in (0,\infty]$ for each $t>0$. 
\item\label{item:UniformlyConvex==>StrictlyConvex} If $b$ is uniformly convex relative to $(S_1,S_2)$, then $b$ is strictly convex on $[x,y]$ for all $x\in S_1$ and $y\in S_2$ satisfying $x\neq y$.  Indeed, by our assumption there exists a gauge $\psi:[0,\infty)\to[0,\infty]$ satisfying $b(\lambda x+(1-\lambda)y)\leq \lambda b(x)+(1-\lambda)b(y)-\lambda(1-\lambda)\psi(\|x-y\|)$ for all $\lambda\in (0,1)$. Since $x\neq y$,  the definition of $\psi$ implies that $\psi(\|x-y\|)>0$. Hence $b(\lambda x+(1-\lambda)y)<\lambda b(x)+(1-\lambda)b(y)$, namely $b$ is strictly convex on $[x,y]$. This observation can be extended: since intervals are bounded and convex subsets, we conclude that if $b$ is assumed to be uniformly convex relative to all pairs $(S_1,S_2)$ of nonempty, bounded and convex subsets of $K$, then $b$ is strictly convex on $K$.

\item Relative uniform convexity is a notion weaker than uniform convexity. Indeed, a function which is uniformly convex on some $\emptyset\neq S\subseteq K$ is uniformly convex relative to $(S_1,S_2)$ for all nonempty subsets $S_1$, $S_2$ of $S$, where the relative gauge coincides with the gauge. On the other hand, as will be shown in Subsections \bref{subsec:BGS-NoUniformlyConvex}, \bref{subsec:NoUniformConv q=0.5}, \bref{subsec:BurgNoGlobalUC} and \bref{subsec:IteratedLogNoGlobalUC} respectively, the negative Boltzmann-Gibbs-Shannon entropy, a certain instance of the negative  Havrda-Charv\'at-Tsallis entropy, the negative Burg entropy and the negative iterated log entropy are not uniformly convex on (the interior of) their effective domains. On the other hand, the first three functions are uniformly convex relative to some pairs $(S_1,S_2)$, where $S_2$ is almost the whole interior of the effective domain. 
\item A pre-gauge (from \beqref{eq:pre-gauge}) quantifies, in some sense, how much $b$ is strictly convex on $[x,y]$, namely, how much the basic convexity inequality becomes a strict inequality on $[x,y]$: the more the pre-gauge is positive at $(x,y)$, the more the function is strictly convex on $[x,y]$. If $b$ is convex on $[x,y]$ but not strictly convex there, then any pre-gauge of $b$ on $S_1\times S_2$ must be nonpositive at $(x,y)\in S_1\times S_2$. 
\item If $\emptyset \neq S\subseteq K$, then a gauge is an improved pre-gauge on $S^2$ since it provides a positive lower bound for the (basic convexity) difference $\lambda b(x)+(1-\lambda)b(y)-b(\lambda x+(1-\lambda)y)$ for given $x,y\in S$, a lower bound which does not depend on the points $x$ and $y$ but only on the distance between them. 
\item A nonnegative relative pre-gauge of $b$ also quantifies the strict convexity of $b$. A nonnegative relative gauge is an improved relative pre-gauge which quantifies the strict convexity of $b$ in a partial-uniform way: it holds uniformly in the sense that only the distance between $x$ and $y$ matters, as long as $x$ is taken from $S_1$ and $y$ is taken from $S_2$. 
\item It is possible to generalize even further the concept of relative pre-gauge so that a better lower bound is given for the convexity difference. Indeed, we can consider any function $\psi:S_1\times S_2\times [0,1]\to\R$ which satisfies the inequality $\psi(x,y,\lambda)+b(\lambda x+(1-\lambda)y)\leq \lambda b(x)+(1-\lambda)b(y)$ for every $(x,y,\lambda)\in S_1\times S_2\times [0,1]$. Such a ``relative primitive gauge'' $\psi$, when it is nonnegative, quantifies in a finer way the (strict) convexity of $b$. 

\item\label{item:mu[x,y]} A simple but useful observation which will be used in later sections is the following one: if for each $(x,y)\in S_1\times S_2$ the function  $b$ is strongly convex on the line segment $[x,y]$ with a strong convexity parameter $\mu[x,y]$, then $\phi(x,y):=0.5\mu[x,y]\|x-y\|^2$ is a relative pre-gauge for $b$ on $(S_1,S_2)$. In particular, if one is able to find a function $\psi:[0,\infty)\to [0,\infty)$ which is positive on $(0,\infty)$ and satisfies the inequality $\psi(\|x-y\|)\leq \phi(x,y)$ for each $(x,y)\in S_1\times S_2$, then $b$ is uniformly convex relative to $(S_1,S_2)$ with $\psi$ as a relative gauge. 
\item\label{item:UniformlyConvexTranslation} It is immediate to check that if $b$ is uniformly convex relative to $(S_1,S_2)$ with a gauge $\psi$ and if $z_0\in X$ is given, then the function $\tilde{b}:X\to(-\infty,\infty]$ which is defined by $\tilde{b}(\tilde{x}):=b(\tilde{x}+z_0)$ for each $\tilde{x}\in X$ (and having an effective domain $\dom(\tilde{b})=\dom(b)-z_0$) is uniformly convex relative to $(S_1-z_0,S_2-z_0)$ with a  gauge $\tilde{\psi}=\psi$.  In particular, if $b$ is strongly convex relative to $(S_1,S_2)$, then $\tilde{b}$ is strongly convex relative to $(S_1-z_0,S_2-z_0)$ (with the same parameter of strong convexity). Moreover, if, given $x\in \dom(b)$  and $r_x\geq 0$, we know that $b$ is uniformly (or strongly) convex relative to $(\{x\},\{y\in \dom(b): \|y\|\geq r_x\})$ with a gauge $\psi$, then a simple verification (using the triangle inequality) shows that if $\tilde{x}:=x-z_0$ and $r_{\tilde{x}}:=r_x+\|z_0\|$, then $\tilde{b}$ is uniformly (or strongly) convex relative to  $(\{\tilde{x}\},\{\tilde{y}\in \dom(\tilde{b}): \|\tilde{y}\|\geq r_{\tilde{x}}\})$ with the same gauge. 
\item The notions of relative uniform convexity, pre-gauge,  and relative pre-gauge seem to be new, and so is the possibility to  consider uniform and strong convexity of a function on a subset $S$ of $K$ which is not necessarily convex. In addition, it seems that the possibility to define a modulus of uniform convexity for functions which are merely convex and not necessarily uniformly convex is new too. This possibility has some implications: for instance, $\psi_{b,S}(t)$ may grow to infinity as at least as fast as a quadratic function in $t$ even if there exists some $t_1>0$ such that $\psi_{b,S}(t)=0$ for all $t\in [0,t_1]$  (see Lemma \bref{lem:MonotonePsi} below). 
\item An important reason for considering relative uniform convexity is the fact that it enables one to show some boundedness properties related to the level sets of $b$ and beyond: see Proposition \bref{prop:BregmanProperties}\beqref{BregProp:BregPsiS1S2},\beqref{BregProp:B(x,S)Bounded==>SisBounded},\beqref{BregProp:LevelSetBounded} and Corollary \bref{cor:BregmanFunction-b'-is-unicont} below. Such properties guarantee that $b$ satisfies Definition  \bref{def:BregmanDiv}\beqref{BregmanDef:BoundedLevelSet} above. Once one knows that the corresponding level sets are bounded and if, in addition, the space $X$ is reflexive, then one can conclude that any sequence contained in the corresponding level set contains a weakly convergent subsequence, a very useful property for proving the convergence of many algorithmic schemes. 
\item Additional information related to uniformly  convex functions can be found in, for instance,   \cite{VladimirovNesterovChekanov1978jour} (the original work in this domain; properties and examples), 
\cite[pp. 63-66]{Nesterov2004book} (strong convexity in finite-dimensional Euclidean spaces),\cite{Zalinescu1983jour},\cite[pp. 203--221]{Zalinescu2002book} (properties and examples),  \cite{ButnariuIusemZalinescu2003jour} (relation to total convexity), and \cite{BGHV2009jour,BorweinVanderwerff2012jour} (properties and examples).  
\end{enumerate}
\end{remark}

We finish this section with a lemma which describes some properties of the modulus of uniform convexity. Part of it is known in a slightly different setting, namely inequality \beqref{eq:psi_ct} below when $b$ is uniformly convex and $S=K$: see \cite[Lemma 1]{VladimirovNesterovChekanov1978jour} and \cite[Proposition 3.5.1, pp. 203-204]{Zalinescu2002book}. Our proof is inspired by \cite[proof of Proposition 3.5.1]{Zalinescu2002book}. 
\begin{lem}\label{lem:MonotonePsi} 
Let $(X,\|\cdot\|)$ be a normed space and $b:X\to(-\infty,\infty]$. Let $K:=\dom(b)$ and suppose that $b$ is convex on a nonempty convex subset $S$ of $K$. Let $\psi_{b,S}$ be the modulus of uniform convexity of $b$ on $S$ (see Definition \bref{def:TypesOfConvexity}\beqref{def:ModulusConvexity}). Then $\phi(t):=\psi_{b,S}(t)/t^2$ is increasing on $(0,\infty)$ and $\psi_{b,S}$ is increasing on $[0,\infty)$ (and both of them are strictly increasing on $[0,\diam(K))$ if $b$ is uniformly convex). In addition, 
\begin{equation}\label{eq:psi_ct}
\psi_{b,S}(ct)\geq c^2\psi_{b,S}(t),\quad\forall c\geq 1, \forall t\geq 0
\end{equation}
and 
\begin{equation}\label{eq:psi_ct,c<1}
\psi_{b,S}(ct)\leq c^2\psi_{b,S}(t),\quad\forall c\in (0,1), \forall t\geq 0.
\end{equation}
Hence if $S$ is not a singleton, then $\psi_{b,S}(c)$ decays to zero at least as fast as a quadratic function in $c$ when $c\to 0$. If $S$ is unbounded and if $\psi_{b,S}(t_0)>0$ for some $t_0>0$ (as happens, in particular, if $b$ is uniformly convex on $S$), then $\psi_{b,S}(c)$ grows to infinity at least as fast as a quadratic function in $c$ when $c$ grows to infinity. 
\end{lem}
\begin{proof}
We start by proving \beqref{eq:psi_ct}. When $c=1$, then \beqref{eq:psi_ct} is trivial. When $t=0$, then \beqref{eq:psi_ct} holds because of Remark \bref{rem:TypesOfConvexity}\beqref{item:psi(0)=0}. Now assume that $t>0$ and $c\in (1,2)$. This special case will be used later for proving the more general case in which $c$ can be  arbitrary in $(1,\infty)$. If $\psi_{b,S}(ct)=\infty$, then \beqref{eq:psi_ct} is satisfied. Consider now the case $\psi_{b,S}(ct)<\infty$ and fix an arbitrary $\epsilon\in (0,\infty)$. It follows  from \beqref{eq:Relative_psi} that there exist $x,y\in S$ and $\lambda\in (0,1)$ such that $\|x-y\|=ct$ and 
\begin{equation}\label{eq:psi_epsilon}
\frac{\lambda b(x)+(1-\lambda)b(y)-b(\lambda x+(1-\lambda)y)}{\lambda(1-\lambda)}<\epsilon+\psi_{b,S}(ct).
\end{equation}
The left-hand side of \beqref{eq:psi_epsilon} is well defined because $b$ is assumed to be finite on $S$.  
It can be assumed that $\lambda\in (0,0.5]$, because if $\lambda\in (0.5,1)$, then we let $\lambda':=1-\lambda$, $x':=y$, $y':=x$, and an immediate verification shows that  \beqref{eq:psi_epsilon} holds with $\lambda'$, $x'$, $y'$ instead of $\lambda$, $x$, $y$  respectively. 

Let $y_{\lambda}:=\lambda x+(1-\lambda)y$, $y_c:=(1/c)x+(1-(1/c))y$. Then $\|y_c-y\|=t$ and $y_{\lambda}=c\lambda y_c+(1-c\lambda)y$. Since $c\in(1,2)$ and $\lambda\in (0,0.5]$, we have $c\lambda\in (0,1)$ and $1/c\in (0,1)$. Therefore the convexity of $S$ implies that $y_c$ and $y_{\lambda}$ belong to $S$. We conclude from \beqref{eq:Relative_psi} that $\psi_{b,S}(t)<\infty$ and also that  
\begin{equation}\label{eq:b(y_c)}
b(y_c)\leq \frac{1}{c}\cdot b(x)+\left(1-\frac{1}{c}\right)\cdot b(y)-\frac{1}{c}\left(1-\frac{1}{c}\right)\psi_{b,S}(ct),
\end{equation}
\begin{equation}\label{eq:b(y_lambda)}
b(y_{\lambda})\leq c\lambda b(y_c)+(1-c\lambda)b(y)-c\lambda(1-c\lambda)\psi_{b,S}(t).
\end{equation}
Now we start with \beqref{eq:b(y_lambda)}, bound from above its right-hand side  (first, by bounding from above $c\lambda b(y_{\lambda})$ by the right-hand side of \beqref{eq:b(y_c)} after multiplying \beqref{eq:b(y_c)} by $c\lambda$, second, by taking into account \beqref{eq:psi_epsilon} and performing elementary algebraic manipulations, and third, by recalling that $y_{\lambda}=\lambda x+(1-\lambda)y$). By additional simple calculations (which, along the way, lead to the cancellation of $b(y_{\lambda})$ on both sides of the inequality)  we arrive at 
\begin{equation}\label{eq:c^2psi<epsilon+}
c^2\psi_{b,S}(t)\leq \psi_{b,S}(ct)+\frac{c(1-\lambda)\epsilon}{1-c\lambda}
<\psi_{b,S}(ct)+\frac{2c\epsilon}{2-c},
\end{equation}
where the second inequality in \beqref{eq:c^2psi<epsilon+} follows from the assumptions that $\psi_{b,S}(ct)<\infty$ and that $\lambda\in(0,0.5]$. Since \beqref{eq:c^2psi<epsilon+} holds for arbitrary small $\epsilon>0$, it follows that \beqref{eq:psi_ct} holds for every $t>0$ and $c\in (1,2)$. Since $tc\in (0,\infty)$ we conclude from  \beqref{eq:psi_ct}  that $\psi_{b,S}(c^2t)=\psi_{b,S}(c(ct))\geq c^2\psi_{b,S}(ct)\geq c^4\psi_{b,S}(t)$. By induction we have $\psi_{b,S}(c^kt)\geq c^{2k}\psi_{b,S}(t)$ for each $k\in\N$, $t>0$ and $c\in (1,2)$. Now fix $c\in (1,\infty)$. Since $c>1$ and  $\lim_{k\to\infty}\sqrt[k]{c}=1$, there exists $k\in\N$ large enough such that $\sigma:=\sqrt[k]{c}\in (1,2)$. From the previous lines we know that $\psi_{b,S}(\sigma^kt)\geq \sigma^{2k}\psi_{b,S}(t)$. This and $\sigma^k=c$ imply that \beqref{eq:psi_ct} holds.  

To show that $\phi(t):=\psi_{b,S}(t)/t^2$ is increasing on $(0,\infty)$ we fix $0<t_1<t_2<\infty$. Then $t_2=ct_1$ for $c:=t_2/t_1>1$. From \beqref{eq:psi_ct} we conclude that $\phi(t_2)\geq \phi(t_1)$ (and strict inequality holds if $b$ is uniformly convex and $t_1,t_2\in (0,\diam(K))$), as required. This inequality and the equality $\psi_{b,S}(t)=t^2\phi(t)$, $t\in (0,\infty)$ show that $\psi_{b,S}$ is monotone increasing  on $[0,\infty)$ (and it is strictly increasing on $[0,\diam(K))$ if $b$ is uniformly convex) because on  $(0,\infty)$ it is a product of two increasing and nonnegative functions and at $t=0$ it vanishes, as explained in  Remark \bref{rem:TypesOfConvexity}\beqref{item:psi(0)=0}. It remains to prove \beqref{eq:psi_ct,c<1}. Fix $c\in (0,1)$ and $t\geq 0$. Denote $s:=1/c\in (1,\infty)$. From \beqref{eq:psi_ct} we have $\psi_{b,S}(t)=\psi_{b,S}(sct)\geq s^2\psi_{b,S}(ct)$. Thus $c^2\psi_{b,S}(t)=(1/s^2)\psi_{b,S}(t)\geq \psi_{b,S}(ct)$, as claimed. 

As for the decaying property of $\psi_{b,S}$ when $S$ is not a singleton ($S$ may be bounded or unbounded), fix some positive $t_0$ which is smaller than $\diam(S)$. This is possible since $\diam(S)\in (0,\infty]$. We have $\psi_{b,S}(t_0)\in [0,\infty)$  according to Remark \bref{rem:TypesOfConvexity}\beqref{item:diam(S_1,S_2)}. Given an arbitrary $s\in (0,t_0)$, we can write $s=ct_0$ for some $c\in (0,1)$. From \beqref{eq:psi_ct,c<1} we have $\psi_{b,S}(s)=\psi_{b,S}(ct_0)\leq c^2\psi_{b,S}(t_0)=(\psi_{b,S}(t_0)/t_0^2)s^2$, that is, $\psi_{b,S}(s)$ decays to 0 at least as fast as a quadratic function in $s$, as required. Finally, we need to consider the growth property of $\psi_{b,S}$ when $S$ is unbounded and $\psi_{b,S}(t_0)>0$ for some $t_0>0$. Since $\diam(S)=\infty$, it follows that $\psi_{b,S}(t_0)$ is finite (Remark \bref{rem:TypesOfConvexity}\beqref{item:diam(S_1,S_2)}), and from \beqref{eq:psi_ct} we have $\psi_{b,S}(c)=\psi_{b,S}(t_0\cdot(c/t_0))\geq (\psi_{b,S}(t_0)/t_0^2)c^2\to\infty$ when $t_0<c\to\infty$, as claimed. 
\end{proof}

\section{Bregman functions and divergences}\label{sec:Bregman}
In this section we discuss the central concepts of a Bregman function and a Bregman divergence and establish various properties related to them. Full proofs of all of the assertions are given for the sake of completeness even in cases when versions of some assertions are known in certain settings. For the reader's  convenience, the section is partitioned into subsections. 

\subsection{Definitions and related remarks}
We start with the following definition. 
\begin{defin}\label{def:preBregman}
Let $(X,\|\cdot\|)$ be a real normed space. Let  $b:X\to(-\infty,\infty]$ and suppose that $\Int(\dom(b))\neq\emptyset$ and also that $b$ is G\^ateaux differentiable in $\Int(\dom(b))$.  Define a function $B:X^2\to(-\infty,\infty]$ by 
\begin{equation}\label{eq:BregmanDistance}
B(x,y):=\left\{\begin{array}{lll}
b(x)-b(y)-\langle b'(y),x-y\rangle, & \forall (x,y)\in \dom(b)\times \Int(\dom(b)),\\
\infty & \textnormal{otherwise}.
\end{array}
\right.
\end{equation} 
Then $B$  is called a pre-Bregman  divergence (pre-Bregman distance, pre-Bregman measure) induced by the pre-Bregman function $b$. The set $\Int(\dom(b))$ is called the zone of $b$.
\end{defin}

A pre-Bregman function is a function which is a candidate to be a Bregman function and a  pre-Bregman divergence is a function which is a candidate to be a Bregman divergence. The exact  definitions of a Bregman function and a Bregman divergence is given below. 
\begin{defin}\label{def:BregmanDiv}
Let $(X,\|\cdot\|)$ be a real normed space and 
let $b:X\to(-\infty,\infty]$. Let $\emptyset\neq U\subseteq X$. Then $b$ is called a {\bf Bregman function}  with respect to $U$ ({\bf the zone of $b$}) and $B:X^2\to(-\infty,\infty]$ is called the {\bf Bregman divergence} (or the Bregman distance, or the Bregman measure of distance) associated with $b$ if all of  the following conditions  hold:
\begin{enumerate}[(i)]
\item\label{BregmanDef:U}  $U=\Int(\dom(b))$ (in particular, $\Int(\dom(b))\neq \emptyset$) and $b$ is G\^ateaux differentiable in $U$. 
\item\label{BregmanDef:ConvexLSC} $b$ is convex and lower semicontinuous on $X$ and strictly convex on $\dom(b)$.
\item\label{BregmanDef:B} $B$ is defined by \beqref{eq:BregmanDistance}.
\item\label{BregmanDef:BoundedLevelSet} For each $\gamma\in \R$ and  each $x\in \dom(b)$, the 
level-set $L_1(x,\gamma):=\{y\in U: B(x,y)\leq \gamma\}$ is  bounded. 
\item\label{BregmanDef:|x-y_i|=0==>B(x,y_i)=0} 
Let $x\in\dom(b)$ and let $(y_i)_{i=1}^{\infty}$ be a given  sequence in $U$. If  $\lim_{i\to\infty}\|x-y_i\|=0$, then  $\lim_{i\to\infty}B(x,y_i)=0$. 
\item\label{BregmanDef:B(x_i,y_i)=0==>|x_i-y|=0} 
Let $(x_i)_{i=1}^{\infty}$ be a given sequence in $\dom(b)$ and $(y_i)_{i=1}^{\infty}$ be a given  sequence in $U$. If $(x_i)_{i=1}^{\infty}$ is bounded, $\lim_{i\to\infty}B(x_i,y_i)=0$,  and there exists $y\in \dom(b)$ such that $\lim_{i\to\infty}\|y_i-y\|=0$, then $\lim_{i\to\infty}\|x_i-y\|=0$. 

\end{enumerate} 
We say that $b$ has the {\bf limiting difference property} if for each $x\in \dom(b)$ and each sequence $(y_i)_{i=1}^{\infty}$ in $U$, if $(y_i)_{i=1}^{\infty}$ converges weakly to some $y\in U$, then $B(x,y)=\lim_{i\to\infty}(B(x,y_i)-B(y,y_i))$. We say that $b$ is {\bf sequentially consistent} if  for all sequences $(x_i)_{i=1}^{\infty}$ in $\dom(b)$ and $(y_i)_{i=1}^{\infty}$ in $\Int(\dom(b))$, if $(y_i)_{i=1}^{\infty}$ is bounded and one has $\lim_{i\to\infty}B(x_i,y_i)=0$, then $\lim_{i\to\infty}\|x_i-y_i\|=0$. 
\end{defin}
Here are a few remarks concerning  Definitions \bref{def:preBregman} and \bref{def:BregmanDiv}.

\begin{remark}\label{rem:BregmanNotMetric}
Of course, the Bregman divergence $B$ (which is always finite on $\dom(b)\times U$) is in general not a metric:  for example, it is not symmetric and does not necessary satisfy the triangle inequality. However, it can be seen from Definition \bref{def:BregmanDiv}  that $B$ still enjoys various properties which make it a substitute for a measure of distance. Additional relevant properties can be found in Proposition  \bref{prop:BregmanProperties} below. These properties are useful in various scenarios, as will be explained  in Remark \bref{rem:BregmanDivDifferentProp} below. 

In the above-mentioned connection, it is worth mentioning that $B$, or at least particular instances of it induced by special pre-Bregman functions $b$, satisfies certain useful analytical relations (such as equations or inequalities) which have geometrical interpretations. Among these relations are the ``three-point identity'' of Chen and Teboulle \cite[Lemma 3.1]{ChenTeboulle1993jour},  and the ``three-point   property'' \cite[p. 212, Inequality  (2.11)]{CsiszarTusnady1984jour}, the ``four-point property'' \cite[p. 212, Inequality  (2.12)]{CsiszarTusnady1984jour} and the ``five-point property'' \cite[p. 207, Inequality  (1.4)]{CsiszarTusnady1984jour} of Csisz\'ar and  Tusn\'ady. For instance, if $X$ is a Hilbert space and for some $\gamma>0$  we have $b(x):=\gamma\|x\|^2$, $x\in X$ (in this case $B(x,y)=\gamma\|x-y\|^2$ for all $x,y\in X$), then the ``three-point identity'' is nothing but a restatement of the cosine rule from trigonometry, and a special case of the ``three-point property'' (which was also observed in \cite[Lemma 1]{Bregman1967jour}) is nothing but a well-known inequality which follows from the cosine rule and the obtuseness of the angle between two vectors: the vector obtained from a point and its orthogonal projection onto a nonempty, closed and convex subset, and the vector obtained from an arbitrary point in that subset and the above-mentioned projection. 
\end{remark}

\begin{remark}
We occasionally refer to (arbitrary) concrete examples of Bregman functions as ``entropies'', since some of the well-known examples of Bregman functions, or their negatives, are traditionally called ``entropies'' (a typical example is the negative Boltzmann-Gibbs-Shannon entropy). Hence, a general term for all of these   functions might be ``Bregman entropies'', but we do not use this phrase.   
\end{remark}

\begin{remark}\label{rem:DefBregman}
 Definitions \bref{def:preBregman} and \bref{def:BregmanDiv} are modeled after, and slightly extend \cite[Definition 3.1, Definition 4.1]{BauschkeBorwein1997jour}, which by themselves extend the classical finite-dimensional definition  of the Bregman function \cite[pp. 153--154]{CensorDePierroIusem1991jour}, \cite[Definition 2.1]{CensorLent1981jour}, \cite[Definition 2.1]{CensorReich1996jour},  \cite[Definition 2.1]{DePierroIusem1986jour}. The notion of a Bregman function which has the limiting difference property appears in a somewhat implicit form in previous works, for instance in \cite[Definition 2.4(5), Example 2.5]{Reem2012incol} and \cite{Reich1996incol}. The notion of Bregman functions which are sequentially consistent appears in various  forms, for instance in \cite[Condition 4.3(ii)]{BauschkeCombettes2003jour}, \cite[p. 249]{ButnariuByrneCensor2003jour}, \cite[Lemma 2.1.2, p. 67]{ButnariuIusem2000book}, and \cite[p. 50]{ButnariuIusemZalinescu2003jour}.  

 In the classical definition $\dom(b)$ should be closed and $b$ should be continuous on it, and also $b'$ should be continuous on $U$. In Definition \bref{def:BregmanDiv} we require less from $b$. We let $b$ to attain the value $\infty$ because, among other things, this allows us to treat in a unified way classical Bregman functions such as the normalized energy function $b=\frac{1}{2}\|x\|^2$, $x\in\R^n$, and the negative Boltzmann-Gibbs-Shannon entropy $b(x)=\sum_{k=1}^n x_k\log(x_k)$, $x=(x_k)_{k=1}^n\in[0,\infty)^n$, as well as another classical entropy (the negative Burg entropy: $b(x)=-\sum_{k=1}^n \log(x_k)$, $x=(x_k)_{k=1}^n\in (0,\infty)^n$), which is not a Bregman function according to the classical definition \cite[Definition 2.1]{CensorLent1981jour} because $\dom(b)$ is not closed and $b$ cannot be extended to a finite and continuous  function (convex or not) defined on $\overline{\dom(b)}=[0,\infty)^n$. 

When constructing Bregman functions, we will usually start with a nonempty, open and convex subset $U\subseteq X$, define $b$ there, and then we will extend it to $\cl{U}$, either in a continuous way (as in the case of the negative Boltzmann-Gibbs-Shannon entropy and the negative Havrda-Charv\'at-Tsallis entropy, which are discussed in Sections \bref{sec:GibbsShannon} and \bref{sec:HavrdaCharvatTsallis}, respectively) or by assigning it the value $\infty$ on $\cl{U}\backslash U$ (as in the case of the negative Burg entropy  which is discussed in Section \bref{sec:Burg}). The values of $b$ outside $\cl{U}$ will not be very important to us, but in order to make sure that $b$ is be lower semicontinuous on $X$ (as will be needed in most of the cases discussed in Proposition \bref{prop:BregmanProperties}), we will define $b(x):=\infty$ for $x\notin \cl{U}$. 
\end{remark}

\begin{remark}\label{rem:LevelSet}
In Definition \bref{def:BregmanDiv}\beqref{BregmanDef:BoundedLevelSet} it is sufficient to require that $\gamma\in (0,\infty)$ since, as shown in Proposition \bref{prop:BregmanProperties}\beqref{BregProp:B=0} below, $B(x,y)\geq 0$ for all $x, y\in X$. 
If $b$ is a classical Bregman function and $X$ is finite-dimensional (that is, $b$ satisfies Definition \bref{def:BregmanDiv}, where here one requires that also $\dom(b)$ is closed, that $b'$ is continuous on $U$, and that $b$ is continuous on $\dom(b)$), then it is also required that the second type level-set $L_2(y,\gamma):=\{x\in \dom(b): B(x,y)\leq \gamma\}$  be bounded for each $\gamma\in \R$ (or, in fact, for each $\gamma\in (0,\infty)$) and $y\in U$. However, it is well known that this requirement is redundant: see, for instance, \cite[Remarks 4.2]{BauschkeBorwein1997jour} or \cite[Theorem 3.2]{ButnariuByrneCensor2003jour}.  In our paper we are less interested in the boundedness of  the second type level-sets, but nevertheless we present several sufficient conditions for this requirement to be fulfilled: see Proposition \bref{prop:BregmanProperties}\beqref
{BregProp:LevelSetBounded},\beqref{BregProp:LevelSet2BoundedSufficientConditions} below. As a matter of fact, in all of the concrete examples of Bregman functions that we present in later sections the second type level-sets are bounded. A different sufficient condition  can be found in \cite[Theorem 3.7(iii)]{BauschkeBorwein1997jour}: it says that the second type level-sets are bounded whenever $X$ is finite-dimensional and $b$ is a lower semicontinuous convex pre-Bregman function which is essentially strictly convex. A generalization of this condition to reflexive Banach spaces can be found in \cite[Lemma 7.3(v)]{BauschkeBorweinCombettes2001jour}. 
\end{remark}

\begin{remark}\label{rem:ContinuityPreBregman}
If $X$ is a Banach space, then any lower semicontinuous and convex pre-Bregman function $b$ is automatically continuous on its zone $U$. Indeed, since $U$, namely, the interior of its effective domain, is nonempty, and since $X$ is Banach, $b$ is continuous on $U$ according to \cite[Proposition 3.3, p. 39]{Phelps1993book_prep}.
\end{remark}

\begin{remark}
In the infinite-dimensional theory of Bregman functions and divergences, such as in \cite{AlberButnariu1997jour,BauschkeBorweinCombettes2001jour,BauschkeBorweinCombettes2003jour,ButnariuIusem2000book,Reich1996incol,ReichSabach2010b-jour} the treatment of $B$ and $b$ is usually not axiomatic as in the finite-dimensional case or as in Definition \bref{def:preBregman}. Indeed, while $B$ is defined as in \beqref{eq:BregmanDistance} (or slight modifications of \beqref{eq:BregmanDistance}, as in \cite[p. 3]{ButnariuIusem2000book}, where the gradient of $b$ is replaced by a one-sided directional derivative), $b$ is assumed to be not only what we called a pre-Bregman function  (Definition \bref{def:preBregman}), but rather a special function: for instance, a uniformly convex and Fr\'echet differentiable function, or a Legendre function that satisfies additional concrete assumptions. However, as far as we know, with the exception of \cite[p. 65]{ButnariuIusem2000book} and  \cite{Reem2012incol}, no general axioms such as the ones given in Definition \bref{def:preBregman} have been  imposed. In the first case $b$ is assumed to be totally convex on $U$ and to satisfy Definition \bref{def:BregmanDiv}\beqref{BregmanDef:BoundedLevelSet} for all $x\in U$, and the theory regarding the divergence $B(x,y)$ is developed mainly to points $x,y\in U$. In the second case the goal is to develop a theory of Bregman distances without Bregman functions, and so the axiomatic approach considered there concerns $B$ and not $b$. We also note that it is known that the classical definition of Bregman functions in finite-dimensional spaces involves  redundancies in the sense that some items in it imply other ones (see \cite{ButnariuByrneCensor2003jour} for a survey and a thorough analysis). 
\end{remark}

\begin{remark}\label{rem:BregmanDivDifferentProp}
A major reason behind  Definition \bref{def:BregmanDiv} is that the properties of $B$ and $b$ allow one to establish the convergence of various algorithmic schemes which aim at solving a rich class of optimization problems (constrained and unconstrained minimization, the feasibility problem, finding zeros and fixed points of nonlinear operators, etc.). Additional or slightly different properties are sometimes needed for establishing certain convergence results, and usually they are achieved by imposing additional assumptions on $b$. Definition \bref{def:BregmanDiv} suffices for most finite-dimensional  Bregman-divergence-type algorithms that we are aware of, but in order to allow more flexibility and to better address infinite-dimensional settings, we establish in Proposition \bref{prop:BregmanProperties} below and the corollaries following it a few additional useful properties of $B$. Later (Section \bref{sec:StrongConvexityWeak-to-weak*}) we present relevant  sufficient conditions which allow one to construct Bregman functions. The examples in later sections are based on these properties and conditions. 
\end{remark}

\begin{remark}\label{rem:BregmanLinear}
Given a continuous linear functional $\ell:X\to\R$, a positive number $\lambda$, and a Bregman function $f:X\to(-\infty,\infty]$ which induces a Bregman divergence $B_f$, it is easy to verify that $b:=\lambda f+\ell$ is a Bregman function with zone $\Int(\dom(f))$ and an effective domain $\dom(f)$, and that $B_{b}=\lambda B_f$.  In particular, $B_{\ell}=0$; moreover, if $f$ has the limiting difference property, then $b$ has it too, and if $f$ is sequentially consistent, then so is $b$.  In addition, given $m$ Bregman functions $b_k:X\to(-\infty,\infty]$ with associated Bregman divergences $B_k$, $k\in \{1,\ldots,m\}$, $m\in\N$, and given $m$ positive numbers $\lambda_k$, $k\in\{1,\ldots,m\}$, denote $b:=\sum_{k=1}^m \lambda_k b_k$ and assume that $\Int(\cap_{k=1}^m\dom(b_k))\neq\emptyset$. Then it is simple to check that $b$ is a Bregman function with an effective domain $\dom(b)=\cap_{k=1}^m\dom(b_k)$, zone $\Int(\dom(b))=\Int(\cap_{k=1}^m\dom(b_k))=\cap_{k=1}^m\Int(\dom(b_k))$, and an associated Bregman divergence $B=\sum_{k=1}^m \lambda_k B_k$; moreover, if $b_k$ has the limiting difference property for each $k\in\{1,\ldots,m\}$, then $b$ has this property too; in addition, if $b_k$ is sequentially consistent for some $k\in\{1,\ldots,m\}$, then $b$ is sequentially consistent (here we also use Proposition \bref{prop:BregmanProperties}\beqref{BregProp:B=0} below), and if  for some $k\in\{1,\ldots,m\}$ all the second type level-sets of $B_k$ are bounded, then all  the second type level-sets of $B$ are bounded. These simple observations can be useful in some scenarios, as is illustrated in Subsections \bref{subsec:BetaEntropy}--\bref{subsec:alpha-beta} below. 
\end{remark}

\begin{remark}\label{rem:BregmanTranslation}
A simple verification shows that if $b:X\to(-\infty,\infty]$ is a Bregman function with zone $U$ and 	 associated Bregman divergence $B$, and if $z_0\in X$ is given, then the function $\tilde{b}:X\to(-\infty,\infty]$, which is defined by $\tilde{b}(\tilde{x}):=b(\tilde{x}+z_0)$ for each $\tilde{x}\in X$, is a Bregman function with a zone $\tilde{U}:=U-z_0$. The associated Bregman divergence $\tilde{B}$ of $\tilde{b}$ satisfies $\tilde{B}(\tilde{x},\tilde{y})=B(\tilde{x}+z_0,\tilde{y}+z_0)$ for each $(\tilde{x},\tilde{y})\in X^2$. Moreover, $b$ has the limiting difference property if and only if $\tilde{b}$ has this property, and $b$ is sequentially consistent if and only if $\tilde{b}$ is sequentially consistent.  
\end{remark}

\subsection{Pre-Bregman functions: Sufficient conditions}
In this subsection  we establish many sufficient conditions which ensure that the considered pre-Bregman  functions satisfy parts or all of Definition \bref{def:BregmanDiv}. We start with a lemma, which is probably known, and which is used in the proof of Proposition \bref{prop:BregmanProperties}\beqref{BregProp:B(x,x_i)-B(x,y_i)=0} below (we note that the convexity  assumption on the set $S$ which appears in Lemma \bref{lem:UniformContinuityConvexSet} is crucial: a simple counterexample is to take  $X=\ell_2$, $Y=\R$, $d=|\cdot|$, $S=\{e_k: k\in\N\}$, where $e_k$ is the $k$-th element in the canonical basis of $X$, and $f:S\to Y$, $f(e_k):=k$ for all $k\in\N$.). 
\begin{lem}\label{lem:UniformContinuityConvexSet}
Suppose that  $S$ is a nonempty, bounded and convex subset of a normed space $(X,\|\cdot\|)$ and let $(Y,d)$ be a metric space. If $f:S\to Y$ is uniformly continuous on $S$, then $f$ is bounded on $S$. 
\end{lem}
\begin{proof}
The uniform continuity of $f$ on $S$ implies that for an arbitrary fixed $\epsilon>0$ (say, $\epsilon:=1$)  there exists $\delta>0$ such that for every $u,v\in S$ satisfying  $\|u-v\|<\delta$ we have $d(f(u),f(v))<\epsilon$.  Now fix $z\in S$. The boundedness of $S$ implies that there exists $M>0$ such that $\|u-z\|\leq M$ for each $u\in S$. Given $u\in S$, $u\neq z$ (we can assume that $S\neq \{z\}$, otherwise the assertion is obvious and the proof is complete), let $m$ be the maximal integer which is smaller than $\|u-z\|/(0.5\delta)$. Then $0\leq m<\|u-z\|/(0.5\delta)\leq m+1$ and thus $|\|u-z\|-0.5m\delta|\leq 0.5\delta$. For each nonnegative integer $k\in \{0,\ldots,m\}$ let $p_k$ be the point defined by $p_k:=u+0.5k\delta(z-u)/\|z-u\|$ and let $p_{m+1}:=z$. Then $\|p_{k}-p_{k+1}\|=0.5\delta$ whenever $k\in\{0,\ldots,m-1\}$ and 
\begin{multline*}
\|p_m-p_{m+1}\|=\left\|\left(u+0.5m\delta\frac{(z-u)}{\|z-u\|}\right)-\left(u+\|z-u\|\frac{(z-u)}{\|z-u\|}\right)\right\|\\
=|0.5m\delta-\|z-u\||\leq 0.5\delta.
\end{multline*} 
Therefore $\|p_k-p_{k+1}\|<\delta$ for each $k\in \{0,1,\ldots,m\}$. Since the convexity of $S$ implies that $p_k\in [u,z]\subseteq S$ for each $k\in \{0,1,\ldots,m\}$, it follows from the triangle inequality, and the choice of $\delta$ and $M$, that 
\begin{equation*} 
d(f(u),f(z))\leq \sum_{k=0}^{m} d(f(p_k),f(p_{k+1}))<(m+1)\epsilon<\left(\frac{\|u-z\|}{0.5\delta}+1\right)\epsilon\leq\left(\frac{M}{0.5\delta}+1\right)\epsilon. 
\end{equation*}
This inequality obviously holds also when $u=z$, and we conclude from the triangle inequality that $d(f(u),f(v))\leq d(f(u),f(z))+d(f(z),f(v))<((1/\delta)4M+2)\epsilon$ for every $u,v\in S$. Hence $f$ is bounded on $S$, as claimed. 
\end{proof} 

The following proposition, which is the main result of this section, describes many properties of Bregman and pre-Bregman functions and divergences, especially their asymptotic behavior under various conditions. It also describes some asymptotic properties (such as strong convergence) of relevant sequences.  Some of these properties are not needed for a pre-Bregman function to be a Bregman function according to Definition \bref{def:BregmanDiv}, but they are still useful in some circumstances, for example, for establishing the convergence of certain algorithmic sequences to the solutions of various optimization problems. For instance, the property of $B$ mentioned in Proposition  \bref{prop:BregmanProperties}\beqref{BregProp:B(x,x_i)-B(x,y_i)=0} is needed in \cite{Reich1996incol} for establishing the convergence of an infinite product of certain operators to a common asymptotic fixed point of them, and the sequential consistency of $B$ mentioned in  Proposition  \bref{prop:BregmanProperties}\beqref{BregProp:B(x_i,y_i)=0AndOneSequenceBounded==>|x_i-y_i|=0}  is useful in the convergence analysis of various algorithms, such as the one discussed in \cite[Section 4]{BauschkeCombettes2003jour} (in \cite[Condition 4.3(ii)]{BauschkeCombettes2003jour} a stronger version is assumed in which both $(x_i)_{i=1}^{\infty}$ and $(y_i)_{i=1}^{\infty}$ belong to $U$).

It is worth noting that versions of some of the properties mentioned in Proposition \bref{prop:BregmanProperties} below are known in some settings. For example, the claim that $B(x,y)>0$ if $x\in\dom(b)$, $y\in \Int(\dom(b))$, $x\neq y$  (Proposition  \bref{prop:BregmanProperties}\beqref{BregProp:B=0}), is well known if the space is finite-dimensional and $b$ is a classical Bregman function (see Remark \bref{rem:DefBregman}  above for a discussion on classical Bregman functions). However, in the infinite-dimensional case it seems that our result, in which we assume less, is new, and it actually generalizes \cite[Lemma 7.3(vi)]{BauschkeBorweinCombettes2001jour} if $b$ is also assumed to be G\^ateaux differentiable at $y$ (there $b$ is assumed to be essentially strictly convex but instead of using $B$ from \beqref{eq:BregmanDistance}, a slightly more general version of $B$, based on directional derivatives, is used in \cite[Lemma 7.3(vi)]{BauschkeBorweinCombettes2001jour}; for a closely related result, in which both $x$ and $y$ are assumed to belong to the algebraic interior of $\dom(b)$ and $b$ is assumed to be strictly convex on the algebraic interior of $\dom(b)$, see \cite[Proposition 1.1.4, p. 4]{ButnariuIusem2000book}). A second  example: Proposition \bref{prop:BregmanProperties}\beqref{BregProp:D(x,y_i)=0U)} generalizes \cite[Proposition 3.2(ii)]{BauschkeBorwein1997jour}. A third example: a variant of Proposition \bref{prop:BregmanProperties}\beqref{BregProp:BregPsiS1S2} is known when $b$ is assumed to be uniformly convex \cite[Theorem 3.10(i)-(iv), p. 215]{Zalinescu2002book} or strongly convex \cite[Corollary 3.11(i)-(iv), pp. 217--218; Remark 3.5.3, p. 218]{Zalinescu2002book}. 

We also note that some parts of Proposition \bref{prop:BregmanProperties} may seem, at first glance, very similar or even identical (for example, Parts \beqref{BregProp:Bounded|x_i-y_i|=0==>B(x_i,y_i)=0}--\beqref{BregProp:Continuous_b|x_i-y_i|=0==>B(x_i,y_i)=0}). However, this is not the case, because each time there are differences in the assumptions and also in the stated results, and these differences require modifications (sometimes significant ones) in the corresponding proofs.

\begin{prop}\label{prop:BregmanProperties}
Let $(X,\|\cdot\|)$ be a real normed space. Suppose that $b:X\to(-\infty,\infty]$ is convex,  lower semicontinuous, and that $U:=\Int(\dom(b))$ is nonempty. Assume also that $b$ is G\^ateaux differentiable in $U$.  Denote by $B:X^2\to (-\infty,\infty]$ the pre-Bregman divergence defined in \beqref{eq:BregmanDistance}. Let $S_1$ and $S_2$ be nonempty subsets of $\dom(b)$. Then the following properties hold:
\begin{enumerate}[(I)]
\item\label{BregProp:BregPsiS1S2}  Given $x\in\dom(b)$, $y\in U$ and a relative gauge $\psi$ of $b$ on $(\{x\},\{y\})$ (the modulus of uniform convexity of $b$ relative to $(\{x\},\{y\})$ or any other relative gauge), we have $\psi(\|x-y\|)\leq B(x,y)$. In particular, if $S_2\cap U\neq\emptyset$ and if $b$ is uniformly convex relative to $(S_1, S_2\cap U)$ with some relative gauge $\psi$, then $\psi(\|x-y\|)\leq B(x,y)$ for each $(x,y)\in S_1\times (S_2\cap U)$. Specifying even further, if $S_2\cap U\neq\emptyset$ and if $b$ is uniformly convex relative to $(S_1, S_2\cap U)$ with some relative gauge $\psi$ which is assumed to be strictly increasing and continuous on some interval $J\subseteq [0,\infty)$ having the property that $\psi(J)$ contains $B(x,y)$ for each $(x,y)\in S_1\times (S_2\cap U)$, then $\|x-y\|\leq \psi^{-1}(B(x,y))$ for each $(x,y)\in S_1\times (S_2\cap U)$. 

\item\label{BregProp:Convex} For each $y\in U$, the function $x\mapsto B(x,y)$, $x\in X$, is convex and lower semicontinuous.  In addition, if $b$ is uniformly convex relative to $(S_1,S_2)$ with a gauge $\psi$, then for each $y\in U$ the function $x\mapsto B(x,y)$, $x\in X$  is uniformly convex relative to $(S_1,S_2)$ with the same gauge $\psi$. 

\item\label{BregProp:B=0} If $x=y\in U$, then $B(x,y)=0$. Furthermore, $B(x,y)\geq 0$ for each $x\in X$ and $y\in X$. If, in addition, $b$ is strictly convex on the open interval $(x,y)$ for some $x\in \dom(b)$ and  $y\in U$ satisfying $x\neq y$, then $B(x,y)>0$. In particular, if $b$ is strictly convex on $U$, and $x\in \dom(b)$ and $y\in U$ are given, then $B(x,y)=0$ if and only if $x=y$.

\item\label{BregProp:UniformlyConvexAndB(x_i,y_i)=0==>|x_i-y_i|=0} 
Assume that $b$ is uniformly convex on a nonempty and convex subset $S$ of $\dom(b)$ satisfying $S\cap U\neq\emptyset$. Let $(x_i)_{i=1}^{\infty}$ be a given sequence in $S$ and $(y_i)_{i=1}^{\infty}$ a given  sequence in $S\cap U$. If  $\lim_{i\to\infty}B(x_i,y_i)=0$, then $\lim_{i\to\infty}\|x_i-y_i\|=0$.

\item\label{BregProp:B(x_i,y_i)=0AndOneSequenceBounded==>|x_i-y_i|=0} 
Suppose that $b$ is uniformly convex on every nonempty, bounded and convex subset of $\dom(b)$. Then given two sequences $(x_i)_{i=1}^{\infty}$ in $\dom(b)$ and $(y_i)_{i=1}^{\infty}$ in $U$ satisfying $\lim_{i\to\infty}B(x_i,y_i)=0$, if one of these sequences is bounded, then the other sequence is bounded too and, moreover, $\lim_{i\to\infty}\|x_i-y_i\|=0$.  In particular, $b$ is sequentially consistent (see Definition \bref{def:BregmanDiv}). Specializing even further, under the assumption that $b$ is uniformly convex on  every nonempty, bounded and convex subset of $\dom(b)$, if  $(x_i)_{i=1}^{\infty}$ is in $\dom(b)$ and $(y_i)_{i=1}^{\infty}$ is in $U$, and one of these sequences converges strongly to some $z\in \overline{\dom(b)}$, and also $\lim_{i\to\infty}B(x_i,y_i)=0$, then the other sequence converges strongly to $z$ too.

\item\label{BregProp:x_iCompact} Suppose that $b$ is continuous on $\dom(b)$ and strictly convex on $U$ and that $b'$ is continuous on $U$. Let $(x_i)_{i=1}^{\infty}$ be a given sequence in $X$ which is contained in a compact subset of $\dom(b)$ (in particular, this condition holds when $X$ is finite-dimensional and $(x_i)_{i=1}^{\infty}$ is bounded). Let  $(y_i)_{i=1}^{\infty}$ be a given  sequence in $U$ such that $\lim_{i\to\infty}\|y_i-y\|=0$ for some $y\in U$. If $\lim_{i\to\infty}B(x_i,y_i)=0$, then $\lim_{i\to\infty}\|x_i-y\|=0$. 

\item\label{BregProp:CompactB(x_i,y_i)=0==>|x_i-y_i|=0}
Suppose that $(x_i)_{i=1}^{\infty}$ is a bounded sequence in $\dom(b)$ and $(y_i)_{i=1}^{\infty}$ is a sequence in $U$ which is contained in a compact subset of $\overline{\dom(b)}$ (in particular, this condition is satisfied if $X$ is finite-dimensional and $(y_i)_{i=1}^{\infty}$ is bounded). If $B$ is a Bregman divergence and $\lim_{i\to\infty}B(x_i,y_i)=0$, then $\lim_{i\to\infty}\|x_i-y_i\|=0$.

\item\label{BregProp:Bounded|x_i-y_i|=0==>B(x_i,y_i)=0} Assume that $b'$ is  bounded on all the subsets of $U$ which are bounded and convex. Let $(x_i)_{i=1}^{\infty}$ and $(y_i)_{i=1}^{\infty}$ be sequences of elements of $U$. If one of these sequences is bounded and $\lim_{i\to\infty}\|x_i-y_i\|=0$, then  $\lim_{i\to\infty}B(x_i,y_i)=0$. 

\item\label{BregProp:Continuous_b|x_i-y_i|=0==>B(x_i,y_i)=0} Suppose that $b$ is continuous on $\dom(b)$. Assume that $b'$ is  bounded on all bounded and convex subsets of $U$. Let $(x_i)_{i=1}^{\infty}$ be a given sequence in $\dom(b)$ and $(y_i)_{i=1}^{\infty}$ a given  sequence in $U$, where one of these sequences is bounded. If $\lim_{i\to\infty}\|x_i-y_i\|=0$, then $\lim_{i\to\infty}B(x_i,y_i)=0$. In particular, if $b$ is continuous on $\dom(b)$ and $b'$ is uniformly continuous on all bounded and convex subsets of $U$, then given a sequence $(y_i)_{i=1}^{\infty}$ in $U$ which converges to some $x\in\dom(b)$, we have $\lim_{i\to\infty}B(x,y_i)=0$.

\item\label{BregProp:D(x,y_i)=0U)} Suppose that $b$ is continuous on $U$ and that $b'$ is locally bounded at each point of $U$. If $(y_i)_{i=1}^{\infty}$ is a sequence in $U$ which satisfies $\lim_{i\to\infty}\|x-y_i\|=0$ for some $x\in U$, then $\lim_{i\to\infty}B(x,y_i)=0$. In particular, if $X$ is finite-dimensional and $b:X\to(-\infty,\infty]$ is merely assumed to be a lower semicontinuous and convex pre-Bregman function with zone $U\neq \emptyset$ (as we  assume in the formulation of Proposition \bref{prop:BregmanProperties}), then for every sequence $(y_i)_{i=1}^{\infty}$ in $U$ which converges to some $x\in U$, we have $\lim_{i\to\infty}B(x,y_i)=0$.

\item\label{BregProp:D(x,y_i)=0dom(b)} Suppose that $b$ is continuous on $\dom(b)$  and that $b'$ has the following property: for each $x\in\dom(b)$ and each sequence  $(y_i)_{i=1}^{\infty}$ in $U$ which converges in norm to $x$, the relation $\lim_{i\to\infty}\langle b'(y_i),x-y_i\rangle=0$ holds. Given a sequence $(y_i)_{i=1}^{\infty}$  in $U$, if $\lim_{i\to\infty}\|x-y_i\|=0$ for some $x\in \dom(b)$, then $\lim_{i\to\infty}B(x,y_i)=0$.

\item\label{BregProp:B(p,A)Finite} Let $S\subseteq U$ be nonempty. 
Suppose that $x\in U$ and $\sup\{\|x-y\|: y\in S\}<\infty$. If $b'$ is  bounded on all bounded and convex subsets of $U$, then $\sup\{B(x,y): y\in S\}<\infty$.  

\item\label{BregProp:B(x,S)Finite2} Suppose that $S\subseteq U$ is nonempty. Let $x\in \dom(b)$.  Assume that $\sup\{\|x-y\|: y\in S\}<\infty$ and that $b$ is continuous on $\dom(b)$. Assume also that $b'$ is  bounded on all bounded and convex subsets of $U$. Then  $\sup\{B(x,y): y\in S\}<\infty$. 

\item\label{BregProp:B(x,S)Bounded==>SisBounded} Let $\emptyset\neq S\subseteq U$ and $x\in \dom(b)$ be given. Assume that $S=S'\cup S''$ where $S'$  is bounded (and possibly empty) and $S''$ is contained in a nonempty subset $V$ of $U$ such that $b$ is uniformly convex relative to $(\{x\},V)$ with a gauge $\psi$ which  satisfies $\lim_{t\to\infty}\psi(t)=\infty$. Suppose that  $\sup\{B(x,y): y\in S\}<\infty$.  Then $S$ is bounded. Similarly, given $\emptyset\neq S\subseteq \dom(b)$ and $y\in U$, if $S=S'\cup S''$ where $S'$  is bounded (and possibly empty) and $S''$ is contained in a nonempty subset $W$ of $\dom(b)$ such that $b$ is uniformly convex relative to $(W,\{y\})$ with a gauge $\psi$ which  satisfies $\lim_{t\to\infty}\psi(t)=\infty$, and, in addition,  $\sup\{B(x,y): x\in S\}<\infty$, then $S$ is bounded.

\item\label{BregProp:LevelSetBounded} Suppose that for each $x\in\dom(b)$, there exists $r_x\geq 0$  such that the subset $U\cap \{w\in X: \|w\|\geq r_x\}$ is nonempty and $b$ is uniformly convex relative to $(\{x\},U\cap \{w\in X: \|w\|\geq r_x\})$ with a gauge $\psi_x$ which  satisfies $\lim_{t\to\infty}\psi_x(t)=\infty$.  Then for each $x\in \dom(b)$ and each $\gamma\in [0,\infty)$, the first type level-set $L_1(x,\gamma):=\{y\in U: B(x,y)\leq \gamma\}$ is bounded.  Moreover, if, given $x\in\dom(b)$, the gauge $\psi_x$ satisfies $\psi_x(0)=0$, and it is continuous and strictly increasing on $[0,\infty)$, then $\max\{2\psi_x^{-1}(\gamma),2r_x,\psi_x^{-1}(\gamma)+r_x+\|x\|\}$ is an upper bound on the diameter of $L_1(x,\gamma)$. 

Similarly, if for each $y\in U$, there exists $r_y\geq 0$ having the property that the subset $\{w\in \dom(b): \|w\|\geq r_y\}$ is nonempty and $b$ is uniformly convex relative to $(\{w\in \dom(b): \|w\|\geq r_y\},\{y\})$ with a gauge $\psi_y$ which  satisfies $\lim_{t\to\infty}\psi_y(t)=\infty$, then for each $y\in \dom(b)$ and each $\gamma\in [0,\infty)$ the second type level-set $L_2(y,\gamma):=\{x\in \dom(b): B(x,y)\leq \gamma\}$ is bounded. Moreover, if, given $y\in U$, the gauge $\psi_x$ satisfies $\psi_x(0)=0$, and it is continuous and strictly increasing on $[0,\infty)$, then $\max\{2\psi_y^{-1}(\gamma),2r_y,\psi_y^{-1}(\gamma)+r_y+\|y\|\}$ is an upper bound on the diameter of $L_2(y,\gamma)$.

Finally, if $b$ is uniformly convex on $\dom(b)$, then all the first and the second type level-sets of $B$ are bounded. 
\item\label{BregProp:LevelSetBounded2SuffCond} Each of the following conditions is sufficient for the boundedness of all the level-sets of the first type $L_1(x,\gamma)$, $x\in\dom(b)$, $\gamma\in [0,\infty)$:
\begin{enumerate}[(i)]
\item $U$ is bounded;
\item $U$ is unbounded and $\lim_{\|y\|\to\infty,\, y\in U}B(x,y)=\infty$ for all $x\in\dom(b)$.
\end{enumerate} 

\item\label{BregProp:LevelSet2BoundedSufficientConditions} Each of the following conditions is sufficient for the boundedness of all the level-sets of the second type $L_2(y,\gamma)$, $y\in U$, $\gamma\in [0,\infty)$:
\begin{enumerate}[(i)]
\item $U$ is bounded;
\item $U$ is unbounded and $\lim_{\|x\|\to\infty, \,x\in \dom(b)}B(x,y)=\infty$ for all $y\in U$;
\item $X$ is finite-dimensional and $b$ is strictly convex on $U$ and is essentially smooth;
\item\label{item:L2StrictlyConvex} $X$ is finite-dimensional and $b$ is strictly convex on $\dom(b)$.
\end{enumerate}

\item\label{BregProp:B(x,x_i)-B(x,y_i)=0} Suppose that $b'$ is uniformly continuous on all bounded and convex subsets of $U$. Given two sequences $(x_i)_{i=1}^{\infty}$  and  $(y_i)_{i=1}^{\infty}$ in $U$ such that one of them is bounded, suppose that  $\lim_{i\to\infty}\|x_i-y_i\|=0$. Then $\lim_{i\to\infty}(B(x,x_i)-B(x,y_i))=0$ for each $x\in \dom(b)$.

\item\label{BregProp:B(x,y)=B(x,y_i)-B(y,y_i)} Suppose that $b'$ is weak-to-weak$^*$ sequentially  continuous on $U$. If $(y_i)_{i=1}^{\infty}$ is a sequence in $U$ which converges weakly to  $y\in U$, then $B(x,y)=\lim_{i\to\infty}(B(x,y_i)-B(y,y_i))$ for all $x\in \dom(b)$. In particular, under the assumptions that $X$ is finite-dimensional and $b'$ is continuous on $U$, if $(y_i)_{i=1}^{\infty}$ is a sequence in $U$ which converges to some $y\in U$ and if $x\in \dom(b)$, then $B(x,y)=\lim_{i\to\infty}(B(x,y_i)-B(y,y_i))$.

\item Given $x\in\dom(b)$, if $b$ is uniformly convex relative to $(\{x\},S_2\cap U)$ with a relative gauge $\psi_x$ which is continuous and strictly increasing on $[0,\infty)$, and satisfies $\lim_{t\to\infty}\psi_x(t)=\infty$ and $\psi_x(0)=0$, and if $\lim_{i\to\infty}B(x,y_i)=0$ for some sequence $(y_i)_{i=1}^{\infty}$ in $S_2\cap U$, then $(y_i)_{i=1}^{\infty}$ converges strongly to $x$. Similarly, given $y\in U$, if $b$ is uniformly convex relative to $(S_1,\{y\})$  with a relative gauge $\psi_y$ which is continuous and strictly increasing on $[0,\infty)$, and satisfies $\lim_{t\to\infty}\psi_y(t)=\infty$ and  $\psi_y(0)=0$, and if $\lim_{i\to\infty}B(x_i,y)=0$ for some sequence $(x_i)_{i=1}^{\infty}$ in $S_1$, then $(x_i)_{i=1}^{\infty}$ converges strongly to $y$. 
\end{enumerate} 
\end{prop}

\begin{proof}
\begin{enumerate}[(I)] 
\item We first recall that the modulus of uniform convexity $\psi_{b,\{x\},\{y\}}$ (see \beqref{eq:Relative_psi}) exists even if $b$ is not uniformly convex relative to $(\{x\},\{y\})$. Now, since $b$ is G\^ateaux differentiable at $y$, it follows from \beqref{eq:RelativeUniformConvexity} that for every $\lambda\in (0,1)$, we have   
\begin{multline}\label{eq:(1-lambda)psiB(x,y)} 
\psi(\|x-y\|)\leq \frac{\lambda b(x)+(1-\lambda)b(y)-b(\lambda x+(1-\lambda)y)}{\lambda(1-\lambda)}\\
=\frac{1}{1-\lambda}\left(b(x)-b(y)-\frac{b(y+\lambda(x-y))-b(y)}{\lambda}\right)\\
\xrightarrow[\lambda\to 0]{}b(x)-b(y)-\langle b'(y),x-y\rangle=B(x,y),
\end{multline}
as required. If, in addition, $S_2\cap U\neq\emptyset$ and $b$ is uniformly convex relative to $(S_1,S_2\cap U)$ with some relative gauge $\psi$, then, in particular, $\psi$ is a relative gauge of $b$ on $(\{x\},\{y\})$  for every $x\in S_1$ and $y\in S_2\cap U$. Thus we can conclude from previous lines that $\psi(\|x-y\|)\leq B(x,y)$  for all $x\in S_1$ and $y\in S_2\cap U$. Suppose further  that $\psi$ is strictly increasing and continuous on an interval $J\subseteq [0,\infty)$ and $\psi(J)$  contains $B(x,y)$ for each $(x,y)\in S_1\times (S_2\cap U)$. The first assumption implies that $\psi^{-1}$ exists on $\psi(J)$ and is strictly  increasing there, and the second assumption combined with \beqref{eq:(1-lambda)psiB(x,y)} imply that $\|x-y\|\leq \psi^{-1}(B(x,y))$ for each $(x,y)\in S_1\times (S_2\cap U)$. 

\item Fix $y\in U$ and for each $x\in X$ let $h(x):=B(x,y)$. Then $h$ is the sum of the continuous linear (hence convex and lower semicontinuous) function $L(x):=-b(y)-\langle b'(y),x-y\rangle$, $x\in X$, and the function $b$, which is assumed to be convex and lower semicontinuous. Thus $h$ is convex and lower semicontinuous. If $b$ is also  assumed to be uniformly convex relative to $(S_1,S_2)$ with gauge $\psi$, then by adding the left-hand side of the equality $L(\lambda x+(1-\lambda)y)=\lambda L(x)+(1-\lambda)L(y)$ (for an arbitrary $\lambda\in (0,1)$) to the left-hand side of \beqref{eq:RelativeUniformConvexity}, and the right-hand side of this inequality to the right-hand side of \beqref{eq:RelativeUniformConvexity}, we see that $h$ is uniformly convex relative to $(S_1,S_2)$ and $\psi$ is a gauge of $h$. 

\item The equality $B(x,y)=0$ is immediate from \beqref{eq:BregmanDistance}  when $x=y\in U$. Now fix some $y\in X$ and let $x\in X$. If $y\notin U$, then $B(x,y)=\infty>0$. Assume now that $y\in U$. If $x\notin \dom(b)$, then $B(x,y)=\infty>0$. If $x\in \dom(b)$, then from \beqref{eq:(1-lambda)psiB(x,y)} (with $\psi_{b,\{x\},\{y\}}$ instead of $\psi$) and the fact that the optimal gauge of a convex function is nonnegative (even if the function itself is not uniformly convex: see Remark \bref{rem:TypesOfConvexity}\beqref{item:PsiIsNonnegative}), we have $B(x,y)\geq \psi_{b,\{x\},\{y\}}(\|x-y\|)\geq 0$, as required.  

Now assume that for some $x\in \dom(b)$ and $y\in U$ satisfying $x\neq y$, we know that $b$ is  strictly convex on the open interval $(x,y)$. Our goal is to prove that $B(x,y)>0$ in this case. The assumed strict convexity of $b$ on $(x,y)$ implies that $h(\lambda):=b(\lambda x+(1-\lambda)y)<\lambda b(x)+(1-\lambda)b(y)$ for each $\lambda\in (0,1)$. This inequality and the fact that $h(0)=b(y)<\infty$ imply that $h$ is  finite on $[0,1)$ and also that  $g(\lambda):=(h(\lambda)-h(0))/\lambda<b(x)-b(y)$ for all $\lambda\in (0,1)$. It will be shown in a moment that $g$ is increasing on $(0,1)$. Thus $\lim_{\lambda \to 0+}g(\lambda)$ exists and  $g(\lambda)$ decreases to it as $\lambda\to 0^+$. In particular, the limit is smaller than $b(x)-b(y)$. Since $b$ is G\^ateaux differentiable at $y$, it follows from \beqref{eq:Gateaux}  that $\lim_{\lambda \to 0+}g(\lambda)=\langle b'(y),x-y\rangle$. These facts and \beqref{eq:BregmanDistance} imply that $B(x,y)>0$, as claimed. 

It remains to show that $g$ is increasing (actually strictly increasing) on $(0,1)$. This fact is known, but we provide a proof anyway, since the proof is very short. Fix arbitrary $0<\lambda_1<\lambda_2<1$. The inequality $g(\lambda_1)<g(\lambda_2)$ is equivalent to the inequality $h(\lambda_1)<(\lambda_1/\lambda_2)h(\lambda_2)+(1-(\lambda_1/\lambda_2))h(0)$. This last  inequality holds because $\lambda_1/\lambda_2\in (0,1)$ and $h$ is strictly convex on $(0,1)$ (as follows from the strict convexity of $b$ on the interval $(x,y)$ and the definition of $h$). 

Finally, we need to show that if $b$ is strictly convex on $U$ and $x\in \dom(b)$ and $y\in U$ are given, then $B(x,y)=0$ if and only if $x=y$. We have already seen that if $x=y$, then $B(x,y)=0$. On the other hand, suppose that $B(x,y)=0$. Assume to the contrary that $x\neq y$. Since $y\in U$ and $x\in \dom(b)\subseteq \cl{U}$, the open line segment $(x,y)$ is contained in $U$ (see, for instance, \cite[Theorem 2.23(b), p. 28]{VanTiel1984book}). Hence we conclude from previous paragraphs that $B(x,y)>0$, a contradiction which proves that $x=y$. 

\item From the uniform convexity of $b$ on $S$ and from Part \beqref{BregProp:BregPsiS1S2} we know that  for each $i\in \N$, the inequality $\psi_{b,S}(\|x_i-y_i\|)\leq B(x_i,y_i)$ holds, where $\psi_{b,S}$ is the modulus of uniform convexity of $b$ on $S$ defined in \beqref{eq:Relative_psi}. Assume to the contrary that $\|x_i-y_i\|\nrightarrow 0$ as $i\to \infty$. Then there exist $\epsilon>0$ and subsequences $(x_{i_k})_{k=1}^{\infty}$, $(y_{i_k})_{k=1}^{\infty}$ of $(x_i)_{i=1}^{\infty}$ and $(y_i)_{i=1}^{\infty}$, respectively,  such that $\|x_{i_k}-y_{i_k}\|\geq \epsilon$ for all $k\in\N$. Since the modulus of uniform convexity is an increasing function on $[0,\infty)$ as follows from Lemma  \bref{lem:MonotonePsi} and since it attains positive values on $(0,\infty)$, the assumption $\lim_{i\to\infty}B(x_i,y_i)=0$ implies that  for each $k\in \N$, 
\begin{equation*}
0<\psi_{b,S}(\epsilon)\leq \psi_{b,S}(\|x_{i_k}-y_{i_k}\|)\leq B(x_{i_k},y_{i_k})\xrightarrow[k\to \infty]{}0. 
\end{equation*}
This is a contradiction which proves the assertion. 

\item Assume first that both sequences $(x_i)_{i=1}^{\infty}$ and $(y_i)_{i=1}^{\infty}$ are bounded. Let $S$ be the convex hull of the set $\{x_i,y_i: i\in \N\}$. This is a nonempty, bounded and convex subset of the convex set $\dom(b)$. Hence by our assumption $b$ is uniformly convex on it. Since $x_i\in S$ and $y_i\in S\cap U$  for each $i\in\N$, it follows from Part \beqref{BregProp:UniformlyConvexAndB(x_i,y_i)=0==>|x_i-y_i|=0} that $\lim_{i\to\infty}\|x_i-y_i\|=0$, as required. 

Now we show that the claim mentioned in the previous paragraph holds if we merely assume that only one of the sequences $(x_i)_{i=1}^{\infty}$ or $(y_i)_{i=1}^{\infty}$ (and not both simultaneously) is bounded. We modify below an argument which appears in \cite[p. 219]{SolodovSvaiter2000jour} in a somewhat different context. Assume first that $(x_i)_{i=1}^{\infty}$ is bounded. If, to the contrary, $(y_i)_{i=1}^{\infty}$ is unbounded, then this assumption and the fact that $(x_i)_{i=1}^{\infty}$ is bounded imply there exist some $\epsilon>0$ and subsequences $(x_{i_j})_{j=1}^{\infty}$ of $(x_{i})_{i=1}^{\infty}$ and $(y_{i_j})_{j=1}^{\infty}$ of $(y_{i})_{y=1}^{\infty}$, respectively, such that $\|y_{i_j}-x_{i_j}\|>\epsilon$ for each $j\in\N$. Define for all $j\in\N$
\begin{equation*}
z_{j}:=\frac{\epsilon}{\|y_{i_j}-x_{i_j}\|}y_{i_j}+\left(1-\frac{\epsilon}{\|y_{i_j}-x_{i_j}\|}\right)x_{i_j}. 
\end{equation*}
Then $\|z_j-x_{i_j}\|=\epsilon$ for each $j\in\N$ and hence $(z_j)_{j=1}^{\infty}$ is bounded. Moreover, since $z_j$ is a strict convex combination of a point from $\dom(b)$ and a point from $\Int(\dom(b))$, we can use \cite[Theorem 2.23(b), p. 28]{VanTiel1984book} to conclude that $z_j\in\Int(\dom(b))$. Thus \cite[Lemma 2.2]{SolodovSvaiter2000jour} (which is formulated in a finite-dimensional setting, but its simple proof holds word for word in any real normed space with the same assumptions on $b$ as we assume) and Part \beqref{BregProp:B=0} above imply that $B(x_{i_j},z_j)\leq B(x_{i_j},y_{i_j})$ for every $j\in\N$. Since we assume that $\lim_{i\to\infty}B(x_{i},y_{i})=0$, we have $\lim_{j\to\infty}B(x_{i_j},y_{i_j})=0$. Since $B$ is nonnegative (Part \beqref{BregProp:B=0}), we conclude that $\lim_{j\to\infty}B(x_{i_j},z_j)=0$. Hence we are in the setting of the first paragraph, where $(x_i)_{i=1}^{\infty}$ is replaced by $(x_{i_j})_{j=1}^{\infty}$ and $(y_i)_{i=1}^{\infty}$ is replaced by $(z_{j})_{j=1}^{\infty}$, and therefore $\lim_{j\to\infty}\|x_{i_j}-z_j\|=0$, in contrast to the equality $\|z_j-x_{i_j}\|=\epsilon$ for each $j\in\N$ which was established earlier. This contradiction proves that indeed $(y_i)_{i=1}^{\infty}$ is bounded, as claimed. 

We still need to consider the case where $(y_i)_{i=1}^{\infty}$ is bounded and to prove that this assumption implies that $(x_i)_{i=1}^{\infty}$ is bounded too. The proof of this case is similar to the previous case,  where we interchange the roles of $(x_i)_{i=1}^{\infty}$ and $(y_i)_{i=1}^{\infty}$, namely $z_{j}:=(\epsilon/\|x_{i_j}-y_{i_j}\|)x_{i_j}+(1-(\epsilon/\|x_{i_j}-y_{i_j}\|))y_{i_j}$ and so on (in this case \cite[Lemma 2.2]{SolodovSvaiter2000jour} and Part \beqref{BregProp:B=0} above imply that $B(z_j,y_{i_j})\leq B(x_{i_j},y_{i_j})$ for every $j\in\N$). 

Finally, it remains to consider the case where $b$ is uniformly convex on every nonempty, bounded and convex subset of $\dom(b)$, and where $(x_i)_{i=1}^{\infty}$ is a sequence in $\dom(b)$ and $(y_i)_{i=1}^{\infty}$ is a sequence in $U$ such that one of these sequences converges  strongly to some $z\in \overline{\dom(b)}$ and $\lim_{i\to\infty}B(x_i,y_i)=0$. The sequence which converges strongly is bounded. This fact, the assumption $\lim_{i\to\infty}B(x_i,y_i)=0$ and the previous paragraphs imply that the other sequence is bounded too and $\lim_{i\to\infty}\|x_i-y_i\|=0$. This equality and the triangle inequality imply that the other sequence converges to $z$ as well. 

\item Consider an arbitrary cluster point $x$ of $(x_i)_{i=1}^{\infty}$, i.e., $\lim_{k\to\infty}\|x-x_{i_k}\|=0$ for some subsequence $(x_{i_k})_{k=1}^{\infty}$ of $(x_i)_{i=1}^{\infty}$. At least one cluster point exists since  $(x_i)_{i=1}^{\infty}$ is contained in a compact subset of $X$. Since $\lim_{i\to\infty}\|y_i-y\|=0$ and $y\in U$, the boundedness of $(y_i)_{i=1}^{\infty}$ (which follows from the assumption that $(y_i)_{i=1}^{\infty}$ converges), together with \beqref{eq:BregmanDistance}, the continuity of $b$ on $\dom(b)$, and the continuity of $b'$ on $U$, all imply that $\lim_{k\to\infty}B(x_{i_k},y_{i_k})=B(x,y)$. Since $\lim_{i\to\infty}B(x_i,y_i)=0$ holds by our assumption, it follows that $B(x,y)=0$. Assume to the contrary that $x\neq y$. Since $\dom(b)$ is convex  (because $b$ is convex) and $y\in U=\Int(\dom(b))$ and $x\in\dom(b)$, the nonempty open interval $(x,y)$ is contained in $U$ (see \cite[Theorem 2.23(b), p. 28]{VanTiel1984book}). From this fact, the assumption that $b$ is  strictly convex on $U$, and from Part \beqref{BregProp:B=0} one has $x=y$, a contradiction. Thus indeed $x=y$. Since $x$ was an arbitrary cluster point of $(x_i)_{i=1}^{\infty}$, we conclude that all the cluster points of this sequence coincide with $y$. Hence $y=\lim_{i\to\infty}x_i$, as claimed. 

\item 
Assume to the contrary that the assertion is false, namely it is not true that $\lim_{i\to\infty}\|x_i-y_i\|=0$. Hence for some $\epsilon>0$ and a (strictly monotone) subsequence $(i_j)_{j=1}^{\infty}$ of natural numbers we have 
\begin{equation}\label{eq:x_i_j y_i_j} 
\|x_{i_j}-y_{i_j}\|\geq \epsilon.      
\end{equation}
Since $(y_i)_{i=1}^{\infty}$ is contained in a compact subset of $\overline{\dom(b)}$, also its subsequence $(y_{i_j})_{j=1}^{\infty}$ is contained in that compact subset. Hence $(y_{i_j})_{j=1}^{\infty}$ has a convergent subsequence $(y_{i_{j_k}})_{k=1}^{\infty}$ which converges to some $y$ in the closure of $\dom(b)$. Since we assume that $\lim_{i\to\infty}B(x_i,y_i)=0$, it follows  that $\lim_{k\to\infty}B(x_{i_{j_k}},y_{i_{j_k}})=0$. Since $B$ is a Bregman divergence, 
from Definition \bref{def:BregmanDiv}\beqref{BregmanDef:B(x_i,y_i)=0==>|x_i-y|=0} 
we conclude that also $\lim_{k\to\infty}x_{i_{j_k}}=y$ (here we also use the fact that 
$(x_i)_{i=1}^{\infty}$ is bounded and hence so is $(x_{i_{j_k}})_{k=1}^{\infty}$). Thus $\lim_{k\to\infty}\|x_{i_{j_k}}-y_{i_{j_k}}\|=\|y-y\|=0$. This equality contradicts \beqref{eq:x_i_j y_i_j}. 

\item 
Since $\lim_{i\to\infty}\|x_i-y_i\|=0$ and one of the sequences is bounded, the other sequence is bounded too. Let $W$ be the convex hull of the subset $\{x_i,y_i: i\in \N\}$. Then $W$ is a convex subset of the convex subset $U$ and $W$ is bounded because $\{x_i,y_i: i\in \N\}$ is a bounded subset of $U$. Since $b'$ is bounded on bounded  and convex subsets of $U$, there is $\lambda>0$ such that $\|b'(z)\|\leq \lambda$ for all $z\in W$. From the previous lines and the mean value theorem \cite[Theorem 1.8, p. 13]{AmbrosettiProdi1993book} it follows that $b$ is Lipschitz continuous on $W$ with a Lipschitz constant $\lambda$. The previous lines,  \beqref{eq:BregmanDistance}, the triangle inequality, the assumption $\lim_{i\to\infty}\|x_i-y_i\|=0$, the fact that $y_i\in W$ for each $i\in\N$, and Part \beqref{BregProp:B=0}, all imply that $0\leq B(x_i,y_i)\leq |b(x_i)-b(y_i)|+\|b'(y_i)\|\|x_i-y_i\|\leq 2\lambda\|x_i-y_i\|\to 0$ as $i$ tends to infinity, as required.

\item As in the proof of Part \beqref{BregProp:Bounded|x_i-y_i|=0==>B(x_i,y_i)=0}, both $(x_i)_{i=1}^{\infty}$  and $(y_i)_{i=1}^{\infty}$ are bounded. Let $W$ be the convex hull of the  bounded subset $\{x_i,y_i: i\in \N\}\subseteq \dom(b)$. Then $W\cap U$ is convex and bounded (and nonempty since $y_i\in W\cap U$ for each $i\in\N$) and hence, by our assumption on $b'$, there exists $\lambda>0$ such that $\|b'(z)\|\leq \lambda$ for all $z\in W\cap U$. Fix $i\in \N$ and define 
 $h_i(t):=b(tx_i+(1-t)y_i)$ for all $t\in [0,1]$.  Since $x_i\in\dom(b)$ and $y_i\in U=\Int(\dom(b))$, it follows that $tx_i+(1-t)y_i\in\dom(b)$. Hence $h_i$ is well defined and from \cite[Theorem 2.23(b), p. 28]{VanTiel1984book} we have $[y_i,x_i):=\{tx_i+(1-t)y_i: t\in [0,1)\}\subseteq U$. Since $[y_i,x_i)\subseteq W$ holds trivially, it follows that $b'(tx_i+(1-t)y_i)$ exists and $\|b'(tx_i+(1-t)y_i)\|\leq \lambda$ for all $t\in [0,1)$ and all $i\in \N$. This inequality, the fact that $b$ is G\^ateaux differentiable, and direct differentiation of $h_i$ according to the definition, all imply that  $h'_i$ exists in the interval $(0,1)$ and $h'_i(t)=\langle b'(tx_i+(1-t)y_i),x_i-y_i\rangle$ for each $t\in (0,1)$ and each $i\in\N$. Since $b$ is continuous on $\dom(b)$, it is continuous on the segment $[x_i,y_i]$ and hence $h_i$ is continuous on $[0,1]$. Therefore we can use the classical mean value theorem for one-dimensional real functions to conclude that there exists $t_i\in (0,1)$ such that $h_i(1)-h_i(0)=h'_i(t_i)(1-0)$. The above lines imply that for all $i\in\N$, 
\begin{multline}\label{eq:bLipschitz}
|b(x_i)-b(y_i)|=|h_i(1)-h_i(0)|=|h'_i(t_i)|=|\langle  b'(t_ix_i+(1-t_i)y_i),x_i-y_i\rangle|\\
\leq \|b'(t_ix_i+(1-t_i)y_i)\|\|x_i-y_i\|\leq\lambda \|x_i-y_i\|. 
\end{multline}
 This inequality, the previous lines, \beqref{eq:BregmanDistance}, the triangle inequality, the assumption $\lim_{i\to\infty}\|x_i-y_i\|=0$, the fact that $y_i\in W$ for each $i\in\N$, and Part \beqref{BregProp:B=0}, all imply that $0\leq B(x_i,y_i)\leq |b(x_i)-b(y_i)|+\|b'(y_i)\|\|x_i-y_i\|\leq 2\lambda\|x_i-y_i\|\to 0$ as $i$ tends to infinity, as required.
 
Finally, if we assume that $b$ is continuous on $\dom(b)$ and that $b'$ is uniformly continuous on all bounded and convex subsets of $U$, then it follows from Lemma \bref{lem:UniformContinuityConvexSet} that $b'$ is bounded on all bounded and convex subsets of $U$. Thus we can use the previous paragraph (with $x_i=x$ for all $i\in\N$) to conclude that if  $(y_i)_{i=1}^{\infty}$ is a sequence in $U$ which converges to some $x\in \dom(b)$ (hence, in particular, $(y_i)_{i=1}^{\infty}$ is bounded), then  $\lim_{i\to\infty}B(x,y_i)=0$.

\item Since $b'$ is locally bounded at each point of $U$, it is locally bounded at $x\in U$ and hence there exist a neighborhood $V(x)$ of $x$ and $\lambda>0$ such that for all $v\in V(x)$ we have $\|b'(v)\|<\lambda$. Since and $(y_i)_{i=1}^{\infty}$ converges to $x$, for each $i\in\N$ sufficiently large we have $y_i\in V(x)$ and thus $\|b'(y_i)\|<\lambda$. Hence $|\langle b'(y_i),x-y_i\rangle|\leq \lambda\|x-y_i\|\xrightarrow[i\to \infty]{}0$. Therefore we obtain from the continuity of $b$ on $U$ that $B(x,y_i)=b(x)-b(y_i)-\langle b'(y_i),x-y_i\rangle\xrightarrow[i\to \infty]{}0$, as required. 

Finally, suppose that $X$ is finite dimensional and $b:X\to(-\infty,\infty]$ is convex and G\^ateaux differentiable on $U:=\Int(\dom(b))\neq\emptyset$. Then $b'$ is actually continuous on $U$ because $b$ is convex and finite in $U$ (see \cite[Corollary 25.5.1, p. 246]{Rockafellar1970book}). Hence $b'$ is locally bounded at each point of $U$. In addition, since $X$ is finite dimensional, it is a Banach space, and therefore Remark \bref{rem:ContinuityPreBregman} implies that $b$ is continuous on $U$ (alternatively, since $X$ is finite dimensional and $b$ is convex and G\^ateaux differentiable at each point of $U$, it is Fr\'echet differentiable on $U$ by \cite[Theorem 25.2, p. 244]{Rockafellar1970book} and thus continuous there). As a result, the previous paragraph  implies that if $(y_i)_{i=1}^{\infty}$ is a sequence in $U$ which converges to some $x\in U$, then $\lim_{i\to\infty}B(x,y_i)=0$. 

\item $B(x,y_i)=b(x)-b(y_i)-\langle b'(y_i),x-y_i\rangle\xrightarrow[i\to \infty]{}0$ according to the assumptions. 

\item
 According to our assumptions, $r:=\sup\{\|x-y\|: y\in S\}<\infty$ and $x\in U$. Since $S\subseteq U$, it follows that   $\{x\}\bigcup S$ is a bounded subset of the convex subset $U$, and hence the convex hull $W$ of $\{x\}\bigcup S$ is a  bounded subset of $U$ as well. Because of our assumption on $b'$, there is $\lambda>0$ such that $\sup\{\|b'(z)\|: z\in W\}\leq \lambda$.   From the previous lines and the mean value theorem \cite[Theorem 1.8, p. 13]{AmbrosettiProdi1993book} it follows that $b$ is Lipschitz continuous on $W$ with a Lipschitz constant $\lambda$. This fact, \beqref{eq:BregmanDistance}, and the triangle inequality imply that 
 $B(x,y)\leq |b(x)-b(y)|+\|b'(y)\|\|x-y\|\leq 2\lambda\|x-y\|\leq 2\lambda r<\infty$ for all $y\in S$, as required. 
  
\item  According to our assumptions, $r:=\sup\{\|x-y\|: y\in S\}<\infty$ and $x\in \dom(b)$. Since $S\subseteq U\subseteq\dom(b)$, it follows that   $\{x\}\bigcup S$ is a bounded subset of the convex subset $\dom(b)$, and hence the convex hull $W$ of $\{x\}\bigcup S$ is a  bounded and convex subset of $\dom(b)$ as well. 
Therefore $W\cap U$ is a convex and bounded (and nonempty since $S\subseteq W$) subset of $U$ and from our assumption on $b'$ there is $\lambda>0$ such that $\sup\{\|b'(z)\|: z\in W\cap U\}\leq \lambda$. We can now follow the proof of Part \beqref{BregProp:Continuous_b|x_i-y_i|=0==>B(x_i,y_i)=0} (the lines before \beqref{eq:bLipschitz} and \beqref{eq:bLipschitz} itself, up to obvious modifications in the notation) to conclude that the inequality $\|b(x)-b(y)\|\leq \lambda\|x-y\|$ holds for all $y\in W\cap U$.  The previous lines, \beqref{eq:BregmanDistance}, and the triangle inequality imply that for all $y\in S$, we have
 $B(x,y)\leq |b(x)-b(y)|+\|b'(y)\|\|x-y\|\leq 2\lambda\|x-y\|\leq 2\lambda r<\infty$.

\item From Part \beqref{BregProp:BregPsiS1S2}, the assumption that $b$ is uniformly convex relative to $(\{x\},V)$, the assumption that $S''\subseteq V$, and the assumption that $\sigma:=\sup\{B(x,y): y\in S\}<\infty$, it follows that 
\begin{equation}\label{eq:psi<sigma}
\psi(\|x-y\|)\leq B(x,y)\leq \sigma,\quad \forall\, y\in S''. 
\end{equation}
Assume to the contrary that $S$ is unbounded. Hence $\sup\{\|x-y\|: y\in S\}=\infty$ and there exists a sequence $(y_i)_{i=1}^{\infty}$ of elements in $S$ such that $\lim_{i\to\infty}\|x-y_i\|=\infty$. 
Because $S=S'\cup S''$ and since $S'$ is assumed to be bounded, it follows that $y_i\in S''$ for all $i$ sufficiently large. Since we assume that $\lim_{t\to\infty}\psi(t)=\infty$, it follows that $\psi(\|x-y_i\|)>\sigma$ for all $i$ sufficiently large. This is a contradiction to \beqref{eq:psi<sigma}. The other case (in  which $b$ is uniformly convex relative to $(W,\{y\})$ and so forth) can be proved in a similar way.

\item  The first sub-part is an immediate corollary of Part \beqref{BregProp:B(x,S)Bounded==>SisBounded} where $S:=L(x,\gamma)$, $V:=U\cap \{w\in X: \|w\|\geq r_x\}$, $S':=\{w\in X: \|w\|<r_x\}\cap S$,  $S'':=V\cap S$. If, for some $x\in \dom(b)$, we also assume that $\psi_x$ is strictly increasing and continuous on $[0,\infty)$, then the assumption $\lim_{t\to\infty}\psi_x(t)=\infty$ implies, by using elementary calculus, that $\psi_x^{-1}$ exists on $[0,\infty)$. Hence from Part \beqref{BregProp:BregPsiS1S2} (with $S_1:=\{x\}$ and $S_2:=S''$) we have 
\begin{equation}\label{eq:S''psi-1}
\|x-z\|\leq \psi_x^{-1}(B(x,z))\leq \psi_x^{-1}(\gamma),\quad \forall z\in S''. 
\end{equation}
Now fix  $y$ and $z$ in $S$. Since $S=S'\cup S''$, either both points are in $S'$, or both of them are in $S''$, or one point is in $S'$ and the other is in $S''$. In the first case, since $S'$ is contained in the ball of radius $r_x$ centered at the origin, we have $\|y-z\|\leq 2r_x$. In the second case, $\|y-z\|\leq \|y-x\|+\|x-z\|\leq 2\psi_x^{-1}(\gamma)$ because of \beqref{eq:S''psi-1}. In the third case, if $y\in S'$ and $z\in S''$, then $\|y-z\|\leq \|y-x\|+\|x-z\|\leq \|x\|+\|y\|+\psi_x^{-1}(\gamma)\leq \|x\|+r_x+\psi_x^{-1}(\gamma)$, where we use the triangle inequality, the definition of $S'$ and \beqref{eq:S''psi-1}. A similar calculation holds if $z\in S'$ and $y\in S''$. We conclude that the diameter of $L(x,\gamma)$ is bounded above by $\max\{2\psi_x^{-1}(\gamma),2r_x,\psi_x^{-1}(\gamma)+r_x+\|x\|\}$, as required. The proof in the case of the second type level-set $L_2(y,\gamma)$, $y\in U$, $\gamma\in\R$ follows similar lines. 

Finally, suppose that $b$ is uniformly convex on $\dom(b)$. For each $x\in \dom(b)$ and each $y\in U$, if we let $r_x$ and $r_y$ to be any (fixed) nonnegative numbers, then $b$ is, in particular, uniformly convex relative to $(\{x\},\{w\in U: \|w\|\geq r_x\})$ and $(\{w\in \dom(b): \|w\|\geq r_x\},\{y\})$ with $\psi_x:=\psi_{b,\dom(b)}=:\psi_y$. Now, if $\dom(b)$ is unbounded, then since $\lim_{t\to\infty}\psi_{b,\dom(b)}(t)=\infty$ (Lemma \bref{lem:MonotonePsi}), we conclude from the previous paragraphs that all the first and second type level-sets of $B$ are bounded. These sets are obviously bounded also in the case where  $\dom(b)$ is bounded.

\item The proof of the first condition is immediate, and the proof of the second condition is simple too (by assuming to the contrary that $L(x,\gamma)$ is unbounded and arriving at a contradiction).

\item The proof of the first two conditions is as in the previous part. As for the third, since $b$ is essentially smooth and $X$ is finite-dimensional, it follows from \cite[Theorem 26.1, pp. 251--252]{Rockafellar1970book} and the definition of the subdifferential that $\partial b(y)=\emptyset$ for all $y\notin \Int(\dom(b))=U$ and $\partial b(y)=\{b'(y)\}$ for all $y\in U$. Therefore $\dom(\partial b)=U$. Since $b$ is strictly convex on $U$, it is strictly convex on $\dom(\partial b)$, and since $X$ is finite-dimensional, this property of $b$ is nothing but essential strict convexity (on $X$). Hence we can use either \cite[Theorem 3.7(iii)]{BauschkeBorwein1997jour} or \cite[Lemma 7.3(v)]{BauschkeBorweinCombettes2001jour} to conclude  that $L_2(y,\gamma)$ is bounded for all $y\in U$ and $\gamma\in \R$. Consider now the fourth part. Since $\dom(\partial b)\subseteq \dom(b)$ and since $b$ is assumed to be strictly convex on $\dom(b)$, it follows that $b$ is essentially strictly convex. Thus we can use either \cite[Theorem 3.7(iii)]{BauschkeBorwein1997jour} or \cite[Lemma 7.3(v)]{BauschkeBorweinCombettes2001jour} to conclude that all the second type level-sets of $B$ are bounded. 

\item Fix $x\in\dom(b)$. By \beqref{eq:BregmanDistance}, we have  
\begin{multline}\label{eq:Bx_iy_ib'}
B(x,x_i)-B(x,y_i)\\
=(b(x)-b(x_i)-\langle b'(x_i),x-x_i\rangle)-
(b(x)-b(y_i)-\langle b'(y_i),x-y_i\rangle)\\
=b(y_i)-b(x_i)-\langle b'(x_i),x-x_i\rangle+
\langle b'(x_i)+b'(y_i)-b'(x_i),x-y_i\rangle\\
=b(y_i)-b(x_i)-\langle b'(x_i),y_i-x_i\rangle
+\langle b'(y_i)-b'(x_i),x-y_i\rangle\\
=B(y_i,x_i)+\langle b'(y_i)-b'(x_i),x-y_i\rangle.
\end{multline}
Since $b'$ is uniformly continuous on each bounded and convex subset $S$ of $U$, it follows from Lemma \bref{lem:UniformContinuityConvexSet} that $b'$ is also bounded on each such subset $S$. We conclude from Part \beqref{BregProp:Bounded|x_i-y_i|=0==>B(x_i,y_i)=0} that 
$\lim_{i\to\infty}B(y_i,x_i)=0$. Because of \beqref{eq:Bx_iy_ib'}, it is sufficient to show that $\lim_{i\to\infty}(\langle b'(y_i)-b'(x_i),x-y_i\rangle)=0$ in order to conclude that $\lim_{i\to\infty}(B(x,x_i)-B(x,y_i))=0$. 

Indeed, since one of the sequences $(x_i)_{i=1}^{\infty}$ or $(y_i)_{i=1}^{\infty}$ is bounded, the condition $\lim_{i\to\infty} \|x_i-y_i\|=0$  implies the other sequence is bounded as well. Hence, the convex hull $W$ of $\{x_i,y_i: i\in\N\}$ is a bounded and convex subset of the convex subset $U$. Since  $(y_i)_{i\in\N}$ is bounded, there exists $r\in (0,\infty)$ such that 
$\|x-y_i\|<r$ for all $i\in \N$. Given $\epsilon>0$, the  uniform continuity of $b'$ on bounded and convex subsets of $U$ implies that $b'$ is uniformly continuous on $W$. Thus  there exists $\delta>0$ such that for all $(u,v)\in W^2$ satisfying  $\|u-v\|<\delta$, the inequality
$\|b'(u)-b'(v)\|<\epsilon/r$ holds. Since $\lim_{i\to\infty} \|x_i-y_i\|=0$, for all $i$ sufficiently large,  we have $\|x_i-y_i\|<\delta$. It follows that for all such $i$, 
\begin{equation*}
|\langle b'(y_i)-b'(x_i),x-y_i\rangle|\leq \|b'(y_i)-b'(x_i)\|\|x-y_i\|<(\epsilon/r)r=\epsilon. 
\end{equation*}
The assertion follows because $\epsilon$ was an arbitrary positive number. 

\item Fix $x\in\dom(b)$. Using  \beqref{eq:BregmanDistance} and simple calculations, we get 
\begin{multline*}
B(x,y_i)-B(y,y_i)=b(x)-b(y)-\langle b'(y_i), x-y\rangle\\
=b(x)-b(y)-\langle b'(y), x-y\rangle-\langle b'(y_i)-b'(y), x-y\rangle\\
=B(x,y)-\langle b'(y_i)-b'(y), x-y\rangle.
\end{multline*}
Since  $b'$ is weak-to-weak$^*$ sequentially continuous and $y_i\to y$ weakly, it follows, in particular, that $\lim_{i\to\infty}(\langle b'(y_i)-b'(y), x-y\rangle)=0$. Consequently, $B(x,y)=\lim_{i\to\infty}(B(x,y_i)-B(y,y_i))$. In the particular case where $X$ is finite-dimensional, if $b'$ is continuous on $U$, then it is  weak-to-weak$^*$ sequentially continuous because the weak and strong topologies (or the weak$^*$ and the strong topologies on the dual) coincide. Therefore, if $(y_i)_{i=1}^{\infty}$ is a sequence in $U$ which converges to some $y\in U$ and if $x\in \dom(b)$, then $B(x,y)=\lim_{i\to\infty}(B(x,y_i)-B(y,y_i))$ from the previous lines.
\item The assumptions on $\psi_x$ and well-known results from calculus imply that $\psi_x$ is invertible on $[0,\infty)$ and maps $[0,\infty)$ onto itself, that the inverse $\psi_x^{-1}$ is strictly increasing and continuous on $[0,\infty)$, and also that $\psi_x^{-1}(0)=0$. These properties, the assumption $\lim_{i\to\infty}B(x,y_i)=0$ and Part \beqref{BregProp:BregPsiS1S2} (with $S_1:=\{x\}$), imply that $0\leq \|y_i-x\|\leq \psi_x^{-1}(B(x,y_i)){\,\xrightarrow[i\to \infty]{}}\,0$, namely $\lim_{i\to\infty}\|y_i-x\|=0$, as required. The proof of the second assertion (regarding the convergence of $(x_i)_{i=1}^{\infty}$ to $y$) is similar. 
\end{enumerate} 
\end{proof}

The following two corollaries present useful systems of conditions which suffice to ensure that a given function be a Bregman function. 
\begin{cor}\label{cor:BregmanFunction-b'-is-unicont}
Let $(X,\|\cdot\|)$ be a real normed space. Suppose that $b:X\to(-\infty,\infty]$ is convex and  lower semicontinuous on $X$, continuous on $\dom(b)$, and G\^ateaux differentiable in the subset $U:=\Int(\dom(b))$ which is assumed to be nonempty. Suppose further that $b$ is uniformly convex on all nonempty, bounded, and convex subsets of $\dom(b)$, that $b'$ is uniformly continuous  on all nonempty, bounded, and convex subsets of $U$, and that for each $x\in \dom(b)$, there exists $r_x\geq 0$ such that $b$ is uniformly convex relative to $(\{x\},U\cap \{w\in X: \|w\|\geq r_x\})$ with a gauge $\psi_x$ satisfying $\lim_{t\to\infty}\psi_x(t)=\infty$. Then $b$ is a sequentially consistent Bregman function. If, in addition, $b'$ is weak-to-weak$^*$ sequentially continuous on $U$, then $b$ satisfies the limiting difference property. 
\end{cor}
\begin{proof}
Parts \beqref{BregmanDef:U} and \beqref{BregmanDef:ConvexLSC} of Definition \bref{def:BregmanDiv} are satisfied by the assumptions on $b$ and Remark \bref{rem:TypesOfConvexity}\beqref{item:UniformlyConvex==>StrictlyConvex}, Definition \bref{def:BregmanDiv}\beqref{BregmanDef:B} is just the definition of the Bregman divergence $B$, Definition \bref{def:BregmanDiv}\beqref{BregmanDef:BoundedLevelSet} follows from Proposition \bref{prop:BregmanProperties}\beqref{BregProp:LevelSetBounded}, Definition \bref{def:BregmanDiv}\beqref{BregmanDef:|x-y_i|=0==>B(x,y_i)=0} follows from Proposition  \bref{prop:BregmanProperties}\beqref{BregProp:Continuous_b|x_i-y_i|=0==>B(x_i,y_i)=0}. Definition \bref{def:BregmanDiv}\beqref{BregmanDef:B(x_i,y_i)=0==>|x_i-y|=0} and the sequential consistency of $b$ follow from Proposition \bref{prop:BregmanProperties}\beqref{BregProp:B(x_i,y_i)=0AndOneSequenceBounded==>|x_i-y_i|=0}. The limiting difference property follows from Proposition \bref{prop:BregmanProperties}\beqref{BregProp:B(x,y)=B(x,y_i)-B(y,y_i)}. 
\end{proof}

\begin{cor}\label{cor:BregmanFunctionFiniteDim-b'-is-cont}
Let $(X,\|\cdot\|)$ be a real normed space. Suppose that $b:X\to(-\infty,\infty]$ is convex and  lower semicontinuous on $X$ and continuous and uniformly convex on $\dom(b)$. Assume further that  $U:=\Int(\dom(b))$ is nonempty, that $b'$ exists and is uniformly continuous on every bounded and convex subset of $U$ and weak-to-weak$^*$ sequentially continuous on $U$. Then $b$ is a sequentially consistent Bregman function which satisfies the limiting difference property and the second type level-sets of $B$ are bounded. In particular, if $X$ is finite-dimensional, then any  function $b:X\to(-\infty,\infty]$, which is  convex and  lower semicontinuous on $X$,  continuous and uniformly convex on $\dom(b)$, and continuously differentiable on $U$, is a sequentially consistent Bregman function which has the limiting difference property and the second type level-sets of $B$ are bounded. Specializing even more, if $X$ is finite-dimensional, then any  function $b:X\to\R$ which is continuously differentiable and uniformly convex on $X$ is a sequentially consistent Bregman function which has the limiting difference property and the second type level-sets of $B$ are bounded.
\end{cor}
\begin{proof}
The first part is a consequence of Corollary \bref{cor:BregmanFunction-b'-is-unicont} and Proposition \bref{prop:BregmanProperties}\beqref{BregProp:LevelSetBounded} because the modulus of uniform convexity $\psi_{b,X}$ of $b$ on $X$ satisfies $\lim_{t\to\infty}\psi(t)=\infty$ as proved in Lemma \bref{lem:MonotonePsi}. The second part  is a particular case of the first part because the strong and weak topologies coincide when the dimension is finite (and so do the strong and weak$^*$ topologies on the dual space) and any function which is continuous on all bounded and convex subsets of $U$ is, because of compactness, uniformly continuous on each such subset. The third part is a consequence of the second one in the particular case where $U=X$.
\end{proof}

\section{Strong convexity and weak-to-weak$^*$ sequential continuity}\label{sec:StrongConvexityWeak-to-weak*} 
In this section we discuss a few sufficient conditions which are related to some parts of Proposition \bref{prop:BregmanProperties}. More precisely,  in Subsection \bref{subsec:StrongConvexity} below we discuss issues related to strong convexity, and in Subsection \bref{subsec:weak-to-weak*} we discuss issues related to weak-to-weak$^*$ sequential continuity. 

\subsection{Strong convexity}\label{subsec:StrongConvexity}
In this subsection we present simple sufficient conditions for the strong convexity of a function. 
When combined with Remark \bref{rem:TypesOfConvexity}\beqref{item:mu[x,y]}, they allow 
one to construct uniformly convex (or relatively uniformly convex) functions and hence to obtain functions  satisfying important sufficient conditions needed in Proposition \bref{prop:BregmanProperties} and the corollaries coming after it. We note that 
Proposition \bref{prop:BregmanStronglyConvex}\beqref{item:eta_S} below is known in the literature in  different settings, for instance in finite-dimensional Euclidean spaces (see, for example, \cite[Theorem 2.1.11, p. 66]{Nesterov2004book}). 

\begin{prop}\label{prop:BregmanStronglyConvex}
Let $X\neq\{0\}$ be a real normed space with a norm $\|\cdot\|$. Let $b:X\to(-\infty,\infty]$ and assume that $U:=\Int(\dom(b))\neq\emptyset$. Suppose further that $b$ is continuous on $\dom(b)$ and twice continuously Fr\'echet differentiable in $U$. Then the following statements hold: 
\begin{enumerate}[(I)]
\item\label{item:eta_xy} Given $x,y\in U$, suppose that $b$ has a strictly positive definite Hessian on $[x,y]$ in the sense that 
\begin{equation}\label{eq:b''xy}
\eta[x,y]:=\inf\{b''(z)(w,w): z\in [x,y],w\in X,\|w\|=1\}>0.
\end{equation}
Then $b$ is strongly convex on $[x,y]$ with a strong convexity parameter $\mu:=\eta[x,y]$. 

\item\label{item:eta_S} Given a nonempty and convex subset $S$ of $U$, if $b$ has a strictly positive definite Hessian on $S$ in the sense that 
\begin{equation}\label{eq:b''}
\eta[S]:=\inf\{b''(z)(w,w): z\in S,w\in X,\|w\|=1\}>0, 
\end{equation}
then $b$ is strongly convex on $S$ with a strong convexity parameter $\mu:=\eta[S]$. Moreover,  $b$ is strongly convex on $\cl{S}\cap\dom(b)$ with the same  strong convexity parameter $\eta[S]$.
\item\label{item:StronglyConvex-clU} If \beqref{eq:b''} holds for each nonempty, bounded, and convex subset $S$ of $U$, then $b$ is strongly convex on these subsets and, moreover, strongly convex also on each nonempty, bounded and convex subset of $\dom(b)$. 
\end{enumerate}
\end{prop}

\begin{proof}
\begin{enumerate}[(I)]
\item We first prove Item \beqref{item:eta_xy}. Fix $z_1,z_2\in [x,y]$ and define, for all $t\in [0,1]$, the real function $h(t):=b(tz_1+(1-t)z_2)$. Since $[x,y]\subseteq U$, the chain rule and direct differentiation  show that $h'(t)=\langle b'(tz_1+(1-t)z_2),z_1-z_2\rangle$ and  $h''(t)=b''(tz_1+(1-t)z_2)(z_1-z_2,z_1-z_2)$ for all $t\in [0,1]$, where here we adopted the standard convention of identifying the  second derivative (Hessian) with a bilinear form \cite[p. 23]{AmbrosettiProdi1993book}.  If $z_1\neq z_2$, then  from \beqref{eq:b''xy}, we have 
\begin{equation}\label{eq:x1x2}
b''(z)(z_1-z_2,z_1-z_2)\geq \eta[x,y]\|z_1-z_2\|^2
\end{equation}  
for all $z\in [x,y]$, and, in particular, for $z:=tz_1+(1-t)z_2$. Inequality \beqref{eq:x1x2}  also holds true when $z_1=z_2$, since in this case both sides are equal to 0. Therefore  $h''(t)\geq \eta[x,y]\|z_1-z_2\|^2$ for each $t\in [0,1]$. Since $h''$ is continuous, integration and the fundamental theorem of calculus yield the inequality  $h(t)\geq h(0)+h'(0)t+0.5\eta[x,y]\|z_1-z_2\|^2 t^2$ for all $t\in [0,1]$. Hence, by substituting $t:=1$ and using the fact that $h'(0)=\langle b'(z_2),z_1-z_2\rangle$,  we obtain 
\begin{equation}\label{eq:bx_1x_2}
b(z_1)\geq b(z_2)+\langle b'(z_2),z_1-z_2\rangle+\frac{1}{2}\eta[x,y]\|z_1-z_2\|^2.
\end{equation}
This inequality holds for all $z_1,z_2\in [x,y]$. Hence, by fixing $\lambda\in [0,1]$  and  substituting   $z_1:=x$ and $z_2:=\lambda x+(1-\lambda)y$ in \beqref{eq:bx_1x_2}, we get 
\begin{multline}\label{eq:bx<=y}
b(x)\geq b(\lambda x+(1-\lambda)y)+\langle b'(\lambda x+(1-\lambda)y),(1-\lambda)(x-y)\rangle
+\frac{1}{2}\eta[x,y](1-\lambda)^2\|x-y\|^2.
\end{multline}
Similarly, by substituting   $z_1:=y$ and $z_2:=\lambda x+(1-\lambda)y$ in \beqref{eq:bx_1x_2}, we obtain 
\begin{equation}\label{eq:by<=x}
b(y)\geq b(\lambda x+(1-\lambda)y)+\langle b'(\lambda x+(1-\lambda)y),\lambda(y-x)\rangle+\frac{1}{2}\eta[x,y]\lambda^2\|x-y\|^2.
\end{equation}
By multiplying \beqref{eq:bx<=y} by $\lambda$, multiplying  \beqref{eq:by<=x} by $1-\lambda$, and adding these inequalities, we get  
\begin{equation}\label{eq:eta[x,y]}
\lambda b(x)+(1-\lambda)b(y)\geq b(\lambda x+(1-\lambda)y)+\frac{1}{2}\eta[x,y]\lambda(1-\lambda)\|x-y\|^2.
\end{equation}
In other words, $b$ is indeed strongly convex on $[x,y]$ with a strong convexity parameter  $\mu:=\eta[x,y]$. 

\item From \beqref{eq:b''xy} and \beqref{eq:b''} we have $\eta[S]\leq \eta[x,y]$ for all $x,y\in S$. This  fact and \beqref{eq:eta[x,y]} show that 
\begin{multline}\label{eq:eta_S}
\lambda b(x)+(1-\lambda)b(y)\geq 
b(\lambda x+(1-\lambda)y)+0.5\eta[x,y]\lambda(1-\lambda)\|x-y\|^2\\
\geq b(\lambda x+(1-\lambda)y)+0.5\eta[S]\lambda(1-\lambda)\|x-y\|^2,
\end{multline}
that is, $b$ is strongly convex on $S$ with $\mu:=\eta[S]$ as a strong convexity parameter. 

Now let $x,y\in \cl{S}\cap\dom(b)$ be arbitrary and let $(x_i)_{i=1}^{\infty}$ and $(y_i)_{i=1}^{\infty}$ be two sequences in $S$ which converge to $x$ and $y$,  respectively (of course, $\cl{S}\cap\dom(b)$ is nonempty because it contains $S$). From the previous paragraph we know  that \beqref{eq:RelativeUniformConvexity} holds for all $i\in\N$, where $x_i$ and $y_i$ replace $x$ and $y$, respectively, where $\lambda\in (0,1)$ is arbitrary, where  $S_1:=S$, $S_2:=S$, and where $\psi(t):=\frac{1}{2}\eta[S]t^2$, $t\in [0,\infty)$. By taking the limit $i\to\infty$ and using the convexity of $S$ and the continuity of $b$ on $\dom(b)$ (and hence on $\cl{S}\cap \dom(b)$) and the continuity of the norm, we see that  \beqref{eq:RelativeUniformConvexity} also holds with $x$ and $y$. In other words, $b$ is strongly convex on $\cl{S}$ with $\eta[S]$ as a strong  convexity parameter. 
\item Now we show that $b$ is strongly convex on any nonempty, convex, and bounded subset  of $\dom(b)$. Let  $K$ be such a subset. Fix some $z_*\in U$ and consider the subset $K(z_*)$ which is the union of all half-open line segments of the form $[z_*,z)$, $z\in K$, that is, $K(z_*):=\cup_{z\in K}[z_*,z)$. This is a bounded subset of $X$ since its diameter is bounded by $2\sup\{d(z_*,z): z\in K\}<\infty$ by the triangle inequality and the fact that  $K$ is bounded. Since $U$ is open and convex and since each $z\in K$ satisfies $z\in\cl{U}$, it follows from \cite[Theorem 2.23, p. 28]{VanTiel1984book} that each segment $[z_*,z)$, $z\in K$ is contained in $U$. Hence $K(z_*)\subseteq U$. Let $S$ be the convex hull of $K(z_*)$ (in fact, $S=K(z_*)$, but we will not use this fact). Since $K(z_*)$ is bounded, $S$ is a bounded subset of $U$. Since we assume that \beqref{eq:b''} holds for each nonempty, bounded, and convex subset of $U$, it follows from Part \beqref{item:eta_S} that $b$ is strongly convex on $S$ with some parameter $\mu[S]>0$. 

Fix $x,y\in K$ and $\lambda\in [0,1]$. From the construction of $S$ it follows that there are sequences $(x_i)_{i=1}^{\infty},(y_i)_{i=1}^{\infty}$ in $S$ such that $x=\lim_{i\to\infty}x_i$ and $y=\lim_{i\to\infty}y_i$: we can simply take, say,  $x_i:=(1-(1/i))x+(1/i)z_*$ and $y_i:=(1-(1/i))y+(1/i)z_*$ for each $i\in\N$. Since $b$ is strongly convex on $S$ with the parameter $\mu[S]$, it follows that \beqref{eq:RelativeUniformConvexity} holds with $S_1:=S$, $S_2:=S$, with $\psi(t)=\frac{1}{2}\mu[S]t^2$  as a gauge, with an arbitrary $\lambda\in (0,1)$, and with $x$ and $y$ replaced by $x_i$ and  $y_i$, respectively. By going to the limit $i\to\infty$ in \beqref{eq:RelativeUniformConvexity} and using the continuity on $\dom(b)$ of both $b$ and the norm, we see that \beqref{eq:RelativeUniformConvexity} holds also with $x$ and $y$. Since $x$ and $y$ were arbitrary in $S$ and $\lambda$ was arbitrary in $(0,1)$, it follows that $b$ is strongly convex on $K$ with $\mu[S]$ as a strong convexity parameter. 

\end{enumerate}
\end{proof}

\begin{remark}
It was essentially observed in \cite[Theorem 1 and its proof]{Araujo1988jour} that  if $X\neq \{0\}$ is a real Banach space, and  there is some function $b:U\to\R$ which is twice continuously Fr\'echet   differentiable on a nonempty and open subset $U\subseteq X$ and satisfies \beqref{eq:b''} at some $z\in U$, then $X$ is Hilbertian. However, very few details were given in \cite{Araujo1988jour} regarding this claim. In fact, a more general claim holds: if $X\neq \{0\}$ is a real normed space with a norm $\|\cdot\|$ and there is a bounded bilinear form $M:X^2\to\R$ which has the properties that it is symmetric (that is, $M(x,y)=M(y,x)$ for all $(x,y)\in X^2$) and for some $\eta>0$ and all unit vectors $x\in X$ we have $M(x,x)\geq \eta$ (we refer to this latter condition as the ``coercivity assumption''), then $M$ induces an inner product on $X$ and $g(x):=\sqrt{M(x,x)}$, $x\in X$, is a well-defined norm on $X$ which is equivalent to the original norm $\|\cdot\|$. We sketch the proof of this claim in the next paragraph. Thus, if the normed space $X$ in Proposition \bref{prop:BregmanStronglyConvex} is a Banach space, then by taking $M:=b''(z)$ and using the well-known fact that a continuous bilinear form defined on a Banach space is bounded, we conclude that $X$ is Hilbertian. 

To see that $\langle x,y\rangle:=M(x,y)$, $(x,y)\in X^2$, is an inner product on $X^2$, we observe that linearity in each component is a  consequence of the bilinearity of $M$, and $\langle\cdot,\cdot\rangle$ is symmetric because $M$ is symmetric. In addition, the coercivity assumption on $M$ implies that $\langle x,x\rangle=M(x,x)\geq \eta \|x\|^2\geq 0$ and hence $\langle\cdot,\cdot\rangle$ is nonnegative, and if $\langle x,x\rangle=0$, then $\eta\|x\|^2=0$, hence $x=0$. The bilinearity of $M$ also implies that $\langle 0,0\rangle=M(0,0)=0$. We conclude from the previous lines that $\langle\cdot,\cdot\rangle$ is an inner product. Thus the definition of $g$ and the basic theory of inner products (such as the Cauchy-Schwarz inequality)  imply that $g$ is a norm which is induced by an inner product. It remains to see that $g$ is equivalent to $\|\cdot\|$. Indeed, since $M$ is bounded, $|M(x,y)|\leq \|M\|\|x\|\|y\|$ for all $(x,y)\in X^2$, and therefore $g(x)\leq \sqrt{\|M\|}\|x\|$ for each $x\in X$. This inequality proves the first part of the equivalent norm claim, since we have $\|M\|>0$ (otherwise $M\equiv 0$, a contradiction to the coercivity assumption $M(x,x)\geq \eta>0$ for every unit vector $x\in X$). On the other hand, the coercivity assumption on $M$ and the definition of $g$ imply  that $\sqrt{\eta}\|x\|\leq g(x)$ for each $x\in X$. Consequently, $g$ is equivalent to $\|\cdot\|$. 
\end{remark}

\begin{prop}\label{prop:StronglyConvexSum}
Let $I\neq \emptyset$ be either a finite or a countable set.  Let $((X_i,\|\cdot\|_{X_i}))_{i\in I}$ be a sequence of normed spaces. For each $i\in I$, suppose that $b_i:X_i\to (-\infty,\infty]$ is strongly convex on some nonempty subset $S_i$ of $C_i:=\dom(b_i)$ with a strong convexity parameter $\mu_i$. Assume also that $\mu:=\inf\{\mu_i: i\in I\}>0$. Define $X:=\{(x_i)_{i\in I}: x_i\in X_i\,\forall i\in I,\, \sum_{i\in I}\|x_i\|_{X_i}^2<\infty\}$ and endow $X$ with an arbitrary norm $\|\cdot\|$ which is semi-equivalent to the norm  $\|(x_i)_{i\in I}\|_{\#}:=\sqrt{\sum_{i\in I}\|x_i\|^2_{X_i}}$, $(x_i)_{i\in I}\in X$, in the sense that there exists $c>0$ such that  $\|(x_i)_{i\in I}\|_{\#}\geq c \|(x_i)_{i\in I}\|$ for each $(x_i)_{i\in I}\in X$. Assume that  $\sum_{i\in I}b_i(x_i)$ is well defined (converges to a real number)  for each $(x_i)_{i\in I}\in C:=\bigoplus_{i\in I}\dom(b_i)$ and let $b:X\to(-\infty,\infty]$ be defined by $b((x_i)_{i\in I}):=\sum_{i\in I}b_i(x_i)$ for each $(x_i)_{i\in I}\in C$ and $b((x_i)_{i\in I}):=\infty$ if $(x_i)_{i\in I}\notin C$. Then $b$ is strongly convex on $S:=\bigoplus_{i\in I}S_i$ with $\mu[S]:=c^2 \mu$ as a strong convexity parameter. 
\end{prop}
\begin{proof}
Fix $\lambda\in (0,1)$, $x:=(x_i)_{i\in I}\in S$, and $y:=(y_i)_{i\in I}\in S$. From the definition of $b$, the convexity of each $C_i$, $i\in I$, the convexity of $C$, and the strong convexity of every $b_i$ on $S_i$, $i\in I$, we have 
\begin{multline}\label{eq:bI}
b(\lambda x+(1-\lambda)y)=b((\lambda x_i+(1-\lambda)y_i)_{i\in I})=\sum_{i\in I}b_i(\lambda x_i+(1-\lambda)y_i)\\
\leq \sum_{i\in I}\left(\lambda b_i(x_i)+(1-\lambda)b_i(y_i)-\frac{1}{2}\mu_i\lambda(1-\lambda)\|x_i-y_i\|_{X_i}^2\right)\\
=\lambda\sum_{i\in I}b_i(x_i)+(1-\lambda)\sum_{i\in I}b_i(y_i)-\frac{1}{2}\lambda(1-\lambda)\sum_{i\in I}\mu_i\|x_i-y_i\|_{X_i}^2\\
=\lambda b(x)+(1-\lambda)b(y)-\frac{1}{2}\lambda(1-\lambda)\sum_{i\in I}\mu_i\|x_i-y_i\|_{X_i}^2\\
\leq \lambda b(x)+(1-\lambda)b(y)-0.5\lambda(1-\lambda)\mu\sum_{i\in I}\|x_i-y_i\|_{X_i}^2\\
=\lambda b(x)+(1-\lambda)b(y)-\frac{1}{2}\lambda(1-\lambda)\mu\|x-y\|_{\#}^2\\
\leq \lambda b(x)+(1-\lambda)b(y)-\frac{1}{2}\lambda(1-\lambda)\mu c^2\|x-y\|^2.
\end{multline}
The above inequality proves the assertion up to clarifying the small issue related to the convergence of the series which appear in \beqref{eq:bI} when $I$ is infinite. Most of these series are nothing but the values of $b$ at some points, and hence they converge. The only doubt is regarding the series $\sum_{i\in I}\mu_i\|x_i-y_i\|_{X_i}^2$. This series converges absolutely  to either a nonnegative number or to infinity. However, the sum cannot be infinity because it is bounded from above by $(\lambda b(x)+(1-\lambda)b(y)-b(\lambda x+(1-\lambda)y))/(0.5\lambda(1-\lambda))$ as follows from repeating the analysis of \beqref{eq:bI} with partial sums and taking their size to infinity (namely, one works with $I_n$  instead of $I$, where $(I_n)_{n=1}^{\infty}$ is an increasing family of finite subsets of $I$ satisfying  $I=\cup_{n=1}^{\infty}I_n$, and then takes $n$ to infinity). 
\end{proof}

\subsection{weak-to-weak$^*$ sequential continuity}\label{subsec:weak-to-weak*}
We finish this section with the following proposition which describes a sufficient condition for a mapping to be weak-to-weak$^*$ sequentially continuous, hence helping in establishing examples of functions satisfying Proposition \bref{prop:BregmanProperties}\beqref{BregProp:B(x,y)=B(x,y_i)-B(y,y_i)}. See also Remark \bref{rem:SchauderBasis} following this proposition for examples of corresponding Banach spaces satisfying the conditions mentioned in Proposition \bref{prop:weak-to-weak*}. Before formulating this proposition  we need a short  discussion and a definition. Recall that a Schauder basis of a real  infinite-dimensional Banach space $(X,\|\cdot\|)$ is a sequence $(e_k)_{k=1}^{\infty}$ of elements in $X$ having the property that each $x\in X$ can be represented uniquely as a countable linear combination of the basis, namely, for each $x\in X$, there exists a unique sequence $(x(k))_{k=1}^{\infty}$ of real numbers (the coordinates of $x$)  such that $x=\sum_{k=1}^{\infty}x(k) e_k$. A standard (algebraic) basis in a finite-dimensional space can also be regarded as a Schauder basis. 

\begin{defin}\label{def:FiniteContinuousL2}
Let $(X,\|\cdot\|)$ be a real Banach space which has a Schauder basis $(e_k)_{k\in\N}$ 
(or $(e_k)_{k=1}^n$ when $\dim(X)=n\in\N$).  
\begin{enumerate}[(a)]
\item 
Given a nonempty subset  $U\subseteq X$ and a function $g:U\to \R$, we say that $g$ depends continuously on finitely many components if there exists a nonempty finite subset of indices $I\subset \N$ and a  continuous function $\tilde{g}:\Pi(U)\to\R$ such that $g(x)=\tilde{g}(\Pi(x))$  for all $x=\sum_{k=1}^{\infty}x(k) e_k\in U$, where $\Pi:X\to X$ is the function defined by  $\Pi(x):=\sum_{k\in I}x(k) e_k$ for all $x\in X$, that is, $\Pi$ is a linear projection from $X$ onto the finite-dimensional normed subspace $S_I:=\textnormal{span}\{e_k: k\in I\}$. Here the norm on $S_I$ (and hence on $\Pi(U):=\{\Pi(u): u\in U\}$) is induced by the norm of $X$. 
\item We say that a sequence $(x_i)_{i=1}^{\infty}$ in $X$ converges component-wise if  $\lim_{i\to\infty}x_{i}(k)$ exists (as a real number) for each $k\in\N$, where $x_{i}(k)$ is the $k$-th coordinate of $x_i$ in its representation by the given Schauder basis. 
\end{enumerate}
\end{defin}
\begin{defin}
We say that a (real or complex) Banach space $(X,\|\cdot\|)$ has the component-* property if  both $X$ and its dual $X^*$ have Schauder bases, and each sequence in $X^*$ which is bounded and converges component-wise also converges in the weak$^*$ topology.
\end{defin}

\begin{prop}\label{prop:weak-to-weak*} Let $(X,\|\cdot\|)$ be a real Banach space with a dual $(X^*,\|\cdot\|_*)$. Suppose that $X$ has the component-* property. Denote by $(e_k)_{k\in\N}$ and  $(f_j)_{j\in\N}$ (or $(e_k)_{k=1}^n$ and $(f_j)_{j=1}^n$  when $\dim(X)=n\in\N$) the Schauder bases of $X$ and $X^*$, respectively.  Given a nonempty subset  $U\subseteq X$, suppose that a function $h:U\to X^*$, which has the form $h(x)=\sum_{j=1}^{\infty}h_j(x)f_j$, $x\in U$, maps every bounded sequence in $U$ to a bounded sequence in $X^*$. Suppose also that for each $j\in\N$, the $j$-th functional coordinate  $h_j:U\to\R$ in the Schauder basis representation of $h$ depends continuously on finitely many components. Then $h$ is weak-to-weak$^*$ sequentially continuous on $U$. 
\end{prop}
\begin{proof}
The assertion is immediate in the case where $X$ is finite-dimensional, since in this case the weak and strong topologies coincide on $X$, and the weak and weak$^*$ topologies coincide on $X^*$. Hence from now on $X$ is assumed to be infinite-dimensional. 
Suppose that $x_{\infty}\in U$  is the weak limit of some sequence $(x_i)_{i=1}^{\infty}$ where $x_i\in U$ for all $i\in\N$. We need to show that $\lim_{i\to\infty}\langle h(x_i),w\rangle=\langle h(x_{\infty}),w\rangle$ for all $w\in X$. By our assumption on $X^*$  it is sufficient to show that the sequence $(h(x_i))_{i=1}^{\infty}$ is bounded and converges  component-wise. Since $(x_i)_{i=1}^{\infty}$ converges weakly (to $x_{\infty}$), this sequence is bounded \cite[p. 58]{Brezis2011book}, and hence, by our assumption on $h$, we conclude that $(h(x_i))_{i=1}^{\infty}$ is bounded.  

It remains to show component-wise convergence, that is, $\lim_{i\to\infty}h_j(x_i)=h_j(x_{\infty})$ for each $j\in \N$. Since  for each $j\in\N$, we assume that the functional coordinate $h_j$ depends continuously on finitely many components, it follows that for each $j\in\N$, there exists  a nonempty  finite subset $I_j\subset\N$ of indices and a continuous function $\tilde{h}_j:\Pi_j(U)\to\R$ such that $h_j(x)=\tilde{h}_j(\Pi_j(x))$ for all $x\in U$. Fix $\epsilon>0$ and $j\in\N$, and denote $S_j:=\textnormal{span}\{e_k: k\in I_j\}\subseteq X$. Since $\tilde{h}_j$ is continuous at each point of $\Pi_j(U)$, and, in particular, at $\Pi_j(x_{\infty})$, there exists $\delta'_j>0$ such that for all $z$ in the intersection of $\Pi_j(U)$ with the ball $C'_j$ of radius $\delta'_j$ and center $\Pi_j(x_{\infty})$, we have  
\begin{equation}\label{eq:h_j<epsilon}
|\tilde{h}_j(z)-\tilde{h}_j(\Pi_j(x_{\infty}))|<\epsilon. 
\end{equation}
Since all the norms on a finite-dimensional normed space are equivalent \cite[p. 197]{GohbergGoldberg1981book}, so are the max norm defined on $S_j$ by $\|\sum_{k\in I_j}\alpha_k e_k\|_{\infty}:=\max\{|\alpha_k|: k\in I_j\}$ and the norm  of $S_j$ which is induced by the norm of $X$. Thus there exists  $\delta_j>0$ such that the max norm ball $C_j$ of radius $\delta_j$ with center $\Pi_j(x_{\infty})$ is contained in the above-mentioned ball $C'_j$.  Since $\lim_{i\to\infty}x_i=x_{\infty}$ weakly, it follows  that for all $k\in I_j$,  the $k$-th coordinate $x_{i}(k)$ of $x_i$ converges, as $i$  tends to infinity, to the $k$-th coordinate $x_{\infty}(k)$ of $x_{\infty}$ (this is because the linear functional which assigns to each $x\in X$  the $k$-th coordinate in the basis representation of $x$ is continuous \cite[p. 83]{Beauzamy1982book}; in \cite[p. 83]{Beauzamy1982book} it is assumed that the basis is normalized, but the proof in the general case is essentially the same as in the normalized case, where the only essential  difference is that the norm of the $k$-th coordinate functional is bounded above by the constant which appears there divided by the norm of the $k$-th basis vector). 

Since $I_j$ is finite, there is $i_0\in\N$ large enough such that $|x_{i}(k)-x_{\infty}(k)|<\delta_j$ for all $k\in I_j$ and all $i\geq i_0$. This implies that for each $i\geq i_0$, the point $z_i:=\Pi_j(x_i)$, namely, $\sum_{k\in I_j}x_{i}(k)e_k$, is in $C_j\cap \Pi_j(U)$. Since $C_j\subseteq C'_j$, we conclude that $z_i\in C'_j\cap \Pi_j(U)$ for all $i\geq i_0$. By letting $z:=z_i$ in \beqref{eq:h_j<epsilon} and using the equality $h_j(x_i)=\tilde{h}_j(\Pi_j(x_i))=\tilde{h}_j(z_i)$ (this equality is just an immediate consequence of the definition of $z_i$ and the assumptions on $h_j$ and $\tilde{h}_j$ for every $j\in\N$), it follows  that  $|h_j(x_i)-h_j(x_{\infty})|=|\tilde{h}_j(z_i)-\tilde{h}_j(\Pi_j(x_{\infty}))|<\epsilon$ for each $i\geq i_0$. In other words,  $\lim_{i\to\infty}h_j(x_i)=h_j(x_{\infty})$, as required.
\end{proof}

\begin{remark}\label{rem:SchauderBasis}
As a result of Proposition \bref{prop:weak-to-weak*}, it is of interest to provide some examples of Banach spaces which have the component-* property. Immediate examples are all finite-dimensional Banach spaces. Below we provide an infinite-dimensional example. More precisely, we claim that any Banach space $(X,\|\cdot\|)$ which is isomorphic to $(\ell_p,\|\cdot\|_p)$  for some $p\in (1,\infty)$ has the component-* property. To see this, we first recall  the well-known fact that if $(Y_1,\|\cdot\|_{Y_1})$ and $(Y_2,\|\cdot\|_{Y_2})$ are isomorphic Banach spaces, that is, there is a continuous and invertible linear operator $A:Y_1\to Y_2$, then $A^{-1}$ is continuous too, as a consequence of the open mapping theorem, and their duals $(Y_1^*,\|\cdot\|_{Y_1^*})$ and $(Y_2,\|\cdot\|_{Y_2^*})$ are isomorphic too via the adjoint operator $A^*:Y_2^*\to Y_1^*$: see, for instance, \cite[pp. 478-479]{DunfordSchwartz1958book}.

 Now, if $(Y_1, \|\cdot\|)_{Y_1}$ has a Schauder basis $(e_k)_{k\in\N}$, then a simple verification shows that $(Ae_k)_{k\in\N}$ is a Schauder basis in $(Y_2,\|\cdot\|_{Y_2})$, and moreover, the coordinates of $z\in Y_2$ with respect to $(Ae_k)_{k\in\N}$ are the same as the coordinates of $A^{-1}z$ with respect to $(e_k)_{k\in\N}$.  This implies that a sequence  $(z_i)_{i\in\N}$ in $Y_2$ converges component-wise, i.e., $\lim_{i\to\infty}z_i(k)$ exists for each $k\in\N$, if and only if $(A^{-1}z_i)_{i\in\N}$ converges component-wise in $Y_1$. Furthermore, $(z_i)_{i\in\N}$ is bounded in $Y_2$ if and only if $(A^{-1}z_i)_{i\in\N}$ is  bounded in $Y_1$. In addition, $(z_i)_{i\in\N}$ converges weakly (to, say, $z$) if and only if $(A^{-1}z_i)_{i\in\N}$ converges weakly (to $A^{-1}z$). (Indeed, if $\lim_{i\to\infty}\langle f,z_i\rangle=\langle f,z\rangle$ for each $f\in Y_2^*$, then given $g\in Y_1^*$, we have, using the property of the adjoint operator, that $\langle g,A^{-1}z_i\rangle=\langle (A^{-1})^*g, z_i\rangle\xrightarrow[i\to\infty]{}\langle (A^{-1})^*g,z\rangle=\langle g, A^{-1}z\rangle$, as required.) As a result, if $Y_1$ has the property that  each sequence in it which is bounded and converges component-wise also converges weakly, then $Y_2$ has this property too. This implication holds for any isomorphic Banach spaces. 

It is a known fact that each of the $\ell_p$ spaces, $1<p<\infty$, has the property that each  sequence in it which is bounded and converges component-wise also converges weakly \cite[p. 339]{DunfordSchwartz1958book}. Since the dual of $\ell_p$ is isometric to  $\ell_q$, where $q=p/(p-1)$, we conclude  from the previous paragraph that any sequence in $\ell_p^*$ which is bounded and converges component-wise (according to the canonical Schauder basis of $\ell_p^*\cong\ell_q$) also converges weakly. But the weak topology on $\ell_p^*$ coincides  with the weak$^*$ topology on it since $\ell_p$ and its dual are reflexive Banach spaces for each $p\in (1,\infty)$. It follows that $\ell_p$ has the component-* property. Since we assume that $(X,\|\cdot\|)$ is isomorphic to $(\ell_p,\|\cdot\|_{p})$, and we know that  $(X^*,\|\cdot\|_*)$ is isomorphic to $(\ell_p)^*$, we can conclude from previous lines that both $X$ and $X^*$ have Schauder bases and each sequence in $X^*$ which is bounded and converges component-wise also converges in the weak$^*$ topology, as required.  

An additional example of a Banach space which has the component-* property is any space which is isomorphic to $\oplus_{i=1}^m \ell_{p_i}$, where  $p_1\ldots,p_m\in (1,\infty)$ are given, $2\leq m\in\N$, and the norm on the direct sum $\oplus_{i=1}^m \ell_{p_i}$ is, say, Euclidean, or, more generally, an $\ell_p$ norm, $p\in (1,\infty)$. 
\end{remark}

\section{The negative Boltzmann-Gibbs-Shannon entropy (the Wiener entropy)}\label{sec:GibbsShannon}
Starting from this section we discuss various examples (new or old) of concrete Bregman functions and divergences, and investigate their properties. In this section we focus on the  negative Boltzmann-Gibbs-Shannon entropy. 
\subsection{Background} 
Let $X:=\R^n$, $n\in\N$, with an arbitrary norm $\|\cdot\|$.  Consider the positive orthant $U:=(0,\infty)^n$ and let $b:X\to(-\infty,\infty]$ be the negative Boltzmann-Gibbs-Shannon entropy function defined by 
\begin{equation}\label{eq:Shannon}
b(x):=\left\{\begin{array}{l}
\displaystyle{\sum_{k=1}^n}x_k\log(x_k),\quad x=(x_k)_{k=1}^n\in \cl{U},\\
\infty, \quad \textnormal{otherwise}.
\end{array}
\right.
\end{equation}
Here $\log$ is the natural logarithm ($b$ has similar properties also if $\log$ is any other logarithm) and $0\log(0):=0$. The induced Bregman divergence $B:X^2\to (-\infty,\infty]$ is the Kullback-Leibler divergence 
\begin{equation}\label{eq:BregmanBGS}
B(x,y)=\left\{\begin{array}{l}
\displaystyle{\sum_{k=1}^n}x_k\left(\log\left(\frac{x_k}{y_k}\right)-1\right)+\sum_{k=1}^n y_k, \quad \forall (x,y)\in \cl{U}\times U,\\
\infty, \quad \textnormal{otherwise}.
\end{array}
\right.
\end{equation}
Following the work of Gibbs in the 19th century \cite{Gibbs1874-1878jour} (which is based on earlier works of Boltzmann), the negative of $b$ is called ``the Gibbs entropy'' or ``the Boltzmann-Gibbs entropy'' in statistical mechanics and thermodynamics. Following the 1948 paper \cite{Shannon1948jour} of Shannon, $-b$ is also called ``the Shannon entropy'' in information theory. It appears in  numerous places in the literature, where it sometimes takes the form $-K\sum_{k=1}^n x_k\log(x_k)$ for some positive constant $K$ (which is frequently normalized to 1 or to $1/\log(2)$). A continuous version of $b$ itself,  (not of $-b$), namely where the sum is replaced by an integral, was introduced independently in the 1948 book of Wiener \cite{Wiener1948book}, and hence a possible name for $b$ can be ``the Wiener entropy''. 

In the context of Bregman divergences, $b$ appears already in the original work of Bregman \cite{Bregman1967jour} (in the form given in \beqref{eq:Shannon}), is frequently called ``the negative Shannon entropy'' or ``the information entropy'' or the ``$x\log(x)$ entropy'' or the ``kernel entropy'', and sometimes it has the slightly modified form $b(x)=\sum_{k=1}^n x_k(\log(x_k)-1)$ (all the properties of $b$ mentioned above and below remain the same despite this linear deformation). Perhaps the first place in which many of its Bregmanian properties have been proved formally is  \cite[Lemma 5]{CensorDePierroElfvingHermanIusem1990incol}. A continuous version of the divergence $B$ (where the sum is replaced by an integral) appears in \cite{KullbackLeibler1951jour} in the context of statistics and information theory, where it is also assumed there that both $x$ and  $y$ are probability density vectors, namely they are positive and their integrals over the measure space are equal to 1 (of course, in the discrete case the latter assumption means that $\sum_{k=1}^nx_k=1=\sum_{k=1}^ny_k$). Kullback and Leibler called their divergence ``the mean information discrimination'' \cite[p. 80]{KullbackLeibler1951jour}. 

As far as we know, both $b$ and $-b$ have been considered in Euclidean  spaces and not in other normed spaces. Below we present some classical properties of $b$ and also shed some new light on it, mainly regarding  strong and relative uniform convexity.

\subsection{Basic properties}\label{subsec:GibbsShannonBasicProperties} 
A simple verification shows that $b'(z)=(\log(z_k)+1)_{k=1}^n$ and also that 
$b''(z)(w,w)=\sum_{k=1}^n(1/z_k)w_k^2$ for every $z\in U$ and every vector $w\in X$. In particular, $b$ is essentially smooth (this is clear if the norm is Euclidean and hence true for our arbitrary norm since all the norms on $\R^n$  are equivalent). Since $b(z)=\sum_{k=1}^nb_k(z_k)$ for all $z\in\dom(b)$, where $b_k:[0,\infty)\to\R$ is defined by $b_k(z_k):=z_k\log(z_k)$ for every $k\in\{1,\ldots,n\}$ and $z_k\in [0,\infty)$, and since $b_k$ is strictly convex on $[0,\infty)$ (strict convexity on $[y_k,z_k]\subset (0,\infty)$ is clear since $b_k''$ is positive there; for $y_k:=0$ and $z_k>0$ one observes that $b_k(\lambda z_k+(1-\lambda)y_k)<\lambda b_k(z_k)$ whenever $\lambda\in (0,1)$ since $\log$ is strictly increasing on $(0,\infty)$), it follows that $b$ is strictly convex on $\dom(b)$ (in fact, as shown in Subsection \bref{subsec:GibbsShannonStronglyConvex} below, $b$ is strongly convex on nonempty, bounded and convex subsets of $\dom(b)$). Hence $b$ is Legendre.  In addition, both derivatives  of $b$ are  continuous on $U$ and $b$ is continuous on $\dom(b)$.  Since $\dom(b)=\cl{U}$ is closed and $b$ is continuous on $\cl{U}$, a simple verification shows that for all $\gamma\in\R$, the $\gamma$-level set $\{x\in X: b(x)\leq\gamma\}$ of $b$ coincides with $\{x\in \cl{U}: b(x)\leq\gamma\}$ and is closed. Therefore $b$ is lower semicontinuous on $X$. 

\subsection{Strong convexity}\label{subsec:GibbsShannonStronglyConvex} 
Since all the norms on a finite-dimensional space are equivalent, there are $c_2>0$ and $c_{\infty}>0$ such that 
\begin{equation}\label{eq:EquivalentNorms}
c_2\|v\|\leq \|v\|_2\,\,\textnormal{and}\,\, 
\|v\|_{\infty}\leq c_{\infty}\|v\|\quad \forall v\in X, 
\end{equation}
where $\|v\|_2:=\sqrt{\sum_{i=1}^n v_i^2}$ and $\|v\|_{\infty}:=\max\{|v_i|: i\in \{1,\ldots,n\}\}$. Given $x,y\in U$, $x\neq y$, let $z=(z_i)_{i=1}^n\in [x,y]$. In particular, all the components of $x,y,z$ are positive and for each $i\in \{1,\ldots,n\}$, either $x_i\leq z_i\leq y_i$ or $y_i\leq z_i\leq x_i$. Thus $\|z\|_{\infty}\leq \max\{\|x\|_{\infty},\|y\|_{\infty}\}$.  The above facts and \beqref{eq:EquivalentNorms} imply that for every unit vector $w$, one has
\begin{multline}\label{eq:W}
b''(z)(w,w)=\sum_{i=1}^n \frac{w_i^2}{z_i}\geq \sum_{i=1}^n \frac{w_i^2}{\|z\|_{\infty}}=\frac{\|w\|_2^2}{\|z\|_{\infty}}\geq 
\frac{c_2^2\|w\|^2}{\max\{\|x\|_{\infty},\|y\|_{\infty}\}}\\
=\frac{c_2^2}{\max\{\|x\|_{\infty},\|y\|_{\infty}\}}\geq \frac{c_2^2}{c_{\infty}\max\{\|x\|,\|y\|\}}. 
\end{multline}
This is true for each $z\in [x,y]$. It follows from Proposition \bref{prop:BregmanStronglyConvex}\beqref{item:eta_xy} that  $\mu_{1}[x,y]:=c_2^2/\max\{\|x\|_{\infty},\|y\|_{\infty}\}$ is a strong convexity parameter of $b$ on $[x,y]$. A smaller but more convenient parameter of strong convexity is $\mu_2[x,y]:=c_2^2/(c_{\infty}\max\{\|x\|,\|y\|\})$, as follows again from \beqref{eq:W}.

 Now let  $S$ be an arbitrary nonempty, bounded and convex subset of  $U$. Then there is $M_S>0$ such that $\|s\|\leq M_S$ for all $s\in S$.  From Proposition  \bref{prop:BregmanStronglyConvex}\beqref{item:eta_S} and the previous paragraphs we conclude that $b$ is  strongly convex  on $S$ with $\mu[S]:=c_2^2/(c_{\infty}M_S)$ as a strong convexity parameter. From Proposition  \bref{prop:BregmanStronglyConvex}\beqref{item:StronglyConvex-clU} it follows that $b$ is strongly convex on all nonempty, convex and bounded subsets of $\cl{U}$ (in particular, this shows in a different way that $b$ is strictly convex on $U$). 
 
\subsection{Relative uniform convexity}\label{subsec:GibbsShannonRelativeUC}
Now fix an arbitrary $x\in \cl{U}$. We show below that $b$ is uniformly convex relative to the pair $(\{x\},U\cap \{w\in X: \|w\|>2\|x\|\})$ with 
\begin{equation}
\psi(t):=\frac{c_2^2}{4c_{\infty}}t, \quad t\in [0,\infty),  
\end{equation}
as a relative gauge. Indeed, given $y\in U$  satisfying $\|y\|>2\|x\|$, the triangle inequality $\|y\|-\|x\|\leq \|x-y\|$ and the above lines imply that 
\begin{equation*}
\psi(\|x-y\|)=\frac{c_2^2\|x-y\|}{4c_{\infty}}< \frac{c_2^2\|x-y\|}{2c_{\infty}}\cdot\left(1-\frac{\|x\|}{\|y\|}\right)
\leq \frac{c_2^2\|x-y\|^2}{2c_{\infty}\|y\|}=\frac{\mu_2[x,y]}{2}\|x-y\|^2.
\end{equation*}
Since we already know from previous paragraphs that $b$ is strongly convex on $[x,y]$ with $\mu_2[x,y]$ as a strong convexity parameter (see Subsection \bref{subsec:GibbsShannonStronglyConvex}), we draw the desired conclusion from Remark  \bref{rem:TypesOfConvexity}\beqref{item:mu[x,y]}.

\subsection{No global uniform convexity}\label{subsec:BGS-NoUniformlyConvex}
We show below that $b$ cannot be uniformly convex on $U$, and hence also on $\dom(b)$ (the case $n=1$ is stated without a proof in \cite[p. 186]{BauschkeCombettes2017book}). Indeed, assume to the contrary that $b$ is uniformly convex on $U$. Given $s>1$, let $x(s)\in U$ and $y(s)\in U$ be defined by $x_1(s):=s$, $y_1(s):=s+1$, and $x_i(s):=1=:y_i(s)$ for all $i\in\{1,\ldots,n\}\backslash\{1\}$.  Since all the norms on $\R^n$ are equivalent, there is $\eta>0$ such that $\|z\|\geq \eta\|z\|_2$ for each $z\in \R^n$, where $\|\cdot\|_2$ is the Euclidean norm. Therefore $\|x(s)-y(s)\|\geq \eta\|x(s)-y(s)\|_2=\eta$. Consequently, Lemma \bref{lem:MonotonePsi} and the fact that $b$ is uniformly convex imply that $\psi_{b,U}(\|x(s)-y(s)\|)\geq\psi_{b,U}(\eta)>0$ for each $s>1$, where  $\psi_{b,U}$ is the modulus of uniform convexity of $b$ on $U$. On the other hand, from Proposition \bref{prop:BregmanProperties}\beqref{BregProp:BregPsiS1S2} and \beqref{eq:BregmanBGS}, together with the well-known relation $\log(1+t)=t+o(t)$, which holds for all $t\in (-1,1)$, we have  
\begin{multline}\label{eq:h(lambda)}
\psi_{b,U}(\|x(s)-y(s)\|)\leq B(x(s),y(s))=x_1(s)\left(\log\left(\frac{x_1(s)}{y_1(s)}\right)-1\right)+y_1(s)\\
=1+s\log\left(\frac{s}{s+1}\right)=s\left(\frac{1}{s}-\log\left(1+\frac{1}{s}\right)\right)=\frac{o\left(\frac{1}{s}\right)}{\frac{1}{s}}{\xrightarrow[s\to \infty]{}}\,\,0.
\end{multline}
In particular,  $\psi_{b,U}(\|x(s)-y(s)\|)<\psi_{b,U}(\eta)$ for $s$ sufficiently large.  Thus we arrive at a contradiction. This shows that $b$ cannot be uniformly convex on $U$. 

\subsection{$b$ is a Bregman function}\label{subsec:BregmanGibbsShannon}
We finish this section by re-establishing the well-known fact that $b$ is a Bregman function. In fact, we show that $b$ has the limiting difference property, that it is sequentially consistent, and that the second type level-sets of $B$ are bounded. Indeed, the latter claim is immediate from Proposition \bref{prop:BregmanProperties}\beqref{BregProp:LevelSet2BoundedSufficientConditions} since $b$ is strictly convex on $\dom(b)$. Definition \bref{def:BregmanDiv}\beqref{BregmanDef:U} and \beqref{BregmanDef:ConvexLSC} are satisfied by the assumptions on $b$ and by what we proved in Subsections \bref{subsec:GibbsShannonBasicProperties}--\bref{subsec:GibbsShannonStronglyConvex}. Definition \bref{def:BregmanDiv}\beqref{BregmanDef:B}  is just the definition of the Bregman divergence $B$. Definition \bref{def:BregmanDiv}\beqref{BregmanDef:BoundedLevelSet} is a consequence of Proposition \bref{prop:BregmanProperties}\beqref{BregProp:LevelSetBounded} and Subsection \bref{subsec:GibbsShannonRelativeUC}. Subsection \bref{subsec:GibbsShannonStronglyConvex} and  Proposition \bref{prop:BregmanProperties}\beqref{BregProp:B(x_i,y_i)=0AndOneSequenceBounded==>|x_i-y_i|=0} show that $b$ is sequentially consistent and also imply that Definition \bref{def:BregmanDiv}\beqref{BregmanDef:B(x_i,y_i)=0==>|x_i-y|=0} holds. The limiting difference property follows from Proposition \bref{prop:BregmanProperties}\beqref{BregProp:B(x,y)=B(x,y_i)-B(y,y_i)} since the space is finite-dimensional. 

Finally, to see that Definition \bref{def:BregmanDiv}\beqref{BregmanDef:|x-y_i|=0==>B(x,y_i)=0} holds, we show that $\lim_{i\to\infty}\langle b'(y_i),x-y_i\rangle=0$ and then use Proposition \bref{prop:BregmanProperties}\beqref{BregProp:D(x,y_i)=0dom(b)}. Given $x=(x_k)_{k=1}^n\in \dom(b)$ and a sequence $(y_i)_{i=1}^{\infty}$ in $U$  which converges to $x$, we have $\lim_{i\to\infty}y_{i,k}=x_k$ for all $k\in\{1,\ldots,n\}$, where $y_{i,k}$ is the $k$-th component of $y_i$. We also have $y_{i,k}>0$ for each $i\in\N$ and $k\in\{1,\ldots,n\}$. Now fix some $k\in\{1,\ldots,n\}$ and consider $x_k$. If $x_k>0$, then the continuity of the function $t\mapsto \log(t)(x_k-t)$  on  $(0,\infty)$ and the fact that it vanishes at $t=x_k$ implies that $\lim_{i\to\infty} \log(y_{i,k})(x_k-y_{i,k})=0$. If $x_k=0$, then from the known limit $\lim_{t\to 0+}\log(t)t=0$  we get $\lim_{i\to\infty} \log(y_{i,k})(x_k-y_{i,k})=0$ again. Since  $\langle b'(y_i),x-y_i\rangle=\sum_{k=1}^n \log(y_{i,k})(x_k-y_{i,k})+\sum_{k=1}^n(x_k-y_{i,k})$, we conclude from the previous lines that indeed $\lim_{i\to\infty}\langle b'(y_i),x-y_i\rangle=0$.

\section{The negative Havrda-Charv\'at-Tsallis entropy}\label{sec:HavrdaCharvatTsallis}
\subsection{Background}
Let $X:=\R^n$ ($n\in\N$) with an arbitrary norm $\|\cdot\|$ and let $U:=(0,\infty)^n$.  Fix $q\in \R\backslash\{0,1\}$.  Define the negative Havrda-Charv\'at-Tsallis entropy by
\begin{equation}\label{eq:Havrda-Charvat-Tsallis}
b(x):=\left\{\begin{array}{lll}
\displaystyle{\frac{1}{q-1}}\displaystyle{\sum_{k=1}^n (x_k^{q}-1)}, & x=(x_k)_{k=1}^n\in \cl{U}\,\,\textnormal{and}\,\,q\in (0,1)\cup(1,\infty),\\
\displaystyle{\frac{1}{1-q}}\displaystyle{\sum_{k=1}^n (x_k^{q}-1)}, & x\in  U\,\,\textnormal{and}\,\,q<0,\\
\infty,& \textnormal{otherwise}.
\end{array}
\right.
\end{equation}

The corresponding (pre-)Bregman divergence is 
\begin{equation}\label{eq:Bregman-Havrda-Charvat-Tsallis}
B(x,y)=\left\{\begin{array}{lll}
\displaystyle{\frac{1}{q-1}}\sum_{k=1}^n \left(x_k^{q}-y_k^{q}-q y_k^{q-1}(x_k-y_k)\right),& (x,y)\in \cl{U}\times U\\
& \textnormal{and}\,\,q\in(0,1)\cup(1,\infty),\\
\displaystyle{\frac{1}{1-q}}\sum_{k=1}^n \left(x_k^{q}-y_k^{q}-q y_k^{q-1}(x_k-y_k)\right),& (x,y)\in U\times U\,\,\textnormal{and}\,\,q<0,\\
\infty,&\textnormal{otherwise}.
\end{array}
\right.
\end{equation}
This entropy can be thought of as being a certain generalization of the negative Boltzmann-Gibbs-Shannon entropy  because \beqref{eq:Shannon} is obtained from \beqref{eq:Havrda-Charvat-Tsallis} in the limit $q\to 1$. The negative of $b$ was introduced by Havrda and Charv\'at \cite[Theorem 1]{HavrdaCharvat1967jour} in the context of information theory (the coefficient of the sum was $2^{q-1}/(2^{q-1}-1)$ instead of $1/(q-1)$, and $q$ was assumed to be positive) and was rediscovered by Tsallis \cite{Tsallis1988jour} in the context of statistical mechanics. The original entropy (that is, $-b$) has applications in various areas of science and engineering , among them thermostatistics \cite{BPT1998jour}, astrophysics \cite{PickupCywinskiPappasFaragoFouquet2009jour},  sensor networks \cite{GaoLiu2014jour}, medical signal processing plasma \cite{LiZhou2016jour}, quantum mechanics \cite{AbeOkamoto2001book}, complex systems \cite{Gell-MannTsallis2004book}, image processing \cite{ManicPriyaRajinikanth2016jour}, just to name a few (we note, however, that in these applications the setting is usually the $n$-dimensional Euclidean space, and $x=(x_i)_{i=1}^n$ is a positive probability vector, namely, $x_i>0$ for each $i$ and $\sum_{i=1}^n x_i=1$; sometimes a continuous version of the entropy is considered, in which an integral replaces the discrete sum). Many more details and applications can be found in the book of Tsallis \cite{Tsallis2009book}. As a matter of fact, it seems that since the pioneering work of Tsallis \cite{Tsallis1988jour} there has been a huge amount of research related to this entropy: indeed, in \cite{TsallisContinuouslyUpdatedList} one can find a continuously updated list of works which are directly related to this entropy, and as of October 2018 this (not exhaustive) list contains not less than 6913 items(!).

However, in the context of Bregman functions, the function $b$ from \beqref{eq:Havrda-Charvat-Tsallis} is rarely considered (but see Subsection \bref{subsec:alpha-beta} below for a somewhat related variation of $b$). In fact, we have seen a very brief (and somewhat implicit) related discussion only in \cite[Appendices A, C]{CichockiAmari2010jour} , \cite[p. 2046]{Csiszar1991jour} (for the case $q<1$), and in \cite[p. 129]{Lafferty1999inproc} (for the case $q\in (0,1)$), where in all of these cases no proofs were given that $b$ satisfies Definition \bref{def:BregmanDiv}; see also \cite[p. 1566]{CichockiAmari2010jour} for a continuous analogue  of $B$, namely when the sum is replaced by an integral. In the convex analysis and optimization literature one can see, in a few places, the closely related versions of \beqref{eq:Havrda-Charvat-Tsallis} defined by   $b_1(x):=\sum_{k=1}^n x_k^{-\delta}$  for fixed $\delta>0$ (see \cite[p. 340]{Cruz-NetoFerreiraIusemMonteiro2007jour}; see also \cite[Example 1, item 3]{Nguyen2017jour} for the slight variation $\tilde{b}_1(x):=(1/p)\sum_{k=1}^n x_k^{-p}$, $p\in [1,\infty)$), $b_2(x):=\sum_{k=1}^n((-1/p)x_k^p+ax_k)$ where $p\in(0,1)$ and $a\in [0,\infty)$ are given (see \cite[Example 2.2, p. 322]{BregmanCensorReich1999jour} for $a>0$ and \cite[Examples 6.3, 7.8]{BauschkeBorwein1997jour} for $a=0$), and also $b_3(x):=(1/(1-\alpha))\sum_{k=1}^n(\alpha x_k-x_k^{\alpha})$ where $\alpha\in(0,1)$ is given (see \cite[Examples 3.1(3), p. 679]{Teboulle1992jour}, \cite[Example 1]{BauschkeBolteTeboulle2017jour}); in all of these cases the interior of the effective domain of the above-mentioned functions is $(0,\infty)^n$. With the exception of  \cite[Examples 6.3, 7.8]{BauschkeBorwein1997jour}, the discussions in all of these cases are very brief and no proof is given that these functions are indeed Bregman functions (in  \cite[Example 6.3, Corollary 5.13]{BauschkeBorwein1997jour} it is shown that $b_2$ is Bregman/Legendre; the notion of a ``Bregman/Legendre function'', which was introduced in \cite{BauschkeBorwein1997jour} and was discussed there thoroughly, is closely related to, but somewhat different from, the notion of a ``Bregman function''). Actually, as we  prove in Subsection \bref{subsubsec:q<0} below, the function $b_1$ mentioned above does not satisfy Definition \bref{def:BregmanDiv} when $\delta>0$. In general, it seems that these variations of \beqref{eq:Havrda-Charvat-Tsallis} are not very well known.

\subsection{Basic properties}\label{subsec:HCT-BasicProperties} For all $q\in \R\backslash\{0,1\}$, $z\in U$ and $w\in X$, one has  $b'(z)=((|q|/(q-1))z_k^{q-1})_{k=1}^n$ and 
$b''(z)(w,w)=\sum_{k=1}^n |q| z_k^{q-2}w_k^2$. Hence, if $q<1$, then $b$ is essentially smooth.  In addition, both derivatives  of $b$ are  continuous on $U$ and $b$ is continuous on $\dom(b)$. A similar reasoning to the one mentioned in Subsection \bref{subsec:GibbsShannonBasicProperties} implies that $b$ is strictly convex on $\dom(b)$ (with the exception of the inequality $b_k(\lambda z_k+(1-\lambda)y_k)<\lambda b_k(z_k)+(1-\lambda)b(y_k)$ when $y_k=0$ and $\lambda\in (0,1)$ and $z_k>0$, but this inequality can easily be shown directly by separating into the cases $1<q$ and $q<1$; in Subsection \bref{subsec:HCT-StrongConv} below we show the stronger result that if $q<2$, then $b$ is even strongly convex on any nonempty, bounded and convex subset of $\dom(b)$). Thus if $q<1$, then $b$ is Legendre. 

When $q>0$, the continuity and convexity of $b$ on $\dom(b)=\cl{U}$ imply that $b$ is lower semicontinuous and convex on $X$ (the explanation is similar to the one given in Subsection \bref{subsec:GibbsShannonBasicProperties}). When $q<0$ the (strict) convexity of $b$ on $\dom(b)=U$ implies that $b$ is convex on $X$. To see that $b$ is also lower semicontinuous on $X$ when $q<0$, one can observe that $\lim_{i\to\infty}b(y_i)=\infty$ whenever $(y_i)_{i=1}^{\infty}$ is a sequence in $U$ having the property that  the distance between $y_i$ and the boundary of $U$ tends to zero as $i$ tends to infinity. This fact implies that for each $\gamma\in\R$, the $\gamma$-level set $\{x\in X: b(x)\leq \gamma\}$ of $b$ must be strictly inside $U$ (that is, its distance to the boundary of $U$ is positive) whenever it is nonempty. This observation and the continuity of $b$ on $U$ imply that each such level set is a closed subset of $X$ and hence $b$ is indeed lower semicontinuous on $X$.

\subsection{Strong convexity}\label{subsec:HCT-StrongConv} 
Assume first that $q=2$. From the previous paragraphs and \beqref{eq:EquivalentNorms} we have  $b''(z)(w,w)=2\|w\|_2^2\geq 2c_2^2$ for each unit vector $w$ and each $z\in U$. Hence from Proposition  \bref{prop:BregmanStronglyConvex}\beqref{item:eta_S} it follows that $b$ is strongly convex  on $\cl{U}$ with $2c_2^2$ as a strong convexity parameter. 

Now suppose that $q\in (-\infty,0)\cup(0,1)\cup(1,2)$ and fix $x,y\in U$, $x\neq y$. We can follow word for word the analysis in Subsection \bref{subsec:GibbsShannonStronglyConvex} (that is, \beqref{eq:W} and the discussion before and after it) to conclude that both 
$\mu_{1}[x,y]:=|q|c_2^2/\max\{\|x\|^{2-q}_{\infty},\|y\|^{2-q}_{\infty}\}$ and  $\mu_{2}[x,y]:=|q|c_2^2/(c_{\infty}^{2-q}\max\{\|x\|^{2-q},\|y\|^{2-q}\})$ are strong convexity parameters of $b$ on $[x,y]$. This shows that $b$ is strongly convex on any nonempty, bounded and convex subset $S$ of $U$ with $\mu[S]:=|q|c_2^2/(c_{\infty} M_S)^{2-q}$ as a strong convexity parameter of $b$ on $S$, where $M_S$ is an upper bound on the norm of vectors from $S$. Hence, from Proposition \bref{prop:BregmanStronglyConvex}\beqref{item:StronglyConvex-clU} it follows that $b$ is strongly convex on any nonempty, bounded and convex subset $S\subseteq\dom(b)$. 

Finally, assume that $q>2$. In this case it is not true that $b$ is strongly convex on all nonempty, bounded and convex subsets of $U$ (see Subsection \bref{subsubsec:NoStrongConv q>2} below), but it is true that $b$ is strongly convex on any nonempty subset $V$ of $U$ the distance of which to the boundary  of $U$ is positive. Indeed, given such a subset $V$, there exists some $\epsilon>0$ such that $z_k\geq \epsilon$ for each $k\in\{1,\ldots,n\}$ and $z=(z_k)_{k=1}^n\in V$. Because 
$q-2>0$ it follows that $b''(z)(w,w)\geq q \epsilon^{q-2}\|w\|_2^2$. Taking into account \beqref{eq:EquivalentNorms} and Proposition  \bref{prop:BregmanStronglyConvex}\beqref{item:eta_S}, we see that $b$ is strongly convex on $V$ with $\mu:=q c_2^2 \epsilon^{q-2}$ as a parameter of strong convexity. 

\subsection{Relative uniform convexity}\label{subsec:HCT-RelativeUniformConv} Because of Proposition \bref{prop:BregmanProperties}\beqref{BregProp:LevelSetBounded} we are interested in uniform convexity relative  to pairs of the form $(\{x\},S_2)$, where $x\in\dom(b)$ and $S_2:=\{w\in X: \|w\|>r_x\}$ for some $r_x>0$. For $q=2$, we already know that $b$ is strongly convex on $\dom(b)$ and hence it is uniformly convex relative to these pairs. For $q>2$, the situation is not clear, but at least in the case where both $x$ and $S_2$ are contained in $U_{\epsilon}:=[\epsilon,\infty)^n$ for some $\epsilon>0$ we know that $b$ is uniformly convex  relative to $(\{x\},S_2(x))$ since we actually know from Subsection  \bref{subsec:HCT-StrongConv} that $b$  is strongly convex on $U_{\epsilon}$. 

Now consider the case $q\in (-\infty,0)\cup(0,1)\cup(1,2)$. We claim that for each $x\in \dom(b)$, the function $b$ is uniformly  convex relative to $(\{x\},U\cap \{w\in X: \|w\|>2\|x\|\})$ with the following  relative gauge: 
\begin{equation}\label{eq:psi_q<2}
\psi(t):=\left\{
\begin{array}{lll}
\displaystyle{\frac{|q|c_2^2}{c_{\infty}^{2-q}2^{3-q}}}t^{q}, & \,\,t\in (0,\infty),\vspace*{0.2cm}\\
0, & t=0.
\end{array}
\right.
\end{equation}
Indeed, fix $x\in \dom(b)$. Given $y\in U$  satisfying $\|y\|>2\|x\|$, the triangle inequality $\|y\|-\|x\|\leq \|x-y\|$ and the previous lines show that 
\begin{multline*}
\psi(\|x-y\|)=\frac{|q|c_2^2\|x-y\|^{q}}{2c_{\infty}^{2-q}}\cdot\frac{1}{2^{2-q}}\leq \frac{|q|c_2^2\|x-y\|^{q}}{2c_{\infty}^{2-q}}\cdot\left(1-\frac{\|x\|}{\|y\|}\right)^{2-q}\\
\leq \frac{|q|c_2^2\|x-y\|^{q}}{2c_{\infty}^{2-q}}\frac{\|x-y\|^{2-q}}{\|y\|^{2-q}}
=\frac{\mu_2[x,y]}{2}\|x-y\|^2.
\end{multline*}
Since we already know from Subsection \bref{subsec:HCT-StrongConv} that $b$ is strongly convex on $[x,y]$ with $\mu_2[x,y]$ as a strong convexity parameter, we obtain the desired conclusion from Remark  \bref{rem:TypesOfConvexity}\beqref{item:mu[x,y]}.

\subsection{No global strong convexity}\label{subsec:NoStrongConv q>0} 
In this subsection we show that $b$ is not strongly convex on $U$ when $q\in (0,\infty)\backslash\{1,2\}$. The analysis is separated into cases, according to the possible values of $q$. 

\subsubsection{\bf $q>2$}\label{subsubsec:NoStrongConv q>2} 
The proof is by way of contradiction and it actually shows that $b$ cannot be strongly convex on certain  bounded and convex subsets of $U$ (on intervals of the form $[x,y)$, where $x_i=y_i=1$ for all $i\in\{1,\ldots,n-1\}$, $x_n=a>0$, $y_n=0$). Because of  Proposition \bref{prop:StronglyConvexSum} (with $I=\{1\}$) it is sufficient to assume that $\|\cdot\|=\|\cdot\|_2$, that is, that the norm is the Euclidean norm. So assume that $b$ is strongly convex on $U$. Then its modulus of strong  convexity $\mu[U]$ is positive. In particular, if we fix $\epsilon>0$, then $\mu[U]$ is a strong convexity parameter of $b$ on all the line segments of the form $[x,y]$ for $x:=(x_i)_{i=1}^n\in U$ and $y:=(y_i)_{i=1}^n\in U$ such that  $x_1=2\epsilon$, $y_1=\epsilon$ and $x_i=y_i=1$ for each $i\in \{1,\ldots,n\}\backslash\{1\}$. Proposition \bref{prop:BregmanProperties}\beqref{BregProp:BregPsiS1S2} and \beqref{eq:Bregman-Havrda-Charvat-Tsallis} yield 
\begin{equation}\label{eq:frac_epsilon}
\frac{1}{2}\mu[U]\|x-y\|^2\leq B(x,y)=\frac{x_1^{q}-y_1^{q}-q y_1^{q-1}(x_1-y_1)}{q-1}.
\end{equation} 
 We conclude that $\mu[U]\leq 2(2^{q}-1-q)\epsilon^{q-2}/(q-1)$. By taking into account the derivation and the fact that $q>2$, it follows that the right-hand side of \beqref{eq:frac_epsilon} tends to zero as $\epsilon$ tends to zero (the right-hand side is indeed positive because, as follows from elementary calculus, $2^{q}-q-1>0$ whenever $q>2$). Therefore $\mu[U]\leq 0$, a contradiction to the assumption that $\mu[U]>0$. Hence $b$ is not strongly convex on $U$. 

\subsubsection{\bf $1\neq q\in (0,2)$}\label{subsubsec:NoStrongConv 0<q<2} Suppose now that $q\in (0,2)$, $q\neq 1$, and assume to the contrary that $b$ is  strongly convex on $U$, namely that it has a strong convexity parameter $\mu[U]>0$ on $U$.  In particular, $\mu[U]$ is a strong convexity parameter of $b$ on all the line segments of the form $[x,y]$ for $x:=(1,1\ldots,1)\in U$ and $y=(y_i)_{i=1}^n\in U$ such that $y_1\in(1,\infty)$ and $y_i=1$ for each $i\in \{1,\ldots,n\}\backslash\{1\}$. As before, we can assume that the norm is Euclidean.  From \beqref{eq:Bregman-Havrda-Charvat-Tsallis}, the choice of $x$ and $y$ and Proposition \bref{prop:BregmanProperties}\beqref{BregProp:BregPsiS1S2}, we have
\begin{equation*}
\frac{1}{2}\mu[U]\|x-y\|^2\leq B(x,y)=\frac{1-y_1^{q}-q y_1^{q-1}(1-y_1)}{q-1}.
\end{equation*}
We conclude that $\mu[U]\leq (1-y_1^{q}-q y_1^{q-1}(1-y_1))/(0.5(q-1)(1-y_1)^2)$. Since $q<2$, the right-hand side of this inequality tends to zero as $y_1$ tends to  infinity. Thus $\mu[U]\leq 0$, a contradiction  to the assumption that $\mu[U]>0$. This contradiction shows that $b$ is not strongly convex on $U$. 

\subsection{No global uniform convexity when $q=\frac{1}{2}$}\label{subsec:NoUniformConv q=0.5} We show below that if $q=\frac{1}{2}$, then $b$ cannot even be uniformly convex on $U$. 
Indeed, assume to the contrary that $b$ is uniformly convex on $U$. Denote by $\psi_{b,U}$ the modulus of uniform convexity of $b$ on $U$. From  Proposition \bref{prop:BregmanProperties}\beqref{BregProp:BregPsiS1S2} we have 
\begin{equation}\label{eq:psiV}
\psi_{b,U}(\|x-y\|)\leq B(x,y),\quad \forall (x,y)\in U^2.
\end{equation}
For each $s\in (0,\infty)$, define $x(s)=(x_i(s))_{i=1}^n$ and $y(s)=(y_i(s))_{i=1}^n$ as follows: 
\begin{subequations}
\begin{equation}
x_1(s):=s+\sqrt{s},\quad y_1(s):=s, 
\end{equation}
\begin{equation}
x_i(s):=1=:y_i(s),\quad \forall\, i\in \{1,\ldots,n\}\backslash\{1\}. 
\end{equation}
\end{subequations}
Both $x(s)$ and $y(s)$ belong to $U$, and after substituting them in \beqref{eq:psiV} instead of $(x,y)$,  using \beqref{eq:Bregman-Havrda-Charvat-Tsallis} and making simple manipulations, we arrive at 
\begin{multline}\label{eq:x(s)y(s)}
\psi_{b,U}(\|x(s)-y(s)\|)\leq B(x(s),y(s))=-2\sqrt{x_1}+2\sqrt{y_1}+\frac{x_1-y_1}{\sqrt{y_1}}\\
=\frac{(y_1-x_1)(\sqrt{y_1}-\sqrt{x_1})}{(\sqrt{y_1}+\sqrt{x_1})\sqrt{y_1}}
=\frac{-\sqrt{s}(\sqrt{s}-\sqrt{s+\sqrt{s})}}{(\sqrt{s}+\sqrt{s+\sqrt{s}})\sqrt{s}}
=\frac{\sqrt{1+s^{-0.5}}-1}{\sqrt{1+s^{-0.5}}+1}{\xrightarrow[s\to \infty]{}}\,\,0.
\end{multline}
On the other hand, since all the norms on $\R^n$ are equivalent, there is $\eta>0$ such that $\|z\|\geq \eta\|z\|_2$ for each $z\in \R^n$, where $\|\cdot\|_2$ is the Euclidean norm. Therefore $\|x(s)-y(s)\|\geq \eta\|x(s)-y(s)\|_2=\eta\sqrt{s}\to\infty$ as $s\to\infty$. Consequently, Lemma \bref{lem:MonotonePsi} implies that $\lim_{s\to\infty}\psi_{b,U}(\|x(s)-y(s)\|)=\infty$. This is  a contradiction to \beqref{eq:x(s)y(s)}. Hence $b$ is not uniformly on $U$, as claimed.

\subsection{When is $b$ a Bregman function?}\label{subsec:HCT-Bregman}
We finish this section by discussing the question when $b$ is a Bregman function. The discussion is separated into cases, according to the possible values of $q$. 
\subsubsection{The case $q=2$:} In this case $b$ is a sequentially consistent Bregman function which has the limiting difference property, as follows from previous subsections and Corollary \bref{cor:BregmanFunctionFiniteDim-b'-is-cont}, which also imply that the second type level-sets of $B$ are bounded.

\subsubsection{The case $q>2$:} In this case $b$ is a sequentially consistent Bregman function which has the limiting difference property and $B$ has bounded level-sets of the second type. Indeed, 
Proposition \bref{prop:BregmanProperties}\beqref{BregProp:LevelSet2BoundedSufficientConditions} and the fact that $b$ is strictly convex on $\dom(b)$ imply that the second type level-sets of $B$ are bounded.  
Parts \beqref{BregmanDef:U} and \beqref{BregmanDef:ConvexLSC} of Definition \bref{def:BregmanDiv} are satisfied by the assumptions on $b$ and previous subsections. Definition \bref{def:BregmanDiv}\beqref{BregmanDef:B} is just the definition of the Bregman divergence $B$, Definition \bref{def:BregmanDiv}\beqref{BregmanDef:|x-y_i|=0==>B(x,y_i)=0} follows from Proposition  \bref{prop:BregmanProperties}\beqref{BregProp:Continuous_b|x_i-y_i|=0==>B(x_i,y_i)=0} since $X$ is finite-dimensional and hence $b'$, which is continuous on $\cl{U}$, is uniformly continuous on bounded subsets of $U$. The limiting difference property follows from Proposition \bref{prop:BregmanProperties}\beqref{BregProp:B(x,y)=B(x,y_i)-B(y,y_i)} since $X$ is finite-dimensional and $b'$ is continuous. It remains to prove that $b$ satisfies Definition \bref{def:BregmanDiv}\beqref{BregmanDef:B(x_i,y_i)=0==>|x_i-y|=0} and  Definition \bref{def:BregmanDiv} \beqref{BregmanDef:BoundedLevelSet}. We note that we cannot continue here as in the proof of Corollary \bref{cor:BregmanFunction-b'-is-unicont} above because of the  limitations mentioned in Subsections \bref{subsec:HCT-StrongConv}-\bref{subsec:HCT-RelativeUniformConv} on the (relative) uniform convexity of $b$ in $U$ when $q>2$.

In order to establish Definition \bref{def:BregmanDiv}\beqref{BregmanDef:B(x_i,y_i)=0==>|x_i-y|=0}, we use ideas from the proof of Proposition \bref{prop:BregmanProperties}\beqref{BregProp:x_iCompact} (we cannot use 
 Proposition \bref{prop:BregmanProperties}\beqref{BregProp:x_iCompact} directly because there the  point $y$ must be in $U$ and in our case it can be in $\dom(b)=\cl{U}$). 
 From \beqref{eq:Havrda-Charvat-Tsallis} and \beqref{eq:Bregman-Havrda-Charvat-Tsallis}, respectively, we have $b(x)=\sum_{k=1}^n b_k(x_k)$ and $B(x,y)=\sum_{k=1}^n B_k(x_k,y_k)$, where  for each $k\in\{1,\ldots,n\}$, the function $b_k:\R\to(-\infty,\infty]$ is defined by $b_k(x_k):=(x_k^q-1)/(q-1)$, $x_k\in [0,\infty)$,  and $b_k(x_k):=\infty$ otherwise, and $B_k(x_k,y_k):=b_k(x_k)-b_k(y_k)-b_k'(y_k)(x_k-y_k)$ for all $(x_k,y_k)\in\R\times [0,\infty)$. A simple verification shows that $b_k$ is twice continuously differentiable in  $[0,\infty)$ (with right-hand derivatives at $x_k=0$) and strictly convex on $[0,\infty)$, and that $B_k$ is continuous on $[0,\infty)^2$ for all $k\in\{1\ldots,n\}$. 
 
 Now let $(x_i)_{i=1}^{\infty}$ be a bounded sequence in $\cl{U}$ and suppose that $(y_i)_{i=1}^{\infty}$ is a sequence in $U$ which converges to some $y_{\infty}\in \cl{U}$ and also satisfies $\lim_{i\to\infty}B(x_i,y_i)=0$. Since the space is finite-dimensional, $(x_i)_{i=1}^{\infty}$ has at least one cluster point. Let $x_{\infty}=(x_{\infty,k})_{k=1}^n$ be an arbitrary cluster point  of $(x_i)_{i=1}^{\infty}$. Then $\lim_{j\to\infty}\|x_{\infty}-x_{i_j}\|=0$ for some subsequence $(x_{i_j})_{j=1}^{\infty}$  of $(x_i)_{i=1}^{\infty}$. In particular, $x_{\infty,k}=\lim_{j\to\infty}x_{{i_j},k}$ for each $k\in\{1\ldots,n\}$. Since we assume that  $\lim_{i\to\infty}y_i=y_{\infty}$, the continuity of  $B_k$ on $[0,\infty)^2$ implies that $\lim_{j\to\infty}B_k(x_{i_j,k},y_{i_j,k})=B_k(x_{\infty,k},y_{\infty,k})$ for all $k\in\{1\ldots,n\}$. Since we assume that $\lim_{i\to\infty}B(x_i,y_i)=0$, we have  $0=\lim_{j\to\infty}B(x_{i_j},y_{i_j})=B(x_{\infty},y_{\infty})=\sum_{k=1}^n B_k(x_{\infty,k},y_{\infty,k})$ for all $k\in\{1\ldots,n\}$. Since $B_k$ is non-negative, as follows from Proposition \bref{prop:BregmanProperties}\beqref{BregProp:B=0}, we conclude that  $B_k(x_{\infty,k},y_{\infty,k})=0$  for all $k\in\{1\ldots,n\}$. 
 
Let $k\in \{1,\ldots,n\}$ be fixed. There are two possibilities. If $y_{\infty,k}=0$, then, from the definition of $B_k$, one has $B_k(x_{\infty,k},y_{\infty,k})=x_{\infty,k}^q/(q-1)$, and hence $x_{\infty,k}=0$. If $y_{\infty,k}>0$, then the fact that $b_k$ is strictly convex on $(0,\infty)$ implies, using
Proposition \bref{prop:BregmanProperties}\beqref{BregProp:B=0}, that $x_{\infty,k}=y_{\infty,k}$. Hence in both cases $x_{\infty,k}=y_{\infty,k}$ and this holds for every $k\in\{1,\ldots,n\}$. We conclude that $x_{\infty}=y_{\infty}$. Since $x_{\infty}$ was an arbitrary cluster point of $(x_i)_{i=1}^{\infty}$, it follows that $\lim_{i\to\infty}x_{i}=y_{\infty}$, and this establishes Definition \bref{def:BregmanDiv}\beqref{BregmanDef:B(x_i,y_i)=0==>|x_i-y|=0}. 

To see that $b$ is sequentially consistent, let $(x_i)_{i=1}^{\infty}$ be a sequence in $\dom(b)$ and $(y_i)_{i=1}^{\infty}$ be a bounded sequence in $U$ such that $\lim_{i\to\infty}B(x_i,y_i)=0$. It follows, as in previous paragraphs, that $\lim_{i\to\infty}B_k(x_{i,k},y_{i,k})=0$ for all $k\in\{1,\ldots,n\}$. Assume to the contrary that it is not true that $\lim_{i\to\infty}|x_{i,k}-y_{i,k}|=0$ for some $k\in\{1,\ldots,n\}$. Then there is some $\epsilon>0$ and subsequences $(x_{i_j,k})_{j=1}^{\infty}$ of $(x_{i,k})_{i=1}^{\infty}$ and $(y_{i_j,k})_{j=1}^{\infty}$ of $(y_{i,k})_{i=1}^{\infty}$, respectively, such that $|x_{i_j,k}-y_{i_j,k}|\geq\epsilon$. It must be that $(x_{i_j,k})_{j=1}^{\infty}$ is bounded. Indeed, assume to the contrary that it is unbounded. Then there is an infinite subset $J$ of indices $j\in\N$ such that $x_{i_j,k}\to \infty$ as $j\to\infty$ and $j\in J$. Since $(y_i)_{i=1}^{\infty}$ is bounded, there is $M>0$ such that $|y_{i}|<M$ for all $i\in\N$. Now, since  $\lim_{i\to\infty}B_k(x_{i,k},y_{i,k})=0$, we have, in particular, $B_k(x_{i_j,k},y_{i_j,k})\leq 1$ for all $j$ large enough. Thus the definition of $B_k$ implies that $x_{i_j,k}^q-qM^{q-1}x_{i_j,k}\leq x_{i_j,k}^q-qy_{i_j,k}^{q-1}x_{i_j,k}\leq q-1+qM^q+M^q$ for all $j\in\N$ large enough. However, this inequality cannot hold since its left-hand side tends to infinity when $j\in J$ and $j\to\infty$: this is because $q>2$ and because the function $t\mapsto t^q-ct$ tends to infinity as $t\to\infty$ (here $c$ is an arbitrary fixed positive number). Consequently, $(x_{i_j,k})_{j=1}^{\infty}$ is indeed bounded, as asserted. 

By passing to subsequences and using the compactness of closed and bounded intervals in $[0,\infty)$, we can find  points $x_{\infty,k}$ and $y_{\infty,k}$ in $[0,\infty)$, and an infinite subset $J'$ of $\N$ such that $x_{\infty,k}=\lim_{j\to\infty, j\in J'}x_{i_j,k}$ and $y_{\infty,k}=\lim_{j\to\infty, j\in J'}y_{i_j,k}$, respectively. Since $\lim_{j\to\infty, j\in J'}B_k(x_{i_j,k},y_{i_j,k})=0$ and since $B_k$ is continuous on $[0,\infty)^2$, it follows that  $B_k(x_{\infty,k},y_{\infty,k})=0$. If $y_{\infty,k}>0$, then from Proposition \bref{prop:BregmanProperties}\beqref{BregProp:B=0} and the fact that $|x_{\infty,k}-y_{\infty,k}|\geq \epsilon>0$ we have $B_k(x_{\infty,k},y_{\infty,k})>0$, a contradiction. Hence $y_{\infty,k}=0$, but then the definition of $B_k$ implies that $0=B_k(x_{\infty,k},y_{\infty,k})=x_{\infty,k}^q/(q-1)$, and therefore $x_{\infty,k}=0$. Thus $x_{\infty,k}=y_{\infty,k}$, a contradiction to the inequality $|x_{\infty,k}-y_{\infty,k}|\geq \epsilon>0$. This contradiction shows that the condition $\lim_{i\to\infty}|x_{i,k}-y_{i,k}|=0$ must hold for every $k\in\{1,\ldots,n\}$. Hence $\lim_{i\to\infty}\|x_i-y_i\|=0$, as required.

Finally, we still need to establish Definition \bref{def:BregmanDiv}\beqref{BregmanDef:BoundedLevelSet}. 
This will be a consequence of a direct verification. 
Indeed, fix $x\in \dom(b)$ and $\gamma\in [0,\infty)$ and assume that $L(x,\gamma)\neq\emptyset$, otherwise the assertion is trivial. Let $y\in L(x,\gamma)$ be arbitrary. Then $y\in U$ and  $B(x,y)\leq\gamma$, and from the decomposition $B=\sum_{k=1}^nB_k$ and the fact that each $B_k$ is non-negative we have $B_k(x_k,y_k)\leq \gamma$ for all $k\in\{1,\ldots,n\}$. Thus $-y_k^q+qy_k^q-qy_k^{q-1}x_k\leq (q-1)\gamma-x_k^q$. 
Now fix $k\in\{1,\ldots,n\}$. From previous lines, we obtain 
\begin{equation}\label{eq:y_k^q}
y_k^q\left(q-1-\frac{qx_k}{y_k}\right)\leq (q-1)\gamma-x_k^q. 
\end{equation}
We now separate the analysis into two cases. In the first case $y_k<2qx_k$. Hence $2	x_k$ is an upper bound on $y_k$. In the second case $y_k\geq 2qx_k$. This inequality, as well as the inequalities $q>2$ and $x_k\geq 0$, imply that $q-1-(qx_k/y_k)>1-0.5=0.5$. This fact and \beqref{eq:y_k^q} imply that the right-hand  side of \beqref{eq:y_k^q} is nonnegative and $0.5y_k^q\leq (q-1)\gamma-x_k^q$. Hence $y_k\leq (2(q-1)\gamma-2x_k^k))^{1/q}$. To conclude, if $y\in L(x,\gamma)$, then for all $k\in\{1,\ldots,n\}$ we have $y_k\leq \max\{2qx_k,(2(q-1)\gamma-2x_k^q)^{1/q}\}$. 
 Since all the norms on $\R^n$ are equivalent, there is a constant $\sigma>0$ such that $\sigma\|z\|\leq \|z\|_{\infty}$ for each $z\in\R^n$. As a result of the previous lines, we deduce that 
\begin{equation}
\|y\|\leq\frac{\|y\|_{\infty}}{\sigma}\leq\frac{1}{\sigma} \max_{k\in\{1,\ldots,n\}}\{\max\{2qx_k,(2(q-1)\gamma-2x_k^q)^{1/q}\}\}. 
\end{equation}  
The above inequality shows  that  $L(x,\gamma)$ is bounded, as required.

\subsubsection{The case $q\in (1,2)$} In this case $b$ is a sequentially consistent Bregman function which satisfies the limiting difference property, as a  consequence of previous subsections and Corollary \bref{cor:BregmanFunction-b'-is-unicont} (here we also use the assumption that $X$ is finite-dimensional and the simple observation that $b'$ can obviously be extended to a continuous function defined on $\cl{U}$; thus it follows from classical theorems in analysis that the extension of $b'$ is uniformly  continuous on every compact subset of $\cl{U}$ and hence $b'$ has this property on every bounded subset of $U$). In addition, 
Proposition \bref{prop:BregmanProperties}\beqref{BregProp:LevelSet2BoundedSufficientConditions} and the fact that $b$ is strictly convex on $\dom(b)$ imply that the second type level-sets of $B$ are bounded.  

\subsubsection{The case $q\in (0,1)$} The proof that $b$ is a sequentially consistent Bregman function which satisfies the limiting difference property is as in Subsection \bref{subsec:BregmanGibbsShannon},  where the only essential differences are that now one (of course) uses the corresponding  subsections of Section \bref{sec:HavrdaCharvatTsallis}, and when showing that Definition \bref{def:BregmanDiv}\beqref{BregmanDef:|x-y_i|=0==>B(x,y_i)=0}  holds using Proposition \bref{prop:BregmanProperties}\beqref{BregProp:D(x,y_i)=0dom(b)}, then  instead of using the function $t\mapsto \log(t)(x_k-t)$ and the limit $\lim_{t\to 0+}t\log(t)=0$,	one uses the function $t\mapsto (q/(q-1))t^{q-1}(x_k-t)$ and the obvious limit $\lim_{t\to 0+}t^{q-1}t=0$,  respectively. In addition, 
Proposition \bref{prop:BregmanProperties}\beqref{BregProp:LevelSet2BoundedSufficientConditions} and the fact that $b$ is strictly convex on $\dom(b)$ imply that the second type level-sets of $B$ are bounded. 

\subsubsection{The case $q\in (-\infty,0)$}\label{subsubsec:q<0} In this case it is not true that $b$ is a Bregman function, for instance because not all the level sets  of $B$ are bounded. Indeed, fix some $x\in [1,\infty)^n$ and let $\gamma>\max\{nx_k^q/(1-q): k\in \{1,\ldots,n\}\}$ be arbitrary. By the choice of $\gamma$ and since $q<0$, we have $\lim_{t\to\infty}(x_k^q-t^q-qt^{q-1}(x_k-t))/(1-q)=x_k^q/(1-q)<\gamma/n$ for all $k\in\{1,\ldots,n\}$. Hence there is $t_0>0$ large enough such that for all $t\in(t_0,\infty)$ and all $k\in\{1,\ldots,n\}$, one has $(x_k^q-t^q-qt^{q-1}(x_k-t))/(1-q)<\gamma/n$. The previous inequality and \beqref{eq:Bregman-Havrda-Charvat-Tsallis} imply that the inequality $B(x,y)\leq\gamma$ is satisfied for all $y\in (t_0,\infty)^n$. In other words, the level-set $\{y\in X: B(x,y)\leq \gamma\}$ contains the unbounded set $(t_0,\infty)^n$.

\section{The negative Burg entropy}\label{sec:Burg}
Let $X:=\R^n$ ($n\in\N$) with an arbitrary norm $\|\cdot\|$. Let  $U:=(0,\infty)^n$. The negative Burg entropy $b:X\to(-\infty,\infty]$ is defined by
\begin{equation}\label{eq:BurgEntropy}
b(x):=\left\{\begin{array}{lll}
-\displaystyle{\sum_{k=1}^n} \log(x_k), & x=(x_k)_{k=1}^n\in U,\\
\infty, & x\notin U.
\end{array}
\right.
\end{equation}
The corresponding (pre-)Bregman divergence, also known as the Itakura-Saito divergence, is defined on $X^2$ by 
\begin{equation}\label{eq:ItakuraSaito}
B(x,y):=\left\{\begin{array}{lll}
\displaystyle{\sum_{k=1}^n}\left(\log\left(\displaystyle{\frac{y_k}{x_k}}\right)+\displaystyle{\frac{x_k}{y_k}}-1\right), & (x,y)\in U^2,\\
\infty, & \textnormal{otherwise}.
\end{array}
\right.
\end{equation}
It seems that the negative of $b$ was introduced by Burg in the 
continuous case (an integral instead of a sum) in 1967, in the form 
of a paper presented in a conference \cite{Burg1967conf}. An extended version of this 
unpublished paper appears in Burg's 1975 thesis \cite{Burg1975phd} (see, for instance, \cite[p. 1]{Burg1975phd}). See also \cite{EdwardFitelson1973jour, Frieden1975inbook} and some of the references therein for other works which mention explicitly $b$ or $-b$ (still in the continuous case, and before the appearance of \cite{Burg1975phd}) and attribute $-b$ to Burg \cite{Burg1967conf}. The introduction of the Itakura-Saito divergence is frequently attributed to Itakura and Saito \cite{ItakuraSaito1968inproc}, but in that extended abstract neither $B$ nor a continuous analogous of it, namely with integrals instead of sums,  appear (there is, however, a closely related expression in \cite[p. C-18, equation (7)]{ItakuraSaito1968inproc} in the continuous case). In the context of the theory of Bregman divergences, it seems that $b$ and $B$ were first discussed in a somewhat detailed manner in \cite{CensorLent1987jour}; see also \cite{CensorDePierroIusem1991jour} for a related discussion. 

\subsection{Basic properties} For every $z\in U$ and every vector $w\in X$, one has  $b'(z)=(-1/z_i)_{i=1}^n$ and $b''(z)(w,w)=\sum_{i=1}^n w_i^2/z_i^2$. Both derivative are continuous on $\dom(b)=U$. From considerations similar to the ones given in Subsection  \bref{subsec:HCT-BasicProperties} it follows that $b$ is strictly convex on $U$ and lower semicontinuous on $X$, and also essentially smooth. Thus $b$ is Legendre.

\subsection{Strong convexity} Fix $x,y\in U$, $x\neq y$. We can follow word for word the analysis in Section \bref{sec:GibbsShannon} (near \beqref{eq:W})  to conclude that both 
$\mu_{1}[x,y]:=c_2^2/\max\{\|x\|^{2}_{\infty},\|y\|^{2}_{\infty}\}$ and  $\mu_{2}[x,y]:=c_2^2/(c_{\infty}^2\max\{\|x\|^{2},\|y\|^{2}\})$ are strong convexity parameters of $b$ on $[x,y]$. This shows that $b$ is strongly convex on any bounded, closed and convex subset $S\neq\emptyset$ of $U$ with $\mu[S]:=c_2^2/(c_{\infty}^2M_S^2)$ as a strong convexity parameter of $b$ on $S$, where $M_S$ is an upper bound on the norm of vectors in $S$. 

\subsection{Relative uniform convexity} It is shown below that for each $x\in U$, the function $b$ is uniformly  convex relative to $(\{x\},\{w\in U: \|w\|>r_x\})$ for some $r_x>0$ and with 
\begin{equation}\label{eq:RelativeGaugeBurg}
\psi(t):=\frac{1}{4}\log(1+t), \quad t\in [0,\infty),
\end{equation}
as a relative gauge. The proof does not follow the same reasoning as in Sections  \bref{sec:GibbsShannon} and \bref{sec:HavrdaCharvatTsallis}, namely using the strong convexity  estimate on $[x,y]$ and referring to Remark \bref{rem:TypesOfConvexity}\beqref{item:mu[x,y]},  because by doing so one ends up with $\psi(x,y):=\alpha\|x-y\|^2/\|y\|^2$ as a pre-gauge (where  $\alpha>0$ is some constant) and it is not clear if there exists a relative gauge which is a lower bound of this pre-gauge and also tends to infinity when its argument tends to infinity. Instead, below we work directly with \beqref{eq:RelativeUniformConvexity} and carefully analyze this inequality. 

First, we will define $r_x$. The definition is somewhat involved, but the reasoning behind it will become clear later. Since all the norms on the finite-dimensional space $X$ are equivalent, there exists $\gamma>0$ such that 
\begin{equation}\label{eq:gamma<infty}
\|w\|\leq \gamma\|w\|_{\infty}\quad \forall w\in X.
\end{equation}
Since $\lim_{s\to\infty}(1.5s^{5/8}\log(s)-s+1)=-\infty$, there exists $s_1>1$ such that for all $s>s_1$ the inequality $1.5s^{5/8}\log(s)-s+1<0$ holds. In fact, $s_1$ can be taken as $4^8$. Define 
\begin{equation}\label{eq:t_1}
t_1:=\max\{s_1\gamma x_j: j\in \{1,\ldots,n\}\}. 
\end{equation}
Since $\lim_{t\to\infty}[0.5\log(t)-0.25\log(1+t)-0.5\log(2\gamma \|x\|_{\infty})-4]=\infty$, there exists $t_2>1$ such that 
\begin{equation}\label{eq:0.5log0.25}
\frac{1}{2}\log\left(\frac{t}{2\gamma \|x\|_{\infty}}\right)-4>\frac{1}{4}\log(1+t),\quad \forall\,t>t_2. 
\end{equation}
Now we can define $r_x$ as follows: 
\begin{equation}\label{eq:r_x}
r_x:=\max\{t_1,2t_2,2\|x\|\}.
\end{equation}
We want to prove that for every fixed $y\in \{w\in U: \|w\|>r_x\}$, we have $h(\lambda)\geq 0.25\log(1+\|x-y\|)$ for each $\lambda\in (0,1)$, where 
\begin{equation}\label{eq:h_burg}
h(\lambda):=h_{x,y}(\lambda):=\frac{\lambda b(x)+(1-\lambda) b(y)-b(\lambda x+(1-\lambda)y)}{\lambda(1-\lambda)}.
\end{equation}
Once this inequality is proved, it follows from \beqref{eq:RelativeUniformConvexity} that the function $\psi$ defined in \beqref{eq:RelativeGaugeBurg} is a relative gauge of $b$ on $(\{x\},\{w\in U: \|w\|>r_x\})$. So let $y\in \{w\in U: \|w\|>r_x\}$ be fixed. Since $\|y\|>2\|x\|$ it follows from the triangle inequality that 
\begin{equation}\label{eq:2x<y}
0.25\|y\|<0.5(\|y\|-\|x\|)\leq 0.5\|y-x\|\leq 0.5\|y\|+0.5\|x\|<\|y\|. 
\end{equation}
Let $k\in \{1,\ldots,n\}$ be an index (any index if there are several ones) for which $y_k=\max\{y_i: i\in\{1\ldots,n\}\}$. Then $y_k=\|y\|_{\infty}$. From \beqref{eq:gamma<infty} we have $\|y\|/\gamma\leq y_k$. This inequality, \beqref{eq:r_x}, the assumption that $\|y\|>r_x$ and \beqref{eq:t_1} imply that 
\begin{equation}\label{eq:y_k/x_k}
\frac{y_k}{x_k}\geq\frac{\|y\|}{\gamma x_k}>\frac{r_x}{\gamma x_k}\geq \frac{t_1}{\gamma x_k}\geq s_1.
\end{equation}
The inequality \beqref{eq:2x<y}, when combined with \beqref{eq:r_x} and the fact that $\|y\|>r_x$, imply that 
\begin{equation}\label{eq:and}
\|y\|\geq 0.5\|x-y\|\quad\textnormal{and}\quad \|x-y\|\geq \|y\|/2>t_2.
\end{equation}
Given $s>s_1$ (see the discussion after \beqref{eq:gamma<infty}), consider the function $v:[0.5,1]\to\R$ defined by $v(\lambda):=\lambda+(1-\lambda)s-s^{0.5(1-\lambda)(2+\lambda)}$ for all $\lambda\in [0.5,1]$. Then  for all $\lambda\in [0.5,1]$, 
\begin{equation}
v'(\lambda)=1-s+(0.5+\lambda)s^{0.5(1-\lambda)(2+\lambda)}\log(s)\leq 1-s+1.5s^{5/8}\log(s)<0,
\end{equation}
because the function $\lambda\mapsto 0.5(1-\lambda)(2+\lambda)$ is decreasing and positive on $[0.5,1]$, and because $s>s_1>1$. 
Hence $v$ is decreasing on $[0.5,1]$, and since $v(1)=0$, it follows that $v(\lambda)>0$ for every $\lambda\in [0.5,1)$. Therefore $\lambda+(1-\lambda)s\geq s^{0.5(1-\lambda)(2+\lambda)}=s^{1-\lambda+0.5(1-\lambda)\lambda}$ and thus $s^{0.5\lambda(1-\lambda)}\leq \lambda(1/s)^{1-\lambda}+(1-\lambda)s^{\lambda}$ for every $\lambda\in [0.5,1]$. After applying $\log$ to both sides of this inequality, we have
\begin{equation}\label{eq:0.5log(s)}
\frac{1}{2}\log(s)\leq \frac{1}{\lambda(1-\lambda)}\log\left(\lambda\left(\frac{1}{s}\right)^{1-\lambda}+(1-\lambda)s^{\lambda}\right)
\end{equation}
for every $s>s_1$ and every $\lambda\in [0.5,1)$. When we combine \beqref{eq:0.5log(s)} and \beqref{eq:y_k/x_k} with the immediate inequality $x_k\leq \|x\|_{\infty}$, then for each $\lambda\in [0.5,1)$, we obtain 
\begin{multline}\label{eq:0.5log(s) y_k/x_k x_infty}
\frac{1}{\lambda(1-\lambda)}\log\left(\lambda \left(\frac{x_k}{y_k}\right)^{1-\lambda}+(1-\lambda)\left(\frac{y_k}{x_k}\right)^{\lambda}\right)
\geq \frac{1}{2}\log\left(\frac{y_k}{x_k}\right)\geq\frac{1}{2}\log\left(\frac{\|y\|}{\gamma\|x\|_{\infty}}\right).
\end{multline}
Now we start with \beqref{eq:h_burg}, perform simple manipulations based on \beqref{eq:BurgEntropy} and   \beqref{eq:ItakuraSaito}, combine them with the well-known weighted arithmetic-geometric mean inequality  $\lambda s+(1-\lambda)t\geq s^{\lambda}t^{1-\lambda}$ for all $\lambda\in [0,1]$ and $t,s\in (0,\infty)$ (which is nothing but an immediate consequence of the concavity of $\log$), use the fact that $t\mapsto \log(t)$ is nonnegative on $[1,\infty)$ and increasing on $(0,\infty)$, use  \beqref{eq:0.5log(s) y_k/x_k x_infty}, use \beqref{eq:and}, and also use \beqref{eq:0.5log0.25}. The result of this process yields the following inequality for each $\lambda\in [0.5,1)$: 
\begin{multline}\label{eq:h>psiOn[0.5,1)}
h(\lambda)=\frac{\lambda b(x)+(1-\lambda)b(y)-b(\lambda x+(1-\lambda)y)}{\lambda(1-\lambda)}
=\frac{1}{\lambda(1-\lambda)}\sum_{j=1}^n\log\left(\frac{\lambda x_j+(1-\lambda)y_j}{x_j^{\lambda}y_j^{1-\lambda}}\right)\\
=\frac{1}{\lambda(1-\lambda)}\log\left(\frac{\lambda x_k+(1-\lambda)y_k}{x_k^{\lambda}y_k^{1-\lambda}}\right)+\frac{1}{\lambda(1-\lambda)}\sum_{j\neq k}\log\left(\frac{\lambda x_j+(1-\lambda)y_j}{x_j^{\lambda}y_j^{1-\lambda}}\right)\\
\geq \frac{1}{\lambda(1-\lambda)}\log\left(\frac{\lambda x_k+(1-\lambda)y_k}{x_k^{\lambda}y_k^{1-\lambda}}\right)=\frac{1}{\lambda(1-\lambda)}\log\left(\lambda \left(\frac{x_k}{y_k}\right)^{1-\lambda}+(1-\lambda)\left(\frac{y_k}{x_k}\right)^{\lambda}\right)\\
\geq \frac{1}{2}\log\left(\frac{y_k}{x_k}\right)
\geq  \frac{1}{2}\log\left(\frac{\|y\|}{\gamma\|x\|_{\infty}}\right)
\geq \frac{1}{2} \log\left(\frac{\|y-x\|}{2\gamma\|x\|_{\infty}}\right)\\
> \frac{1}{2}\log\left(\frac{\|y-x\|}{2\gamma\|x\|_{\infty}}\right)-4
\geq \frac{1}{4}\log(1+\|x-y\|),
\end{multline}
as required. It remains to show that $h(\lambda)\geq 0.25\log(1+\|x-y\|)$ for every $\lambda\in (0,0.5)$. Using elementary calculus, we see that the function $u(\lambda):=\log(1-\lambda)+2\lambda$ is increasing on $[0,0.5]$ and since $u(0)=0$, we have $\log(1-\lambda)/\lambda\geq -2$ for all $\lambda\in (0,0.5]$. When we combine this inequality with the inequality $-2/(1-\lambda)\geq -4$ which holds for all $\lambda \in [0,0.5]$, we arrive at the inequality $\log(1-\lambda)/(\lambda(1-\lambda))\geq -4$ for each $\lambda\in (0,0.5]$. In addition, the monotonicity of the $\log$ function and the fact that $\lambda x_k+(1-\lambda)y_k\geq (1-\lambda)y_k>0$ for every $\lambda\in [0,1)$ imply that for all $\lambda\in [0,1)$ the following inequality holds: $\log((\lambda x_k+(1-\lambda)y_k)/(x_k^{\lambda}y_k^{1-\lambda}))\geq\log(((1-\lambda)y_k)/(x_k^{\lambda}y_k^{1-\lambda}))=\log(((1-\lambda)y_k^{\lambda})/x_k^{\lambda})$. 

 The above-mentioned inequalities and considerations mentioned in and before \beqref{eq:h>psiOn[0.5,1)} (considerations which do not depend on $\lambda$: they depend on \beqref{eq:y_k/x_k} and other inequalities which are independent of $\lambda$) show that for each $\lambda\in (0,0.5]$,
\begin{multline*}
h(\lambda)\geq \frac{1}{\lambda(1-\lambda)}\log\left(\frac{\lambda x_k+(1-\lambda)y_k}{x_k^{\lambda}y_k^{1-\lambda}}\right)\\
\geq \frac{1}{\lambda(1-\lambda)}\log\left(\frac{(1-\lambda)y_k^{\lambda}}{x_k^{\lambda}}\right)
=\frac{\log(1-\lambda)}{\lambda(1-\lambda)}+\frac{1}{1-\lambda}\log\left(\frac{y_k}{x_k}\right)\\
\geq -4+\log\left(\frac{y_k}{x_k}\right)\geq -4+\log\left(\frac{\|y-x\|}{2\gamma\|x\|_{\infty}}\right)\geq \frac{1}{4}\log(1+\|x-y\|),
\end{multline*}
as claimed. 

\subsection{No global uniform convexity}\label{subsec:BurgNoGlobalUC}
We show below that $b$ cannot be uniformly convex on $U$, and hence also on $\dom(b)$ (the case $n=1$ is stated without a proof in \cite[p. 186]{BauschkeCombettes2017book}). Indeed, assume to the contrary that $b$ is uniformly convex on $U$. Given $s>1$, let $x(s)\in U$ and $y(s)\in U$ be defined by $x_1(s):=s$, $y_1(s):=s+1$, and $x_i(s):=1=:y_i(s)$ for all $i\in\{1,\ldots,n\}\backslash\{1\}$.  Since all the norms on $\R^n$ are equivalent, there is $\eta>0$ such that $\|z\|\geq \eta\|z\|_2$ for each $z\in \R^n$, where $\|\cdot\|_2$ is the Euclidean norm. Therefore $\|x(s)-y(s)\|\geq \eta\|x(s)-y(s)\|_2=\eta$. Consequently, Lemma \bref{lem:MonotonePsi} implies that $\psi_{b,U}(\|x(s)-y(s)\|)\geq\psi_{b,U}(\eta)$ for each $s>1$, where  $\psi_{b,U}$ is the modulus of uniform convexity of $b$ on $U$. On the other hand, from Proposition \bref{prop:BregmanProperties}\beqref{BregProp:BregPsiS1S2} and \beqref{eq:ItakuraSaito}, we have  
\begin{equation}\label{eq:BurgNoUC}
\psi_{b,U}(\|x(s)-y(s)\|)\leq B(x(s),y(s))=\log\left(1+\frac{1}{s}\right)+\frac{s}{s+1}-1{\xrightarrow[s\to \infty]{}}\,\,0.
\end{equation}
Since $b$ is uniformly convex, we have $\psi_{b,U}(\eta)>0$. Therefore $\psi_{b,U}(\|x(s)-y(s)\|)<\psi_{b,U}(\eta)$ for $s$ sufficiently large, as follows from \beqref{eq:BurgNoUC}. Thus we arrive at a contradiction. This shows that $b$ cannot be uniformly convex on $U$. 

\subsection{$b$ is a sequentially consistent Bregman function which satisfies the limiting difference property} 
Based on the assertions proved in previous subsections, the proof follows the same lines as the proof of Corollary  \bref{cor:BregmanFunction-b'-is-unicont} with the exception that for establishing Definition \bref{def:BregmanDiv}\beqref{BregmanDef:|x-y_i|=0==>B(x,y_i)=0} we use Proposition \bref{prop:BregmanProperties}\beqref{BregProp:D(x,y_i)=0U)} instead of Proposition  \bref{prop:BregmanProperties}\beqref{BregProp:Continuous_b|x_i-y_i|=0==>B(x_i,y_i)=0}. Moreover, 
Proposition \bref{prop:BregmanProperties}\beqref{BregProp:LevelSet2BoundedSufficientConditions} and the fact that $b$ is strictly convex on $\dom(b)$ imply that the second type level-sets of $B$ are bounded.

\section{A negative iterated log entropy}\label{sec:IteratedLog}
Let $X:=\R^n$ ($n\in\N$) with an arbitrary norm $\|\cdot\|$. Let  $U:=(1,\infty)^n$. Let $b:X\to(-\infty,\infty]$ be the ``negative iterated log entropy'' defined by
\begin{equation}\label{eq:IteratedLog-b}
b(x):=\left\{\begin{array}{lll}
-\displaystyle{\sum_{k=1}^n} \log(\log(x_k)), & x=(x_k)_{k=1}^n\in U,\\
\infty, & x\notin U.
\end{array}
\right.
\end{equation}
The corresponding (pre-)Bregman divergence is 
\begin{equation}\label{eq:IteratedLog-B}
B(x,y):=\left\{\begin{array}{lll}
\displaystyle{\sum_{k=1}^n} \left[\log\left(\displaystyle{\frac{\log(y_k)}{\log(x_k)}}\right)+\displaystyle{\frac{x_k-y_k}{y_k\log(y_k)}}\right], & (x,y)\in U^2,\\
\infty, & \textnormal{otherwise}.
\end{array}
\right.
\end{equation}
As far as we know, so far neither $b$ nor $B$ have been considered elsewhere in the theory of Bregman functions and divergences. 

\subsection{Basic properties}\label{subsec:IteratedLogBasic}
A simple verification shows that $b'(z)=(-1/(z_k\log(z_k)))_{k=1}^n$ and also that 
\begin{equation*}
b''(z)(w,w)=\sum_{k=1}^n\displaystyle{\frac{1}{z_k\log(z_k)}}\left(\displaystyle{\frac{1}{z_k\log(z_k)}}+\displaystyle{\frac{1}{z_k}}\right)w_k^2\quad \forall z\in U,\,\forall w\in X. 
\end{equation*}
In particular, $b$, as well as its first and second derivatives, are continuous on $\dom(b)=U$, and for each $z\in U$ and $w\neq 0$, one has $b''(z)(w,w)>0$. Hence, by a well-known classical result, $b$ is strictly convex on $U$ (and hence convex on $X$). Since $b$ is proper and $\dom(\partial b)=U$, we conclude that $b$ is essentially strictly convex. In fact, in Subsection \bref{subsec:IteratedLogStronglyConvex} we show that $b$ is strongly convex on all nonempty, bounded and convex subsets of $U$. Considerations similar to the ones mentioned in Subsection \bref{subsec:HCT-BasicProperties} show that $b$ is lower semicontinuous on $X$. In addition, $b$ is essentially smooth and hence Legendre. 

\subsection{Strong Convexity}\label{subsec:IteratedLogStronglyConvex}
Fix some $x,y\in U$, $x\neq y$. We can follow word for word the analysis in Subsection \bref{subsec:GibbsShannonStronglyConvex} (that is, \beqref{eq:W} and the discussion before and after it) to conclude that if $M_{\infty}(x,y):=\max\{\|x\|_{\infty},\|y\|_{\infty}\}$ and $M(x,y):=\max\{\|x\|,\|y\|\}$, then both $\mu_{1}[x,y]$ and $\mu_{2}[x,y]$ which are defined below are strong convexity parameters of $b$ on $[x,y]$:
\begin{subequations}
\begin{equation}
\mu_{1}[x,y]:=\displaystyle{\frac{c_2^2}{M_{\infty}(x,y)\log(M_{\infty}(x,y))}}\left(\displaystyle{\frac{1}{M_{\infty}(x,y)\log(M_{\infty}(x,y))}}+\displaystyle{\frac{1}{M_{\infty}(x,y)}}\right),
\end{equation}
\begin{equation}\label{eq:IteratedLog m_2[x,y]}
\mu_{2}[x,y]:=\displaystyle{\frac{c_2^2}{c_{\infty}M(x,y)\log(c_{\infty}M(x,y))}}\left(\displaystyle{\frac{1}{c_{\infty}M(x,y)\log(c_{\infty}M(x,y))}}+\displaystyle{\frac{1}{c_{\infty}M(x,y)}}\right). 
\end{equation}
\end{subequations}
Hence, again from an inequality similar to \beqref{eq:W} (but with $\mu_{2}[x,y]$, and then with $\mu[S]$ below instead of the expression written there), if $S$ is a given nonempty, bounded and convex subset  of $U$ and $M_S$ is an upper bound on the norm of the vectors of $S$, then  Proposition \bref{prop:BregmanStronglyConvex}\beqref{item:eta_S} implies that $b$ is strongly convex on $S$ with
\begin{equation}
\mu[S]:=\displaystyle{\frac{c_2^2}{c_{\infty}M_S\log(c_{\infty}M_S)}}\left(\displaystyle{\frac{1}{c_{\infty}M_S\log(c_{\infty}M_S)}}+\displaystyle{\frac{1}{c_{\infty}M_S}}\right) 
\end{equation}
 as a strong convexity parameter.

\subsection{Relative uniform convexity}\label{subsec:IteratedLogRelativeUniformConvexity}
We believe that $b$ is uniformly convex relative to  pairs of the form $(\{x\}, S_2)$, where $x\in\dom(b)$ and $S_2:=\{w\in U: \|w\|>r_x\}$ for some $r_x>0$, but so far we have not been able to prove this. Actually, we conjecture that there is some $\beta\in(0,1)$ such that $\psi(t):=\beta \log(1+\log(1+t))$, $t\in [0,\infty)$,  is a possible relative gauge.

\subsection{No global uniform convexity}\label{subsec:IteratedLogNoGlobalUC}
We show below that $b$ cannot be uniformly convex on $U$, that is, on $\dom(b)$. Indeed, assume to the contrary that $b$ is uniformly convex on $U$. Given $s>1$, let $x(s)=(x_k(s))_{k=1}^n\in U$ and $y(s)=(y_k(s))_{k=1}^n\in U$ be defined by $x_1(s):=s$, $y_1(s):=s+1$, and $x_k(s):=2=:y_k(s)$ for all $k\in\{1,\ldots,n\}\backslash\{1\}$.  Since all the norms on $\R^n$ are equivalent, there is $\eta>0$ such that $\|z\|\geq \eta\|z\|_2$ for each $z\in \R^n$, where $\|\cdot\|_2$ is the Euclidean norm. Therefore $\|x(s)-y(s)\|\geq \eta\|x(s)-y(s)\|_2=\eta$ for all $s>1$. Consequently, Lemma \bref{lem:MonotonePsi} and the fact that $b$ is uniformly convex imply that $\psi_{b,U}(\|x(s)-y(s)\|)\geq\psi_{b,U}(\eta)>0$ for each $s>1$, where  $\psi_{b,U}$ is the modulus of uniform convexity of $b$ on $U$. On the other hand, Proposition \bref{prop:BregmanProperties}\beqref{BregProp:BregPsiS1S2} and \beqref{eq:IteratedLog-B} imply that 
\begin{multline*}
\psi_{b,U}(\|x(s)-y(s)\|)\leq B(x(s),y(s))=\log\left(\frac{\log(s+1)}{\log(s)}\right)\,-\,\frac{1}{(s+1)\log(s+1)}\\
=\log\left(\frac{\log(s)+\log(1+\frac{1}{s}))}{\log(s)}\right)-\,\frac{1}{(s+1)\log(s+1)}{\,\xrightarrow[s\to \infty]{}}\,\,\log(1)=0.
\end{multline*}
In particular,  $\psi_{b,U}(\|x(s)-y(s)\|)<\psi_{b,U}(\eta)$ for $s$ sufficiently large.  Thus we arrive at a contradiction. This shows that $b$ cannot be uniformly convex on $U$.

\subsection{$b$ is a sequentially consistent Bregman function which satisfies the limiting difference property} 
Parts \beqref{BregmanDef:U} and \beqref{BregmanDef:ConvexLSC} of Definition \bref{def:BregmanDiv} are satisfied by the assumptions on $b$ and the assertions proved in previous subsections, Definition \bref{def:BregmanDiv}\beqref{BregmanDef:B} is just the definition of the Bregman divergence $B$, Definition \bref{def:BregmanDiv}\beqref{BregmanDef:|x-y_i|=0==>B(x,y_i)=0} follows from  Proposition \bref{prop:BregmanProperties}\beqref{BregProp:D(x,y_i)=0U)}. Definition \bref{def:BregmanDiv}\beqref{BregmanDef:B(x_i,y_i)=0==>|x_i-y|=0} and the sequential consistency of $b$ follow from Subsection \bref{subsec:IteratedLogStronglyConvex} and Proposition \bref{prop:BregmanProperties}\beqref{BregProp:B(x_i,y_i)=0AndOneSequenceBounded==>|x_i-y_i|=0}. The limiting difference property follows from Proposition \bref{prop:BregmanProperties}\beqref{BregProp:B(x,y)=B(x,y_i)-B(y,y_i)}. In addition, Proposition \bref{prop:BregmanProperties}\beqref{BregProp:LevelSet2BoundedSufficientConditions} and the fact that $b$ is strictly convex on $\dom(b)$ imply that the second type level-sets of $B$ are bounded.  

It remains to show that $b$ satisfies Definition \bref{def:BregmanDiv}\beqref{BregmanDef:BoundedLevelSet}. Assume to the contrary that this is not true, namely that $L_1(x,\gamma)$ is not bounded for some $x\in U$ and $\gamma\in\R$. Then there is a sequence $(y_i)_{i=1}^{\infty}$ in $L_1(x,\gamma)$ such that $\lim_{i\to\infty}\|y_i\|=\infty$. Since $X$ is finite-dimensional and all the norms on $X$ are equivalent, it follows that $\lim_{i\to\infty}\|y_i\|_{\infty}=\infty$. Hence, again from the fact that $X$ is finite-dimensional, there is an index $j\in \{1,\ldots,n\}$ such that 
$\lim_{i\to\infty}|y_{i,j}|=\infty$ and hence (since $y_i\in U$) $\lim_{i\to\infty}y_{i,j}=\infty$, where $y_{i,j}$ is the $j$-th component of $y_i$ for all $i\in \N$. Let $u:(1,\infty)\to\R$ be the function defined by $u(t):=-\log(\log(t))$ for each $t\in (1,\infty)$. Then $b(z)=\sum_{k=1}^{\infty}u(z_k)$ for each $z\in U$. In addition, $B(w,z)=\sum_{k=1}^n B_k(w_k,z_k)$ for all $w, z\in U$, where $B_k(w_k,z_k):=\log(\log(z_k)/\log(w_k))+(1/(z_k\log(z_k)))(w_k-z_k)$ for all $w_k, z_k\in (1,\infty)$, namely $B_k$ is the one-dimensional version of $B$ from \beqref{eq:IteratedLog-B}, that is, $B_k$ is (for all $k\in \{1,\ldots,n\}$) the Bregman divergence associated with $u$. 

Proposition \bref{prop:BregmanProperties}\beqref{BregProp:B=0} ensures that $B_k(w_k,z_k)\geq 0$ for all $w_k,z_k\in (1,\infty)$. Therefore $B(x,y_i)=\sum_{k=1}^n B_k(x_k,y_{i,k})\geq B_j(x_j,y_{i,j})$ for every $i\in\N$. However, since $(y_i)_{i=1}^{\infty}$ is in $L_1(x,\gamma)$, it follows that 
\begin{equation}\label{eq:B_j(x_j,y_{i,j})}
B_j(x_j,y_{i,j})\leq B(x,y_i)\leq \gamma,\quad\forall i\in\N. 
\end{equation}
Since the choice of $j$ implies that $\lim_{i\to\infty} y_{i,j}=\infty$ and since $\lim_{t\to\infty}\log(t)=\infty$, it follows that $\lim_{i\to\infty}\log(\log(y_{i,j})/\log(x_j))=\infty$ and $\lim_{i\to\infty}(x_j-y_{i,j})/(y_{i,j}\log(y_{i,j}))=0$. Thus 
\begin{equation*}
\lim_{i\to\infty}B_j(x_j,y_{i,j})=\lim_{i\to\infty}\left[\log\left(\displaystyle{\frac{\log(y_{i,j})}{\log(x_j)}}\right)+\displaystyle{\frac{x_j-y_{i,j}}{y_{i,j}\log(y_{i,j})}}\right]=\infty,
\end{equation*}
in contrast to \beqref{eq:B_j(x_j,y_{i,j})}. This contradiction proves that $L_1(x,\gamma)$ is bounded for every $x\in U$ and $\gamma\in \R$, as required.

\begin{remark}
For applications, it might be that one would be interested in working with an iterated log type Bregman function having $\tilde{U}:=(0,\infty)^n$ as its zone instead of $(1,\infty)^n$. A simple way to achieve this goal is to take $\tilde{b}(\tilde{x}):=-\sum_{k=1}^n\log(\log(1+\tilde{x}_k))$, $\tilde{x}\in\tilde{U}$, and to apply Remark \bref{rem:TypesOfConvexity}\beqref{item:UniformlyConvexTranslation} and Remark \bref{rem:BregmanTranslation}. 
\end{remark}

\section{An $\ell_2$-type example}\label{sec:Ell2type}
Let $n\in\N\cup\{0\}$ be given. Let $X:=\ell_2$ with the norm  $\|(x_i)_{i=1}^{\infty}\|:=\sum_{i=1}^{2n}|x_i|+\sqrt{\sum_{i=2n+1}^{\infty}x_i^2}$, where, of course, $x=(x_i)_{i=1}^{\infty}\in X$ and $\sum_{i=1}^{2n}|x_i|:=0$ if $n=0$.  Then $(X,\|\cdot\|)$ is isomporphic to $(X,\|\cdot\|_{\ell_2})$. Indeed, when $n=0$, then the assertion is clear, and when $n>0$, then for each $x\in X$, we have
$\sqrt{\sum_{i=1}^{\infty}x_i^2}=\|x\|_{\ell_2}\leq\|x\|\leq 2\sqrt{n}\|x\|_{\ell_2}$ as a result of the inequalities  $\sqrt{t_1+t_2}\leq \sqrt{t_1}+\sqrt{t_2}\leq \sqrt{2}\sqrt{t_1+t_2}$, 
$\sqrt{\sum_{i=1}^{m}t_i^2}\leq \sum_{i=1}^{m}t_i$, and $\sum_{i=1}^{m}t_i\leq \sqrt{m}\sqrt{\sum_{i=1}^{m}t_i^2}$ for each $m\in\N$, $t_i\in [0,\infty)$, $i\in \{1,\ldots,m\}$. Thus $X^*$ is also isomorphic to $\ell_2$. For each $x\in X$, let 
\begin{equation}\label{eq:bInfinite}
b(x):=\sum_{i=1}^{\infty}\left(e^{(x_{2i-1}+x_{2i})^2}+e^{(x_{2i-1}-x_{2i})^2}-2\right).
\end{equation}
The corresponding Bregman divergence $B$ satisfies 
\begin{multline}\label{eq:Bell2}
B(x,y)=\sum_{i=1}^{\infty}\left((e^{(x_{2i-1}+x_{2i})^2}+e^{(x_{2i-1}-x_{2i})^2})-(e^{(y_{2i-1}+y_{2i})^2}+e^{(y_{2i-1}-y_{2i})^2})\right)\\
-2\sum_{i=1}^{\infty}\left(\left((y_{2i-1}+y_{2i})e^{(y_{2i-1}+y_{2i})^2}+(y_{2i-1}-y_{2i})e^{(y_{2i-1}-y_{2i})^2}\right)(x_{2i-1}-y_{2i-1})\right)\\
-2\sum_{i=1}^{\infty}\left(\left((y_{2i-1}+y_{2i})e^{(y_{2i-1}+y_{2i})^2}-(y_{2i-1}-y_{2i})e^{(y_{2i-1}-y_{2i})^2}\right)(x_{2i}-y_{2i})\right).
\end{multline}
To see that $b$ is well defined, one can use the basic limit $\lim_{t\to 0}(e^{t^2}-1)/t^2=1$ to conclude that there exists $r>0$ small enough such that for all $t\in [-r,r]$ one has $|e^{t^2}-1|\leq 2t^2$. Given $x\in X$, one has $\sum_{i=1}^{\infty}x_i^2<\infty$, and, as a result, $x_i\in [-r,r]$ for all large enough $i\in\N$. The above discussion and the inequality $(t_1+t_2)^2\leq 2(t_1^2+t_2^2)$ which holds for all $t_1,t_2\in\R$ show that the series on the right-hand side of \beqref{eq:bInfinite} converges absolutely. As for the series in \beqref{eq:Bell2}, we first observe that  the matrix form of $b'(y)$ [namely the vector $\tilde{y}:=(\tilde{y_i})_{i=1}^{\infty}$ where  $\tilde{y}_{2i-1}:=2(y_{2i-1}+y_{2i})e^{(y_{2i-1}+y_{2i})^2}+2(y_{2i-1}-y_{2i})e^{(y_{2i-1}-y_{2i})^2}$ and $\tilde{y}_{2i}:=2(y_{2i-1}+y_{2i})e^{(y_{2i-1}+y_{2i})^2}-2(y_{2i-1}-y_{2i})e^{(y_{2i-1}-y_{2i})^2}$] belongs to $\ell_2$ for each $y\in X$. Indeed, since $y\in X$, $|y_i|<0.5$ for all $i\in\N$ large enough, and hence $e^{(y_{2i-1}+y_{2i})^2}<e$ and $e^{(y_{2i-1}-y_{2i})^2}<1$ for all $i\in\N$ large enough. Using again the inequality $(t_1+t_2)^2\leq 2(t_1^2+t_2^2)$ and also the triangle inequality, we have $\tilde{y}_{2i-1}^2\leq (2e)^2\cdot 8(y_{2i-1}^2+y_{2i}^2)$ and $\tilde{y}_{2i}^2\leq (2e)^2\cdot 8(y_{2i-1}^2+y_{2i}^2)$. Thus $\sum_{i=1}^{\infty}\tilde{y}_i^2<\infty$. Since both $\tilde{y}$ and $y-x$ belong to $\ell_2$, their inner product is well defined. Hence indeed the series on the right-hand side of \beqref{eq:Bell2} converges absolutely. 

Let $v:\R^2\to\R$ be the function defined by 
\begin{equation}\label{eq:ht1t2}
v(t_1,t_2):=e^{(t_{1}+t_{2})^2}+e^{(t_{1}-t_{2})^2}-2,\quad \forall (t_1,t_2)\in \R^2. 
\end{equation}
For each $i\in\{1,\ldots,n\}$, let $X_i:=\R^2$ with the norm $\|(s_1,s_2)\|_i:=|s_1|+|s_2|$. For each $2n+1\leq i \in\N$ let  $X_i:=\R^2$ with the norm $\|(s_1,s_2)\|_i:=\sqrt{|s_1|^2+|s_2|^2}$. Let $b_i:X_i\to\R$ be defined by $b_i:=v$. The definitions of $b$ and $X$ imply that $b(x)=\sum_{i=1}^{\infty}b_i(x_{2i-1},x_{2i})$ and  $X=\bigoplus_{i=1}^{\infty}X_i$.
The Hessian matrix $H=(v''_{ij}(t))_{i,j=1}^2$ of $v$ at $t:=(t_1,t_2)$ satisfies 
\begin{subequations}
\begin{equation}
v''_{11}(t)=v''_{22}(t)=2e^{(t_1+t_2)^2}(1+2(t_1+t_2)^2)+2e^{(t_1-t_2)^2}(1+2(t_1-t_2)^2), 
\end{equation}
\begin{equation}
v''_{12}(t)=v''_{21}(t)=2e^{(t_1+t_2)^2}(1+2(t_1+t_2)^2)-2e^{(t_1-t_2)^2}(1+2(t_1-t_2)^2). 
\end{equation}
\end{subequations}
The above-mentioned relations and simple algebraic manipulations show that for each unit vector $w=(w_1,w_2)\in \R^2$ and each $t=(t_1,t_2)\in \R^2$, the following holds:
\begin{multline*}
v''(t)(w,w)\\
=2e^{(t_1+t_2)^2}(1+2(t_1+t_2)^2)(w_{1}^2+w_{2}^2+2w_1w_2)
+2e^{(t_1-t_2)^2}(1+2(t_1-t_2)^2)(w_{1}^2+w_{2}^2-2w_1w_2)\\
\geq 4(w_{1}^2+w_{2}^2)=4.
\end{multline*}
This inequality and Proposition \bref{prop:BregmanStronglyConvex}\beqref{item:eta_S} imply that $v$ is strongly convex on $\R^2$ with 4 as a strong convexity parameter. If $n>0$, then since the norm $\|\cdot\|$ of $X$ satisfies $(1/(2\sqrt{n}))\|x\|\leq \|x\|_{\ell_2}=\sqrt{\sum_{i=1}^{\infty}\|(x_{2i-1},x_{2i})\|_{X_i}^2}$ for every $x\in X$, it follows from Proposition \bref{prop:StronglyConvexSum} that $b$ is strongly convex on $X$ with $1/n$ as a parameter of strong convexity. If $n=0$, then $b$ is strongly convex on $X$ with $4$ as a parameter of strong convexity.

For each $x\in \Int(\dom(b))=X$ we can write $b'(x)=\sum_{j=1}^{\infty}h_j(x)f_j$, where for each $j\in \N$,  one has $f_j\in X^*$ and $f_j(z)=z_j$ for all $z=(z_i)_{i=1}^{\infty}\in X$ (namely, $f_j$ is the $j$-th canonical basis functional), and, in addition, for every $j\in \N$, we have  
\begin{multline*}
\begin{array}{lll}
h_{2j-1}(x)&=&2\left((x_{2j-1}+x_{2j})e^{(x_{2j-1}+x_{2j})^2}+(x_{2j-1}-x_{2j})e^{(x_{2j-1}-x_{2j})^2}\right),\\
h_{2j}(x)&=&2\left((x_{2j-1}+x_{2j})e^{(x_{2j-1}+x_{2j})^2}-(x_{2j-1}-x_{2j})e^{(x_{2j-1}-x_{2j})^2}\right).
\end{array}
\end{multline*}
Therefore $h_j$ depends continuously on finitely many (two) variables for each $j\in\N$. Thus, if we show that $b'$ maps bounded and convex subsets of $X$ to bounded subsets of $X^*$, then we can conclude from Proposition \bref{prop:weak-to-weak*} and Remark \bref{rem:SchauderBasis} that $b'$ is weak-to-weak$^*$ sequentially continuous on $X$. In fact, we show below that $b'$ is bounded and uniformly continuous on every nonempty, bounded and convex subset of $X$. Indeed, let $S$ be such a set. Then there exists $\rho>0$ such that $\|x\|<\rho$ for each $x\in S$. In particular $\sqrt{x_{2i-1}^2+x_{2i}^2}<\rho$ for all $i\in\N$. Since $v''$ exists and is continuous on $\R^2$ (where $v$ is defined in \beqref{eq:ht1t2}), it follows from the mean value theorem that $v'$ is Lipschitz continuous on any compact subset of $\R^2$. In particular, $v'$ is Lipschitz continuous on the  disc $\{(t_1,t_2)\in\R^2: \sqrt{t_1^2+t_2^2}\leq \rho\}$, with, say, $\lambda>0$  as a Lipschitz constant. This fact, combined with the fact that the matrix form of $b'$  is $(v'_{1}(x_{2i-1},x_{2i}),v'_{2}(x_{2i-1},x_{2i}))_{i=1}^{\infty}$, where $v'_j=\partial v/\partial t_j$, $j\in \{1,2\}$, and combined with the definition of the norm on $X^*$, the Cauchy-Schwarz inequality and the fact that $\|w\|_{\ell_2}\leq \|w\|$ for each $w\in X$, all imply that for all $x,y\in S$, 
\begin{multline}\label{eq:b'Lip}
\|b'(x)-b'(y)\|=\sup_{\|w\|=1}|(b'(x)-b'(y))(w)|\\
=\sup_{\|w\|=1}\left|\sum_{i=1}^{\infty}\Bigl((v'_{1}(x_{2i-1},x_{2i})-v'_{1}(y_{2i-1},y_{2i}))w_{2i-1}+(v'_{2}(x_{2i-1},x_{2i})-v'_{2}(y_{2i-1},y_{2i}))w_{2i}\Bigr)\right|\leq\\
\sqrt{\sum_{i=1}^{\infty}\Bigl(|v'_{1}(x_{2i-1},x_{2i})-v'_1(y_{2i-1},y_{2i})|^2
+|v'_{2}(x_{2i-1},x_{2i})-v'_2(y_{2i-1},y_{2i})|^2\Bigr)}\cdot\sqrt{\sum_{i=1}^{\infty}\left(w_{2i-1}^2+w_{2i}^2\right)}\\
\leq \sqrt{\sum_{i=1}^{\infty}\lambda^2(|x_{2i-1}-y_{2i-1}|^2+|x_{2i}-y_{2i}|^2)}\cdot\|w\|\\
=\sqrt{\sum_{i=1}^{n}\lambda^2(|x_{2i-1}-y_{2i-1}|^2+|x_{2i}-y_{2i}|^2)+ \sum_{i=n+1}^{\infty}\lambda^2(|x_{2i-1}-y_{2i-1}|^2+|x_{2i}-y_{2i}|^2)}\\
\leq 
\lambda\sqrt{\sum_{i=1}^{n}(|x_{2i-1}-y_{2i-1}|^2+|x_{2i}-y_{2i}|^2)}+ \lambda\sqrt{\sum_{i=n+1}^{\infty}(|x_{2i-1}-y_{2i-1}|^2+|x_{2i}-y_{2i}|^2)}\\
\leq\lambda\sum_{i=1}^{n}(|x_{2i-1}-y_{2i-1}|+|x_{2i}-y_{2i}|)+ \lambda\sqrt{\sum_{i=n+1}^{\infty}(|x_{2i-1}-y_{2i-1}|^2+|x_{2i}-y_{2i}|^2)}=\lambda\|x-y\|,
\end{multline} 
where a sum from 1 to $n$ is considered to be zero if $n=0$. 
Therefore $b'$ is Lipschitz continuous on $S$ with $\lambda$ as a Lipschitz constant. Hence $b'$ is uniformly continuous on $S$ and, in particular, $b$ is continuous. In remains to show that $b'$ is bounded on $S$. Indeed, let $y(0)\in S$ be fixed. From \beqref{eq:b'Lip}, the triangle inequality, and the fact that $\|x-y\|\leq 2\rho$ for all $x,y\in S$, it follows that $\|b'(x)\|\leq \|b'(y(0))\|+\|b'(y(0))-b'(x)\|\leq \|b'(y(0))\|+2\rho\lambda$ for each $x\in S$. 

The above discussion shows that all the conditions of Corollary  \bref{cor:BregmanFunctionFiniteDim-b'-is-cont} are satisfied and hence $b$ is a sequentially consistent Bregman function on $X$ which satisfies the limiting difference property, and, moreover, the second type level-sets of $B$ are bounded.

\section{Additional somewhat known Bregman functions}\label{sec:SomewhatKnownBregman}
In this section we consider additional Bregman functions. They are somewhat known in the sense that they, or closely related versions of them, appear in the literature, but either the discussion in the original works is not explicitly in the context of Bregman functions and/or the setting is a bit different from the setting that we consider. Moreover, as far as we know, no proofs are given in those works or elsewhere of the assertion that the claimed Bregman functions and divergences do satisfy Definition \bref{def:BregmanDiv}. Below we present a relatively brief discussion in which we focus on proving that the discussed functions and associated divergences are indeed Bregman functions and divergences, respectively. 

\subsection{The beta entropy}\label{subsec:BetaEntropy}
Given $n\in\N$, let $X:=\R^n$ with an arbitrary norm $\|\cdot\|$ and let $U:=(0,\infty)^n$. Fix $\beta\in \R$ and define the beta entropy $b:X\to(-\infty,\infty]$ by
\begin{equation*}
b(x):=\left\{\begin{array}{ll}
\displaystyle{\frac{1}{\beta(\beta-1)}}\displaystyle{\sum_{k=1}^n} (x_k^{\beta}-\beta x_k+\beta-1), &\beta\in (0,1)\cup(1,\infty),\, x\in \cl{U},\\
\displaystyle{\sum_{k=1}^n}(x_k\log(x_k)-x_k+1), &\beta=1,\, x\in \cl{U},\\
\displaystyle{\sum_{k=1}^n}(x_k-\log(x_k)+1), &\beta=0,\, x\in U,\\
\infty, & \textnormal{otherwise}.
\end{array}
\right.
\end{equation*}
The corresponding (pre-)Bregman divergence is the beta divergence:
\begin{equation*}
B(x,y)=\left\{\begin{array}{ll}
\displaystyle{\sum_{k=1}^n} \displaystyle{\left(x_k\frac{(x_k^{\beta-1}-y_j^{\beta-1})}{\beta-1}-
\frac{x_k^{\beta}-y_k^{\beta}}{\beta}\right)}, &\beta\in (0,1)\cup(1,\infty),\, (x,y)\in \cl{U}\times U,\\
\displaystyle{\sum_{k=1}^n} \left(x_k\log\left(\displaystyle{\frac{x_k}{y_k}}\right)-x_k+y_k\right), &\beta=1,\, (x,y)\in\cl{U}\times U,\\
\displaystyle{\sum_{k=1}^n}\left(\log\left(\displaystyle{\frac{y_k}{x_k}}\right)+\displaystyle{\frac{x_k}{y_k}}-1\right), &\beta=0,\, (x,y)\in U^2,\\
\infty,& \textnormal{otherwise}.
\end{array}
\right.
\end{equation*}
The beta divergence appears in the context of statistical data analysis and computational learning theory: see, for example, \cite[Section 3]{CichockiAmari2010jour}, \cite[p. 129]{Lafferty1999inproc} and some of the references therein; see also  \cite[p. 2046]{Csiszar1991jour} (note: in both \cite[p. 2046]{Csiszar1991jour} and \cite[p. 129]{Lafferty1999inproc} the $\beta$ parameter is denoted by $\alpha$). For a continuous analogue of the beta divergence $B$ (with integrals instead of sums) see, for instance, \cite[Section 3]{CichockiAmari2010jour} and \cite[p. 754]{JonesTrutzer1989jour}. It can easily be verified that the beta entropy $b$ is, up to  translation by a linear function and multiplication by a positive scalar, nothing but the negative Havrda-Charv\'at-Tsallis entropy (when $\beta\in (0,1)\cup(1,\infty)$), the negative Boltzmann-Gibbs-Shannon entropy (when $\beta=1$), and the negative Burg entropy (when $\beta=0$). As a result, we conclude from Sections \bref{sec:GibbsShannon}--\bref{sec:Burg} and Remark \bref{rem:BregmanLinear} that $b$ is a sequentially consistent Bregman function which satisfies the limiting difference property, and that $B$ is its associated Bregman divergence (and the second type level-sets of $B$ are bounded).

\subsection{The $(\alpha,\beta)$-entropy}\label{subsec:alpha-beta}
Given $n\in\N$, let $X:=\R^n$ with an arbitrary norm $\|\cdot\|$ and let $U:=(0,\infty)^n$. Fix  $\alpha\geq 1$ and $\beta\in (0,1)$, and define the $(\alpha,\beta)$-entropy $b:X\to(-\infty,\infty]$ by
\begin{equation*}
b(x):=\left\{\begin{array}{ll}
\displaystyle{\sum_{k=1}^n} (x_k^{\alpha}-x_k^{\beta}), & x\in \cl{U},\\
\infty, & \textnormal{otherwise}.
\end{array}
\right.
\end{equation*}
The corresponding (pre-)Bregman divergence is the $(\alpha,\beta)$-divergence:
\begin{equation*}
B(x,y)=\left\{\begin{array}{ll}
\displaystyle{\sum_{k=1}^n}\Bigl((x_k^{\alpha}-x_k^{\beta})-(y_k^{\alpha}-y_k^{\beta})& \\
\quad\quad\quad-(x_k-y_k)(\alpha y_k^{\alpha-1}-\beta y_k^{\beta-1})\Bigr), &\alpha>1,\, (x,y)\in\cl{U}\times U,\\
\displaystyle{\sum_{k=1}^n}\left(\beta y_k^{\beta-1}(x_k-y_k)+y_k^{\beta}-x_k^{\beta}\right), 
&\alpha=1,\, (x,y)\in \cl{U}\times U,\\
\infty,& \textnormal{otherwise}.
\end{array}
\right.
\end{equation*}
This entropy is mentioned very briefly in \cite[Example  3, p. 201]{BurachikIusem1998jour}, \cite[Example 6, p. 172]{BurachikIusemSvaiter1997jour}, \cite[Example 28, p. 394]{CensorIusemZenios1998jour},  \cite[pp. 340, 343-344]{Cruz-NetoFerreiraIusemMonteiro2007jour}, \cite[Example 1.3, p. 595]{Iusem1995jour}, \cite[p. 619]{IusemMonteiro2000jour}, \cite[p. 11]{KaplanTichatschke2004jour} and \cite[p. 245]{daSilvaSilvaEcksteinHumes2001jour}, where $X$ is assumed to be Euclidean; no proof is given that $b$ is a Bregman function. If $\alpha>1$, then we can write $b=(\alpha-1)b_{\alpha}+(1-\beta)b_{\beta}$, where $b_{\alpha}$ is defined by \beqref{eq:Havrda-Charvat-Tsallis} with $\alpha$ instead of $q$. Therefore if $\alpha>1$, then it follows from Remark \bref{rem:BregmanLinear} and Subsection \bref{subsec:HCT-Bregman} that $b$ is a sequentially consistent Bregman function which satisfies the limiting difference property and $B$ has bounded level-sets of the second type. If $\alpha=1$, then we can write $b=\tilde{b}_1+(1-\beta)b_{\beta}$, where $\tilde{b}_1(x):=\sum_{k=1}^n(x_k-1)$, $x\in \cl{U}$. Hence Remark \bref{rem:BregmanLinear} and Subsection \bref{subsec:HCT-Bregman} imply that $b$ is again a sequentially consistent Bregman function which satisfies the limiting difference property and $B$ has bounded level-sets of the second type.

\subsection{A mixed finite-dimensional $\ell_2-\ell_p$ entropy}\label{subsec:L2Lp}
Let $p\in (1,2]$ and $X:=\R^n$ ($n\in\N$) with the $\ell_p$ norm $\|\cdot\|_p$.  Let $U:=X$ and 
\begin{equation*}
b(x):=\frac{1}{2}\|x\|_p^2=\frac{1}{2}\left(\sum_{k=1}^{n}|x_k|^p\right)^{2/p},\quad x=(x_k)_{k=1}^n\in U. 
\end{equation*}
The corresponding (pre-)Bregman divergence is 
\begin{multline*}
B(x,y)=\frac{1}{2}\left(\sum_{k=1}^{n}|x_k|^p\right)^{2/p}-\frac{1}{2}\left(\sum_{k=1}^{n}|y_k|^p\right)^{2/p}\\
-\left(\sum_{k=1}^{n}|y_k|^p\right)^{(2/p)-1}\sum_{k=1}^n\sign(y_k)|y_k|^{p-1}(x_k-y_k).
\end{multline*}
The function $b$ appears in \cite[Section 4.2, Appendix 1]{Ben-TalMargalitNemirovski2001jour} and \cite{NemirovskyYudin1983book}, but in the corresponding context in which $b$ appears it is  rather weakly related to the theory of Bregman functions and divergences. Actually, to the best of our knowledge, there has been no attempt there or elsewhere to prove that $b$ is a Bregman function. However, it was shown in \cite[Appendix 1]{Ben-TalMargalitNemirovski2001jour},\cite{NemirovskyYudin1983book} that $b$ is strongly convex with a strong convexity parameter $\mu=p-1$. Since $X$ is finite-dimensional, we conclude from Corollary \bref{cor:BregmanFunctionFiniteDim-b'-is-cont} that $b$ is a sequentially consistent Bregman function which has the limiting difference property and the second type level-sets of $B$ are bounded.  

Alternatively, we can show that $b$ is a Bregman function by showing that it is a fully Legendre function since \cite[Remark 3.10]{ReemReich2018jour} implies that any fully Legendre function defined on a finite-dimensional space is a Bregman function. In order to show that $b$ is fully Legendre, it is sufficient (and necessary) to show that $b$ is differentiable on $\R^n$, strictly convex there and super-coercive (namely $\lim_{\|x\|\to\infty}f(x)/\|x\|=\infty$), again according to \cite[Remark 3.10]{ReemReich2018jour}. Differentiability of $b$ at $x=0$ is immediate, and its differentiability at any $x\neq 0$ is a consequence of the chain rule and the assumption that $p>1$; strict convexity is a consequence of the well-known facts that the $\ell_p$ norm is  strictly convex and that a norm is strictly convex if and only if any power of it with an exponent greater than 1 is a strictly convex function \cite[Theorem 2.3]{Prus2001incol}; super-coercivity of $b$ is a consequence of the fact that all norms on $\R^n$ are equivalent: indeed, this fact implies that there exists $\eta>0$ such that $\|x\|_p\geq \eta\|x\|$ for each $x\in X$, where $\|\cdot\|$ is the Euclidean norm, and hence $b(x)/\|x\|\geq \eta b(x)/\|x\|_p=\eta\|x\|_p\to\infty$ as $\|x\|\to\infty$, as required.

\subsection{Quadratic entropies}\label{subsec:Quadratic}
Let $X\neq\{0\}$ be a real Hilbert space with  norm $\|\cdot\|$, which is induced by the inner product $\langle \cdot,\cdot\rangle$. Suppose that $A:X\to X$ is a continuous linear operator which is strongly monotone (also called ``elliptic'' or ``coercive'' or ``strongly coercive''), namely there exists $\mu>0$ such that $\langle Ay-Ax, y-x\rangle\geq \mu \|x-y\|^2$ for every $(x,y)\in X^2$. Since $A$ is linear this condition is equivalent to
\begin{equation}\label{eq:Elliptic2}
\inf\{\langle Aw,w\rangle: w\in X,\, \|w\|=1\}\geq \mu.
\end{equation}
As follows from \cite[Lemma 3.4]{ReemReich2018jourJMAA}, the coercivity condition \beqref{eq:Elliptic2} holds  when $A$ is positive semidefinite and invertible. Conversely, if \beqref{eq:Elliptic2} holds, then $A$ must be invertible as a result of the Lax-Milgram theorem \cite[Corollary 5.8, p. 140]{Brezis2011book}, \cite[Theorem 2.1 and its proof, p. 169]{LaxMilgram1954incol} (and hence  \cite[Example 3.2]{ReemReich2018jour} implies that the function $b$ defined in \beqref{eq:bQuadratic} below is fully Legendre). A particular case in which \beqref{eq:Elliptic2}  holds is when $X$ is finite-dimensional and $A$ is positive definite. Now let $b:X\to\R$ be the pre-Bregman function defined by 
\begin{equation}\label{eq:bQuadratic}
b(x):=\frac{1}{2}\langle Ax,x\rangle,\quad\forall x\in X.
\end{equation}
Then $b$ is a quadratic function and its associated pre-Bregman divergence is 
\begin{equation}
B(x,y)=\frac{1}{2}\langle Ax-Ay,x-y\rangle \quad \forall\, (x,y)\in X^2. 
\end{equation}
Of course, in the particular case when $A$ is the identity operator we get the very familiar expressions $b(x)=\frac{1}{2}\|x\|^2$ and $B(x,y)=\frac{1}{2}\|x-y\|^2$. Since we have $\langle Ax,x\rangle=\langle \frac{1}{2}(A+A^*)x,x\rangle$ for every $x\in X$ and since $A+A^*$ is symmetric (self-adjoint), where $A^*$ is the adjoint of $A$, we can assume without changing $b$  that $A=A^*$, namely we assume from now on that $A$ is symmetric. 

Quadratic entropies appear in numerous places in the literature, including in the original paper of Bregman \cite{Bregman1967jour} in the special case where $A$ is the identity operator. The case where $A$ is positive definite and $X$ is a finite-dimensional space has been considered in the literature (including in \cite{Bregman1967jour}), but less frequently. We are not aware of places in the  literature which discuss, in the context of a real Hilbert space and Definition \bref{def:BregmanDiv}, the Bregman function \beqref{eq:bQuadratic} which is constructed from the operator $A$ which satisfies \beqref{eq:Elliptic2}. 

We show below that $b$ is a sequentially consistent Bregman function which satisfies the limiting difference property. Indeed, a direct computation shows that 
\begin{equation}\label{eq:b'=0.5AA*}
\langle b'(x),w\rangle=\langle Ax,w\rangle, \quad b''(x)(w,w)=\langle Aw,w\rangle,\,\, \forall (x,w)\in X^2.
\end{equation}
From \beqref{eq:b'=0.5AA*} and \beqref{eq:Elliptic2} it follows that $b''(x)(w,w)=\langle Aw,w\rangle\geq \mu$ for all unit vector $w\in X$. This inequality and Proposition \bref{prop:BregmanStronglyConvex}\beqref{item:eta_S} imply that $b$ is strongly convex on $X$ with $\mu$ as a strong convexity parameter. From \beqref{eq:b'=0.5AA*} and the fact that $A$ is continuous and hence bounded, it follows that $b'$ is Lipschitz continuous (and hence uniformly continuous) on $X$. To see that $b'$ is weak-to-weak$^*$ sequentially continuous, one simply observes that if $(x_i)_{i=1}^{\infty}$ is a sequence in $X$ which converges weakly to $x\in X$, then this assumption, the fact that $A=A^*$ and the fact that the inner product is symmetric show that for every $w\in X$, 
\begin{equation*}
\langle b'(x_i), w\rangle=\langle Ax_i,w\rangle
=\langle x_i,Aw\rangle{\xrightarrow[i\to \infty]{}}\langle x,Aw\rangle=\langle b'(w),x\rangle, 
\end{equation*}
as required. We conclude from Corollary \bref{cor:BregmanFunctionFiniteDim-b'-is-cont} that $b$ is a  sequentially consistent Bregman function which satisfies the limiting difference property and that the second type level-sets of $B$ are bounded.

\section*{Acknowledgments}
Part of the work of the first  author was done when he was at the Institute of Mathematical and Computer Sciences (ICMC), University of S\~ao Paulo,  S\~ao Carlos, Brazil (2014--2016), and was supported by FAPESP 2013/19504-9. He also wants to express his thanks to Alfredo Iusem  for helpful discussions. The second author was partially supported by the Israel Science Foundation (Grants 389/12 and 820/17), by the Fund for the Promotion of Research at the Technion and by the Technion General Research Fund. The third author thanks CNPq  grant 306030/2014-4 and FAPESP 2013/19504-9. All the authors express their thanks to the referees for their helpful remarks and to Jennifer Cobb from SEG (the Society of Exploration Geophysicists) for 
her help regarding certain historical aspects related to \cite{Burg1967conf}.

\bibliographystyle{acm}
\bibliography{biblio}

\begin{thebibliography}{100}

\bibitem{AbeOkamoto2001book}
{\sc Abe, S., and Okamoto, Y.}, Eds.
\newblock {\em Nonextensive statistical mechanics and its applications},
  vol.~560 of {\em Lecture Notes in Physics}.
\newblock Springer-Verlag, Berlin, 2001.
\newblock Papers from the IMS Winter School on Statistical Mechanics:
  Nonextensive Generalization of Boltzmann-Gibbs Statistical Mechanics and its
  Applications held in Okazaki, February 15--18, 1999.

\bibitem{AlberButnariu1997jour}
{\sc Alber, Y., and Butnariu, D.}
\newblock Convergence of {B}regman projection methods for solving consistent
  convex feasibility problems in reflexive {B}anach spaces.
\newblock {\em J. Optim. Theory Appl. 92\/} (1997), 33--61.

\bibitem{AmbrosettiProdi1993book}
{\sc Ambrosetti, A., and Prodi, G.}
\newblock {\em A {P}rimer of {N}onlinear {A}nalysis}.
\newblock Cambridge University Press, New York, USA, 1993.

\bibitem{Araujo1988jour}
{\sc Araujo, A.}
\newblock The nonexistence of smooth demand in general {B}anach spaces.
\newblock {\em J. Math. Econom. 17\/} (1988), 309--319.

\bibitem{BanerjeeGuoWang2005jour}
{\sc Banerjee, A., Guo, X., and Wang, H.}
\newblock On the optimality of conditional expectation as a {B}regman
  predictor.
\newblock {\em IEEE Trans. on Information Theory 51\/} (2005), 2664--2669.

\bibitem{BanerjeeMeruguDhillonGhosh2005jour}
{\sc Banerjee, A., Merugu, S., Dhillon, I.~S., and Ghosh, J.}
\newblock Clustering with {B}regman divergences.
\newblock {\em J. Mach. Learn. Res. 6\/} (2005), 1705--1749.
\newblock A preliminary version in Proceedings of the fourth SIAM International
  Conference on Data Mining, pp. 234--245, Philadelphia, 2004.

\bibitem{BauschkeBolteTeboulle2017jour}
{\sc Bauschke, H.~H., Bolte, J., and Teboulle, M.}
\newblock A descent lemma beyond {L}ipschitz gradient continuity: first-order
  methods revisited and applications.
\newblock {\em Math. Oper. Res. 42\/} (2017), 330--348.

\bibitem{BauschkeBorwein1997jour}
{\sc Bauschke, H.~H., and Borwein, J.~M.}
\newblock Legendre functions and the method of random {B}regman projections.
\newblock {\em J. Convex Anal. 4\/} (1997), 27--67.

\bibitem{BauschkeBorweinCombettes2001jour}
{\sc Bauschke, H.~H., Borwein, J.~M., and Combettes, P.~L.}
\newblock Essential smoothness, essential strict convexity, and {L}egendre
  functions in {B}anach spaces.
\newblock {\em Commun. Contemp. Math. 3\/} (2001), 615--647.

\bibitem{BauschkeBorweinCombettes2003jour}
{\sc Bauschke, H.~H., Borwein, J.~M., and Combettes, P.~L.}
\newblock Bregman monotone optimization algorithms.
\newblock {\em SIAM J. Control Optim. 42\/} (2003), 596--636.

\bibitem{BauschkeCombettes2003jour}
{\sc Bauschke, H.~H., and Combettes, P.~L.}
\newblock Construction of best {B}regman approximations in reflexive {B}anach
  spaces.
\newblock {\em Proc. Amer. Math. Soc. 131\/} (2003), 3757--3766.

\bibitem{BauschkeCombettes2017book}
{\sc Bauschke, H.~H., and Combettes, P.~L.}
\newblock {\em Convex {A}nalysis and {M}onotone {O}perator {T}heory in
  {H}ilbert {S}paces}, 2~ed.
\newblock CMS Books in Mathematics. Springer International Publishing, Cham,
  Switzerland, 2017.

\bibitem{BauschkeMacklemSewellWang}
{\sc Bauschke, H.~H., Macklem, M.~S., Sewell, J.~B., and Wang, X.}
\newblock Klee sets and {C}hebyshev centers for the right {B}regman distance.
\newblock {\em J. Approx. Theory 162\/} (2010), 1225--1244.

\bibitem{Beauzamy1982book}
{\sc Beauzamy, B.}
\newblock {\em Introduction to {B}anach {S}paces and their {G}eometry}, vol.~68
  of {\em North-Holland Mathematics Studies}.
\newblock North-Holland Publishing Co., Amsterdam-New York, 1982.
\newblock Notas de Matem{\'a}tica [Mathematical Notes], 86.

\bibitem{BeckTeboulle2003jour}
{\sc Beck, A., and Teboulle, M.}
\newblock Mirror descent and nonlinear projected subgradient methods for convex
  optimization.
\newblock {\em Oper. Res. Lett. 31\/} (2003), 167--175.

\bibitem{Ben-TalMargalitNemirovski2001jour}
{\sc Ben-Tal, A., Margalit, T., and Nemirovski, A.}
\newblock The ordered subsets mirror descent optimization method with
  applications to tomography.
\newblock {\em SIAM J. Optim. 12\/} (2001), 79--108.

\bibitem{BoissonnatNielsenNock2010jour}
{\sc Boissonnat, J.-D., Nielsen, F., and Nock, R.}
\newblock Bregman {V}oronoi diagrams.
\newblock {\em Discrete Comput. Geom. 44\/} (2010), 281--307.
\newblock A preliminary version in SODA 2007, pp. 746-755.

\bibitem{BPT1998jour}
{\sc Borland, L., Plastino, A.~R., and Tsallis, C.}
\newblock Information gain within nonextensive thermostatistics.
\newblock {\em J. Math. Phys. 39\/} (1998), 6490--6501.
\newblock Erratum: J. Math. Phys. 40 (1999), p. 2196.

\bibitem{BGHV2009jour}
{\sc Borwein, J., Guirao, A.~J., H{\'a}jek, P., and Vanderwerff, J.}
\newblock Uniformly convex functions on {B}anach spaces.
\newblock {\em Proc. Amer. Math. Soc. 137\/} (2009), 1081--1091.

\bibitem{BorweinReichSabach2011jour}
{\sc Borwein, J.~M., Reich, S., and Sabach, S.}
\newblock A characterization of {B}regman firmly nonexpansive operators using a
  new monotonicity concept.
\newblock {\em J. Nonlinear Convex Anal. 12\/} (2011), 161--184.

\bibitem{BorweinVanderwerff2012jour}
{\sc Borwein, J.~M., and Vanderwerff, J.}
\newblock Constructions of uniformly convex functions.
\newblock {\em Canad. Math. Bull. 55\/} (2012), 697--707.

\bibitem{Bregman1967jour}
{\sc Bregman, L.~M.}
\newblock The relaxation method of finding the common point of convex sets and
  its application to the solution of problems in convex programming.
\newblock {\em Comput. Math. Math. Phys. 7\/} (1967), 200--217.

\bibitem{BregmanCensorReich1999jour}
{\sc Bregman, L.~M., Censor, Y., and Reich, S.}
\newblock Dykstra's algorithm as the nonlinear extension of {B}regman's
  optimization method.
\newblock {\em J. Convex Anal. 6\/} (1999), 319--333.

\bibitem{Brezis2011book}
{\sc Brezis, H.}
\newblock {\em Functional {A}nalysis, {S}obolev {S}paces and {P}artial
  {D}ifferential {E}quations}.
\newblock Universitext. Springer, New York, 2011.

\bibitem{BurachikIusem1998jour}
{\sc Burachik, R.~S., and Iusem, A.~N.}
\newblock A generalized proximal point algorithm for the variational inequality
  problem in a {H}ilbert space.
\newblock {\em SIAM J. Optim. 8\/} (1998), 197--216.

\bibitem{BurachikIusemSvaiter1997jour}
{\sc Burachik, R.~S., Iusem, A.~N., and Svaiter, B.~F.}
\newblock Enlargement of monotone operators with applications to variational
  inequalities.
\newblock {\em Set-Valued Anal. 5\/} (1997), 159--180.

\bibitem{BurachikScheimberg2000jour}
{\sc Burachik, R.~S., and Scheimberg, S.}
\newblock A proximal point method for the variational inequality problem in
  {B}anach spaces.
\newblock {\em SIAM Journal on Control and Optimization 39\/} (2000),
  1633--1649.

\bibitem{Burg1967conf}
{\sc Burg, J.~P.}
\newblock Maximum entropy spectral analysis.
\newblock {\em Paper presented at the 37th Meeting of the Society of
  Exploration Geophysicists (SEG)\/} (1967).
\newblock Oklahoma City, Oklahoma, USA.

\bibitem{Burg1975phd}
{\sc Burg, J.~P.}
\newblock {\em Maximum Entropy Spectral Analysis}.
\newblock PhD thesis, Stanford University, CA, USA, 1975.
\newblock
  {\bf\url{http://sepwww.stanford.edu/data/media/public/oldreports/sep06/}}.

\bibitem{ButnariuByrneCensor2003jour}
{\sc Butnariu, D., Byrne, C., and Censor, Y.}
\newblock Redundant axioms in the definition of {B}regman functions.
\newblock {\em J. Convex Anal. 10\/} (2003), 245--254.

\bibitem{ButnariuIusem2000book}
{\sc Butnariu, D., and Iusem, A.~N.}
\newblock {\em Totally {C}onvex {F}unctions for {F}ixed {P}oint {C}omputation
  and {I}nfinite {D}imensional {O}ptimization}.
\newblock Applied Optimization. Kluwer Academic Publishers, Dordrecht, The
  Netherlands, 2000.

\bibitem{ButnariuIusemZalinescu2003jour}
{\sc Butnariu, D., Iusem, A.~N., and Z\u{a}linescu, C.}
\newblock On uniform convexity, total convexity and convergence of the proximal
  point and outer {B}regman projection algorithms in {B}anach spaces.
\newblock {\em J. Convex. Anal. 10\/} (2003), 35--61.

\bibitem{ButnariuReichZaslavski2001jour}
{\sc Butnariu, D., Reich, S., and Zaslavski, A.~J.}
\newblock Asymptotic behavior of relatively nonexpansive operators in {B}anach
  spaces.
\newblock {\em J. Appl. Anal. 7\/} (2001), 151--174.

\bibitem{ButnariuReichZaslavski2001incol}
{\sc Butnariu, D., Reich, S., and Zaslavski, A.~J.}
\newblock Generic power convergence of nonlinear operators in {B}anach spaces.
\newblock In {\em Fixed point theory and applications ({C}hinju/{M}asan,
  2001)}, Y.~J. Cho, J.~K. Kim, and S.~M. Kang, Eds. Nova Sci. Publ.,
  Hauppauge, NY, 2003, pp.~35--49.

\bibitem{ButnariuReichZaslavski2006col}
{\sc Butnariu, D., Reich, S., and Zaslavski, A.~J.}
\newblock Convergence to fixed points of inexact orbits of {B}regman-monotone
  and of nonexpansive operators in {B}anach spaces.
\newblock In {\em Fixed Point Theory and its Applications}, H.~F. Natansky~et
  al., Ed. Yokohama Publ., Yokohama, 2006, pp.~11--32.

\bibitem{ButnariuResmerita2006jour}
{\sc Butnariu, D., and Resmerita, E.}
\newblock Bregman distances, totally convex functions, and a method for solving
  operator equations in {B}anach spaces.
\newblock {\em Abstr. Appl. Anal. 2006\/} (2006), 1--39.
\newblock Art ID 84919.

\bibitem{Cayton}
{\sc Cayton, L.}
\newblock Fast nearest neighbor retrieval for {B}regman divergences.
\newblock In {\em Proceedings of the 25th International Conference on Machine
  Learning (ICML)\/} (Helsinki, 2008), IEEE, pp.~112--119.

\bibitem{CensorDePierroElfvingHermanIusem1990incol}
{\sc Censor, Y., De~Pierro, A.~R., Elfving, T., Herman, G.~T., and Iusem,
  A.~N.}
\newblock On iterative methods for linearly constrained entropy maximization.
\newblock In {\em Numerical analysis and mathematical modelling}, A.~Wakulicz,
  Ed., vol.~24 of {\em Banach Center Publ.} PWN, Warsaw, 1990, pp.~145--163.

\bibitem{CensorDePierroIusem1991jour}
{\sc Censor, Y., De~Pierro, A.~R., and Iusem, A.~N.}
\newblock Optimization of {B}urg's entropy over linear constraints.
\newblock {\em Appl. Numer. Math. 7\/} (1991), 151--165.
\newblock Preliminary version: Tech. Rept. MIPG 113, Medical Image Processing
  Group, Department of Radiology, University of Pennsylvania, Philadelphia, PA
  (1986).

\bibitem{CensorIusemZenios1998jour}
{\sc Censor, Y., Iusem, A.~N., and Zenios, S.~A.}
\newblock An interior point method with {B}regman functions for the variational
  inequality problem with paramonotone operators.
\newblock {\em Math. Programming (Ser. A) 81\/} (1998), 373--400.

\bibitem{CensorLent1981jour}
{\sc Censor, Y., and Lent, A.}
\newblock An iterative row-action method for interval convex programming.
\newblock {\em J. Optim. Theory Appl. 34\/} (1981), 321--353.

\bibitem{CensorLent1987jour}
{\sc Censor, Y., and Lent, A.}
\newblock Optimization of ``{${\rm log}\,x$}'' entropy over linear equality
  constraints.
\newblock {\em SIAM J. Control Optim. 25\/} (1987), 921--933.

\bibitem{CensorReich1996jour}
{\sc Censor, Y., and Reich, S.}
\newblock Iterations of paracontractions and firmly nonexpansive operators with
  applications to feasibility and optimization.
\newblock {\em Optimization 37\/} (1996), 323--339.

\bibitem{CensorZenios1997Book}
{\sc Censor, Y., and Zenios, A.~S.}
\newblock {\em Parallel {O}ptimization: {T}heory, {A}lgorithms, and
  {A}pplications}.
\newblock Numerical Mathematics and Scientific Computation. Oxford University
  Press, New York, 1997.
\newblock With a foreword by George B. Dantzig.

\bibitem{ChenTeboulle1993jour}
{\sc Chen, G., and Teboulle, M.}
\newblock Convergence analysis of a proximal-like minimization algorithm using
  {B}regman functions.
\newblock {\em SIAM J. Optim. 3\/} (1993), 538--543.

\bibitem{CichockiAmari2010jour}
{\sc Cichocki, A., and Amari, S.-i.}
\newblock Families of alpha- beta- and gamma-divergences: flexible and robust
  measures of similarities.
\newblock {\em Entropy 12\/} (2010), 1532--1568.

\bibitem{CollinsSchapireSinger2002jour}
{\sc Collins, M., Schapire, R.~E., and Singer, Y.}
\newblock Logistic regression, {A}da{B}oost and {B}regman distances.
\newblock {\em Mach. Learn. 48\/} (2002), 253--285.

\bibitem{Cruz-NetoFerreiraIusemMonteiro2007jour}
{\sc Cruz~Neto, J.~X., Ferreira, O.~P., Iusem, A.~N., and Monteiro, R. D.~C.}
\newblock Dual convergence of the proximal point method with {B}regman
  distances for linear programming.
\newblock {\em Optim. Methods Softw. 22\/} (2007), 339--360.

\bibitem{Csiszar1991jour}
{\sc Csisz\'ar, I.}
\newblock Why least squares and maximum entropy? {A}n axiomatic approach to
  inference for linear inverse problems.
\newblock {\em Ann. Statist. 19\/} (1991), 2032--2066.

\bibitem{CsiszarTusnady1984jour}
{\sc Csisz\'ar, I., and Tusn\'ady, G.}
\newblock Information geometry and alternating minimization procedures.
\newblock {\em Statist. Decisions suppl. 1\/} (1984), 205--237.
\newblock Recent results in estimation theory and related topics.

\bibitem{DePierro1991incol}
{\sc De~Pierro, A.~R.}
\newblock Multiplicative iterative methods in computed tomography.
\newblock In {\em Mathematical methods in tomography ({O}berwolfach, 1990)},
  G.~T. Herman, A.~K. Louis, and F.~Natterer, Eds., vol.~1497 of {\em Lecture
  Notes in Math.} Springer, Berlin, 1991, pp.~167--186.

\bibitem{DePierroIusem1986jour}
{\sc De~Pierro, A.~R., and Iusem, A.~N.}
\newblock A relaxed version of {B}regman's method for convex programming.
\newblock {\em J. Optim. Theory Appl. 51\/} (1986), 421--440.

\bibitem{DunfordSchwartz1958book}
{\sc Dunford, N., and Schwartz, J.~T.}
\newblock {\em Linear {O}perators. {I}. {G}eneral {T}heory}.
\newblock With the assistance of W. G. Bade and R. G. Bartle. Pure and Applied
  Mathematics, Vol. 7. Interscience Publishers, Inc., New York; London, 1958.

\bibitem{Eckstein1998jour}
{\sc Eckstein, J.}
\newblock Approximate iterations in {B}regman-function-based proximal
  algorithms.
\newblock {\em Math. Programming (Ser. A) 83\/} (1998), 113--123.

\bibitem{EdwardFitelson1973jour}
{\sc Edward, J., and Fitelson, M.}
\newblock Notes on maximum-entropy processing (corresp.).
\newblock {\em IEEE Trans. Inf. Theory 19\/} (1973), 232--234.

\bibitem{Elfving1989jour}
{\sc Elfving, T.}
\newblock An algorithm for maximum entropy image reconstruction from noisy
  data.
\newblock {\em Math. Comput. Modelling 12\/} (1989), 729--745.

\bibitem{Frieden1975inbook}
{\sc Frieden, B.~R.}
\newblock Image enhancement and restoration.
\newblock In {\em Picture Processing and Digital Filtering}, T.~S. Huang, Ed.
  Springer Berlin Heidelberg, Berlin, Heidelberg, 1975, pp.~177--248.

\bibitem{GaoLiu2014jour}
{\sc Gao, Y., and Liu, W.}
\newblock Be{T}rust: A dynamic trust model based on {B}ayesian inference and
  {T}sallis entropy for medical sensor networks.
\newblock {\em J. Sensors [vol. 2014]\/} (2014), 10 pages, Article ID 649392.

\bibitem{Gell-MannTsallis2004book}
{\sc Gell-Mann, M., and Tsallis, C.}, Eds.
\newblock {\em Nonextensive entropy---interdisciplinary applications}.
\newblock Santa Fe Institute Studies in the Sciences of Complexity. Oxford
  University Press, New York, 2004.

\bibitem{Gibbs1874-1878jour}
{\sc Gibbs, J.~W.}
\newblock On the equilibrium of heterogeneous substances (first part).
\newblock {\em Trans. Conn. Acad. Arts Sci. 3\/} (1874--1878), 108--248.
\newblock available at: {\color{blue}{\bf
  \url{https://www.archive.org/download/transactions01conn/transactions01conn.pdf}}}.

\bibitem{GohbergGoldberg1981book}
{\sc Gohberg, I., and Goldberg, S.}
\newblock {\em Basic {O}perator {T}heory}.
\newblock Birkh\"auser, Boston, MA, USA, 1981.

\bibitem{GuptaHuang}
{\sc Gupta, M.~D., and Huang, T.~S.}
\newblock Bregman distance to {L}1 regularized logistic regression.
\newblock In {\em International {C}onference on {P}attern {R}ecognition
  ({ICPR})\/} (Tampa, FL, USA, 2008), IEEE, pp.~1--4.

\bibitem{HavrdaCharvat1967jour}
{\sc Havrda, J., and Charv{\'a}t, F.}
\newblock Quantification method of classification processes. {C}oncept of
  structural {$a$}-entropy.
\newblock {\em Kybernetika 3\/} (1967), 30--35.

\bibitem{ItakuraSaito1968inproc}
{\sc Itakura, F., and Saito, S.}
\newblock Analysis synthesis telephony based on the maximum likelihood method.
\newblock In {\em Reports of the 6th International Congress on Acoustics (Y.
  Kohasi, ed.)\/} (Tokyo, 1968), pp.~C–17–--C–20.

\bibitem{Iusem1995jour}
{\sc Iusem, A.~N.}
\newblock Some properties of generalized proximal point methods for quadratic
  and linear programming.
\newblock {\em J. Optim. Theory Appl. 85\/} (1995), 593--612.

\bibitem{IusemMonteiro2000jour}
{\sc Iusem, A.~N., and Monteiro, R. D.~C.}
\newblock On dual convergence of the generalized proximal point method with
  {B}regman distances.
\newblock {\em Math. Oper. Res. 25\/} (2000), 606--624.

\bibitem{JonesByrne1990jour}
{\sc Jones, L.~K., and Byrne, C.~L.}
\newblock General entropy criteria for inverse problems, with applications to
  data compression, pattern classification, and cluster analysis.
\newblock {\em IEEE Trans. on Information Theory 36\/} (1990), 23--30.
\newblock Correction: IEEE Trans. Inform. Theory 37 (1991), 224--225.

\bibitem{JonesTrutzer1989jour}
{\sc Jones, L.~K., and Trutzer, V.}
\newblock Computationally feasible high-resolution minimum-distance procedures
  which extend the maximum-entropy method.
\newblock {\em Inverse Problems 5\/} (1989), 749--766.
\newblock Corrigendum: Inverse Problems 6 (1990), 873.

\bibitem{KaplanTichatschke2004jour}
{\sc Kaplan, A., and Tichatschke, R.}
\newblock On inexact generalized proximal methods with a weakened error
  tolerance criterion.
\newblock {\em Optimization 53\/} (2004), 3--17.

\bibitem{Kiwiel1997jour}
{\sc Kiwiel, K.}
\newblock Proximal minimization methods with generalized {B}regman functions.
\newblock {\em SIAM J. Control Optim. 35\/} (1997), 1142--1168.

\bibitem{Kreyszig1978book}
{\sc Kreyszig, E.}
\newblock {\em Introductory {F}unctional {A}nalysis with {A}pplications}.
\newblock John Wiley \& Sons, New York-London-Sydney, 1978.

\bibitem{KullbackLeibler1951jour}
{\sc Kullback, S., and Leibler, R.~A.}
\newblock On information and sufficiency.
\newblock {\em Ann. Math. Statistics 22\/} (1951), 79--86.

\bibitem{Lafferty1999inproc}
{\sc Lafferty, J.}
\newblock Additive models, boosting, and inference for generalized divergences.
\newblock In {\em Proceedings of the Twelfth Annual Conference on Computational
  Learning Theory\/} (Santa Cruz, California, USA, 1999), COLT '99, ACM,
  pp.~125--133.

\bibitem{LaxMilgram1954incol}
{\sc Lax, P.~D., and Milgram, A.~N.}
\newblock Parabolic equations.
\newblock In {\em Contributions to the {T}heory of {P}artial {D}ifferential
  {E}quations}, L.~Bers, S.~Bochner, and F.~John, Eds., Annals of Mathematics
  Studies, no. 33. Princeton University Press, Princeton, N. J., 1954,
  pp.~167--190.

\bibitem{LiZhou2016jour}
{\sc Li, T., and Zhou, M.}
\newblock {ECG} classification using wavelet packet entropy and random forests.
\newblock {\em Entropy 18\/} (2016), Article Number 285 (16pp.).

\bibitem{ManicPriyaRajinikanth2016jour}
{\sc Manic, K.~S., Priya, R.~K., and Rajinikanth, V.}
\newblock Image multithresholding based on {K}apur/{T}sallis entropy and
  firefly algorithm.
\newblock {\em Indian J. Science and Technology 9\/} (2016), 6 pages.

\bibitem{MurataTakenouchiKanamoriEguchi2004jour}
{\sc Murata, N., Takenouchi, T., Kanamori, T., and Eguchi, S.}
\newblock Information geometry of {U}-{B}oost and {B}regman divergence.
\newblock {\em Neural Comput. 16\/} (2004), 1437--1481.

\bibitem{NemirovskyYudin1983book}
{\sc Nemirovsky, A.~S., and Yudin, D.~B.}
\newblock {\em Problem {C}omplexity and {M}ethod {E}fficiency in
  {O}ptimization}.
\newblock A Wiley-Interscience Publication. John Wiley \& Sons, Inc., New York,
  1983.
\newblock Translated from the Russian edition (1979), with a preface by E. R.
  Dawson, Wiley-Interscience Series in Discrete Mathematics.

\bibitem{Nesterov2004book}
{\sc Nesterov, Y.}
\newblock {\em Introductory {L}ectures on {C}onvex {O}ptimization: {A} {B}asic
  {C}ourse}, vol.~87 of {\em Applied Optimization}.
\newblock Kluwer Academic Publishers, Boston, USA, 2004.

\bibitem{Nguyen2017jour}
{\sc Nguyen, Q.~V.}
\newblock Forward-backward splitting with {B}regman distances.
\newblock {\em Vietnam J. Math. 45\/} (2017), 519--539.

\bibitem{Phelps1993book_prep}
{\sc Phelps, R.~R.}
\newblock {\em Convex {F}unctions, {M}onotone {O}perators and
  {D}ifferentiability}, 2~ed., vol.~1364 of {\em Lecture Notes in Mathematics}.
\newblock Springer-Verlag, Berlin, 1993.
\newblock Closely related material can be found in ``Lectures on maximal
  monotone operators, arXiv:math$/$9302209 [math.FA] ([v1], 4 Feb 1993)''.

\bibitem{PickupCywinskiPappasFaragoFouquet2009jour}
{\sc Pickup, R.~M., Cywinski, R., Pappas, C., Farago, B., and Fouquet, P.}
\newblock Generalized spin-glass relaxation.
\newblock {\em Phys. Rev. Lett. 102\/} (Mar 2009), 097202.

\bibitem{Prus2001incol}
{\sc Prus, S.}
\newblock Geometrical background of metric fixed point theory.
\newblock In {\em Handbook of {M}etric {F}ixed Point Theory}, W.~A. Kirk and
  B.~Sims, Eds. Kluwer Acad. Publ., Dordrecht, 2001, pp.~93--132.

\bibitem{Reem2012incol}
{\sc Reem, D.}
\newblock The {B}regman distance without the {B}regman function {II}.
\newblock In {\em Optimization {T}heory and {R}elated {T}opics, {C}ontemp.
  {M}ath. ({A}mer. {M}ath. {S}oc., {P}rovidence, {RI})}, S.~Reich and A.~J.
  Zaslavski, Eds., vol.~568. 2012, pp.~213--223.

\bibitem{ReemReich2018jourJMAA}
{\sc Reem, D., and Reich, S.}
\newblock Fixed points of polarity type operators.
\newblock {\em J. Math. Anal. Appl. 467\/} (2018), 1208--1232.
\newblock arXiv:1708.09741 [math.FA] (2017) (current version: [v4]; 8 Apr
  2019).

\bibitem{ReemReich2018jour}
{\sc Reem, D., and Reich, S.}
\newblock Solutions to inexact resolvent inclusion problems with applications
  to nonlinear analysis and optimization.
\newblock {\em Rend. Circ. Mat. Palermo (2) 67\/} (2018), 337--371.
\newblock arXiv:1610.01871 [math.OC] (2016) (current version: [v5]; 22 Aug
  2017).

\bibitem{ReemReichDePierro(TEPROG)2019accept}
{\sc Reem, D., Reich, S., and De~Pierro, A.}
\newblock A telescoping {B}regmanian proximal gradient method without the
  global {L}ipschitz continuity assumption.
\newblock {\em J. Optim. Theory. Appl., accepted for publication (DOI:
  10.1007/s10957-019-01509-8)\/}.
\newblock arXiv:1804.10273 [math.OC] (2018) (current version: [v4], 19 Mar
  2019).

\bibitem{Reich1996incol}
{\sc Reich, S.}
\newblock A weak convergence theorem for the alternating method with {B}regman
  distances.
\newblock In {\em Theory and {A}pplications of {N}onlinear {O}perators of
  {A}ccretive and {M}onotone {T}ype}, A.~G. Kartsatos, Ed., vol.~178 of {\em
  Lecture Notes in Pure and Appl. Math.} Dekker, New York, 1996, pp.~313--318.

\bibitem{ReichSabach2010b-jour}
{\sc Reich, S., and Sabach, S.}
\newblock Two strong convergence theorems for a proximal method in reflexive
  {B}anach spaces.
\newblock {\em Numer. Funct. Anal. Optim. 31\/} (2010), 22--44.

\bibitem{ReichSabach2010jour}
{\sc Reich, S., and Sabach, S.}
\newblock Two strong convergence theorems for {B}regman strongly nonexpansive
  operators in reflexive {B}anach spaces.
\newblock {\em Nonlinear Anal. 73\/} (2010), 122--135.

\bibitem{Rockafellar1970book}
{\sc Rockafellar, R.~T.}
\newblock {\em Convex {A}nalysis}.
\newblock Princeton Mathematical Series, No. 28. Princeton University Press,
  Princeton, NJ, USA, 1970.

\bibitem{Shannon1948jour}
{\sc Shannon, C.~E.}
\newblock A mathematical theory of communication.
\newblock {\em Bell System Tech. J. 27\/} (1948), 379--423.

\bibitem{daSilvaSilvaEcksteinHumes2001jour}
{\sc Silva, P. J.~S., Eckstein, J., and Humes, Jr., C.}
\newblock Rescaling and stepsize selection in proximal methods using separable
  generalized distances.
\newblock {\em SIAM J. Optim. 12\/} (2001), 238--261.

\bibitem{SolodovSvaiter2000jour}
{\sc Solodov, M.~V., and Svaiter, B.~F.}
\newblock An inexact hybrid generalized proximal point algorithm and some new
  results in the theory of {B}regman functions.
\newblock {\em Math. Oper. Res. 51\/} (2000), 214--230.

\bibitem{TaskarLacoste-JulienJordan2006jour}
{\sc Taskar, B., Lacoste-Julien, S., and Jordan, M.~I.}
\newblock Structured prediction, dual extragradient and {B}regman projections.
\newblock {\em J. Mach. Learn. Res. 7\/} (2006), 1627--1653.

\bibitem{Teboulle1992jour}
{\sc Teboulle, M.}
\newblock Entropic proximal mappings with applications to nonlinear
  programming.
\newblock {\em Math. Oper. Res. 17\/} (1992), 670--690.

\bibitem{Teboulle2007jour}
{\sc Teboulle, M.}
\newblock A unified continuous optimization framework for center-based
  clustering methods.
\newblock {\em J. Mach. Learn. Res. 8\/} (2007), 65--102.

\bibitem{Tsallis1988jour}
{\sc Tsallis, C.}
\newblock Possible generalization of {B}oltzmann-{G}ibbs statistics.
\newblock {\em J. Statist. Phys. 52\/} (1988), 479--487.

\bibitem{Tsallis2009book}
{\sc Tsallis, C.}
\newblock {\em Introduction to {N}onextensive {S}tatistical {M}echanics:
  {A}pproaching a {C}omplex {W}orld}.
\newblock Springer, New York, 2009.

\bibitem{TsallisContinuouslyUpdatedList}
{\sc Tsallis, C.}
\newblock Nonextensive statistical mechanics and thermodynamics: bibliography,
  2018.
\newblock This is a continuously updated online list: {\bf
  \color{blue}{\url{http://tsallis.cat.cbpf.br/TEMUCO.pdf}}} . Retrieved
  version: 16 October 2018 (with 6913 bibliographic items). A dedicated website
  (with more items) is: {\bf
  \color{blue}{\url{http://tsallis.cat.cbpf.br/biblio.htm}}}.

\bibitem{VanTiel1984book}
{\sc van Tiel, J.}
\newblock {\em {C}onvex {A}nalysis: {A}n {I}ntroductory {T}ext}.
\newblock John Wiley and Sons, Universities Press, Belfast, Northern Ireland,
  1984.

\bibitem{VladimirovNesterovChekanov1978jour}
{\sc Vladimirov, A.~A., Nesterov, Y.~E., and Chekanov, Y.~N.}
\newblock Uniformly convex functionals.
\newblock {\em Vestnik Moskov. Univ. Ser. XV Vychisl. Mat. Kibernet. 3\/}
  (1978), 12--23.
\newblock (in Russian).

\bibitem{Wiener1948book}
{\sc Wiener, N.}
\newblock {\em Cybernetics, or {C}ontrol and {C}ommunication in the {A}nimal
  and the {M}achine}.
\newblock Actualit\'es Sci. Ind., no. 1053. Hermann et Cie., Paris; The
  Technology Press, Cambridge, Mass.; John Wiley \& Sons, Inc., New York, 1948.

\bibitem{YinOsherGoldfarbDarbon2008jour}
{\sc Yin, W., Osher, S., Goldfarb, D., and Darbon, J.}
\newblock Bregman iterative algorithms for $\ell_1$-minimization with
  applications to compressed sensing.
\newblock {\em SIAM J. Imaging Sci. 1\/} (2008), 143--168.

\bibitem{Zalinescu1983jour}
{\sc Z{\u a}linescu, C.}
\newblock On uniformly convex functions.
\newblock {\em J. Math. Anal. Appl. 95\/} (1983), 344--374.

\bibitem{Zalinescu2002book}
{\sc Z{\u a}linescu, C.}
\newblock {\em Convex {A}nalysis in {G}eneral {V}ector {S}paces}.
\newblock World Scientific Publishing, River Edge, NJ, USA, 2002.

\bibitem{Zaslavski2010b-jour}
{\sc Zaslavski, A.~J.}
\newblock Convergence of a proximal-like algorithm in the presence of
  computational errors.
\newblock {\em Taiwanese J. Math. 14\/} (2010), 2307--2328.

\end{thebibliography}

\end{document}